\documentclass[10pt]{article}
 \usepackage [latin1]{inputenc}
\usepackage{amssymb,amsmath,amsfonts}
\usepackage{dsfont}
\usepackage{amsthm}
\usepackage{graphicx,color,subfig}
\usepackage{xspace}
\usepackage{mathrsfs}
\usepackage{amscd}
\usepackage{epsfig,color,graphics}
\usepackage[colorlinks=true,hyperindex=true,bookmarks=true]{hyperref}
\hypersetup{
    colorlinks,
    citecolor=blue,
    filecolor=purple,
    linkcolor=red,  
    urlcolor=blue
}

\let \Re \relax
\DeclareMathOperator{\Re}{Re}
\let \Im \relax
\DeclareMathOperator{\Im}{Im}
\newcommand{\Con}{\ensuremath{\mathscr C}}

\renewcommand{\S}{\ensuremath{\mathscr S}}

\DeclareMathOperator{\supp}{supp}

\newcommand{\est}[1]{\langle #1 \rangle}

\newcommand{\mb}[1]{\ensuremath{\mathbb{#1}}}
\newcommand{\N}{{\mb{N}}}

\newcommand{\R}{{\mb{R}}}
\newcommand{\C}{{\mb{C}}}

\newcommand{\eps}{\varepsilon}

\renewcommand{\d}{\ensuremath{\partial}}

\DeclareMathOperator{\op}{op}
\DeclareMathOperator{\Op}{Op}

\DeclareMathOperator{\Ai}{Ai}

\newcommand{\D}{\ensuremath{\mathscr D}}

\newcommand{\nhd}{neighborhood\xspace}

\usepackage[active]{srcltx}

\def\DS{\displaystyle}

\def\vv{{\rm v}_h}
\def\qq{{\rm q}_h}
\def\ct{\tilde \chi}   
\def\cc{\chi}

\usepackage[english]{babel}

\newtheorem{lemma}{Lemma}[section]
\newtheorem{theorem}{Theorem}
\newtheorem{proposition}[lemma]{Proposition}

\newtheorem{definition}[lemma]{Definition}
\newtheorem{remark}{Remark}



\begin{document}
\title{Exponential stabilization of waves for the Zaremba boundary condition.}

\author{Pierre Cornilleau\thanks{Teacher at Lyc\'ee Pothier, 2 bis, rue Marcel Proust, 45000 Orleans , France.
%\newline 
e-mail: pierre.cornilleau@ens-lyon.org}  \&
 Luc Robbiano\thanks{Laboratoire de Math\'ematiques de Versailles, Universit\'e de Versailles St Quentin,
CNRS, 45, Avenue des Etats-Unis, 78035 Versailles, France. e-mail : luc.robbiano@uvsq.fr}}

\maketitle

\begin{abstract}
In this article we prove, under some geometrical condition on geodesic flow,  exponential stabilization of 
wave equation with Zaremba boundary condition. We 
prove an estimate on the resolvent of semigroup associated with wave equation 
on the imaginary axis and we deduce 
the stabilization result. To prove this estimate we apply semiclassical measure technics. 
The main difficulties are to prove that support of measure is in characteristic set in a \nhd of 
the jump in the boundary condition and to prove results of propagation in a \nhd of a boundary point 
where Neumann boundary condition is imposed. In fact if a lot of results applied here are proved in
previous articles, these two points are new.
\end{abstract}


\tableofcontents

%%%%%%%%%%%%%%%%%%%%%%%%%%%%%%%%%%%%%%%%%%%
%
%  Introduction
%
%%%%%%%%%%%%%%%%%%%%%%%%%%%%%%%%%%%%%%%%%%%

\section{Introduction and results}

\subsection{Framework}
In this article we are interested by stabilization of wave equation with 
Zaremba boundary condition. 
To be precise we have to introduce some notation. 
Let $\Omega$ be a bounded open set in $\R^d$, with 
$\Con^\infty$ boundary. Let $\d\Omega_D$ and $\d\Omega_N$ two open sets in $\d\Omega$ 
such that $\d\Omega_D\cap\d\Omega_N=\emptyset$ and 
$\overline{ \d\Omega_D}\cap \overline{ \d\Omega_N}=\Gamma$, 
where $\Gamma$ is a smooth 
manifold of dimension $d-2$. The manifold $\Gamma$ is not 
necessary a connected set.
Let $P$ be a second order differential operator. We have
 $P=\sum_{1\le j ,k\le d}D_{x_j}p_{jk}(x)D_{x_k}
+\sum_{1\le j \le d}p_j(x)D_{x_j}+p_0(x)$, where 
$p_{jk}(x)$, $p_j(x)$ are in $\Con^\infty(V)$ where $V$ is a \nhd of $\overline\Omega$.  
The matrix $(p_{jk}(x))_{jk}$ is assumed  positive definite for every $x\in V$.
We assume that the operator defined by $Pu$ for  
$u\in {\cal D}=\{ u\in H^1(\Omega),\ Pu\in L^2(\Omega), 
\  u_{|\d\Omega_D}=0 ,  
(\d_\nu u)_{|\d\Omega_N}=0 \}$,  is self-adjoint and non negative.  
Here $\d_\nu$ is the exterior 
normal derivative.
Let $a\in\Con^\infty(V)$ be such that $a(x)\ge0$ for every $x\in\Omega$. 
We associate with $P$ the 
following wave equation
\begin{align}
\begin{cases}
\d_t^2 u+Pu +a(x)\d_tu=0 \text{ in }\Omega\times (0,\infty)\\
(u,\d_t u)_{|t=0}=(u_0,u_1)\in H^1(\Omega)\oplus L^2(\Omega) \\
 u=0   \text{ on } \d\Omega_D\times (0,\infty)\\
 \d_\nu u=0  \text{ on } \d\Omega_N\times (0,\infty)
\end{cases}
\end{align}
We associated the energy that is 
$E(t,u_0,u_1)=  (Pu|u)_{L^2(\Omega)}+ \int_\Omega |\d_t u|^2 dx$, where
$(v|w)_{L^2(\Omega)}=\int _\Omega v(x)\overline{w(x)}dx$.
Under assumptions on   flows associated with $P$ and $a$ (see sections~\ref{sec: Geometry},
and \ref{sec: Assumptions and results}) we obtain that the energy 
satisfies $E(t)\le Ce^{-ct}$ for some constants $c>0$ and $C>0$. 
We obtain this result by an estimate 
on the resolvent associated with this problem and by the 
Gearhart-Huang-Pr\"uss theorem~\cite{Gearhart-1978,Pruss-1984,Huang85}.    
To prove the resolvent estimate we use semiclassical  measures. 
The method is well-known, since the seminal work by Bardos-Lebeau-Rauch~\cite{BLR1}, 
and was applied in 
different contexts and different variants, defect measures, 
Wigner measures, $H$-measures, see  
Aloui-Khenissi-Robbiano~\cite{AKR-2016},  
Anantharaman-L\'eautaud-Maci\`a~\cite{ALM-2016,ALM-2016-2},  
Burq~\cite{Burq-1997}, Burq-Lebeau~\cite{Burq-Lebeau2001}, 
Dehman-Le~Rousseau-L\'eautaud~\cite{DBLR-2014}, 
G\'erard-Leichtnam~\cite{GL-1993},
G\'erard~\cite{Gerard-1990,Gerard-1991},  
Le~Rousseau-Lebeau-Terpolilli-Tr\'elat~\cite{LRLTT-2016}, Le\-beau~\cite{Leb}, 
Miller~\cite{Miller-2000}, Robbiano-Zuily~\cite{Rob-Zuily-2009}, Tartar~\cite{Tartar-1990} 
for instance.

We can find an introduction to this subject in Zworski~\cite{Zworski-2012}.

Problems for the Zaremba boundary condition was studied by several authors. 
In particular, for elliptic problem, Shamir~\cite{Sh} and
Savar\'e~\cite{Savare-1997}  
proved that the regularity of solution 
is not as for the Dirichlet boundary condition, there is a lack of regularity, 
$s=3/2$ is critical in the 
Sobolev spaces $H^s$ if the datum is in $L^2$. 
The problem is related with boundary problem in non smooth domain, with corner 
for instance. There is a large literature on this subject. 
For damping wave equation with Zaremba 
Boundary condition, the problem was studied by Bey and al.~\cite{BLM-1999}, 
Cornilleau and al.~\cite{CLO}
where they prove exponential decay with multiplier method, and in~\cite{CR2014} 
where we only 
prove  logarithmic decay but without geometric condition 
on the support of the damping. Same kind of result was proven by 
Fu~\cite{Fu-2015} for mixed boundary condition of 
Robin type.

In the following we described the geometry in section~\ref{sec: Geometry}. 
This allows to give the 
precise assumption and the result in 
section~\ref{sec: Assumptions and results}. At the end of this section we give a description 
of  proofs.

%%%%%%%%%%%%%%%%%%%%%%%%%%%%%%%%%%%%%%%%%%%
%
%  Geometry
%
%%%%%%%%%%%%%%%%%%%%%%%%%%%%%%%%%%%%%%%%%%%

\subsection{Geometry}\label{sec: Geometry}
Here we give the geometrical notion we use in this article. This framework comes from Melrose and 
Sj\"ostrand~\cite{MS-1978,MS-1982} and 
the reader may also find in 
H\"ormander~\cite[Chapter 24]{HormanderV3-2007} more informations and proofs.  The characterisation of symplectic sub-manifold 
is probably classical and more details can be found in Grigis~\cite{Grigis-1976}.

\medskip
\paragraph{Assumption on the symbol.}
We define the symbol of $P$ by 
\begin{align}\label{def: symbol h2P-1}
p(x,\xi)=\sum_{1\le j ,k\le d}p_{jk}(x)\xi_j\xi_k -1,
\end{align}
Where $p_{jk}$ are $\Con^\infty(\overline{ \Omega})$.
Locally in a \nhd of the boundary we can define $\Omega $ by $\varphi>0$ with $d\varphi\ne 0$. 
We can also choose coordinates  (i.e. \emph{normal geodesic coordinates}) such that
$\varphi(x)=x_d$ and $p(x,\xi)= \xi_d^2+R(x,\xi')-1$ where $x=(x',x_d) $ and $\xi=(\xi',\xi_d)$. 

\medskip
\paragraph{Symplectic sub-manifold $\Sigma$.}
We can define a symplectic manifold $\Sigma$,
 contained into $T^*\R^d\cap\{ \varphi=0 \}$. We set $\Sigma=\{  (x,\xi),\ \varphi(x)=0 \text{ and } 
 \{\varphi, p\} (x,\xi)=0\}$. The set $\Sigma$ is a 
 symplectic manifold as $\{\varphi  ,\{\varphi,p\}\}\ne0$. In coordinates $(x',x_d)$, we have 
 $\Sigma=\{x_d=0,\xi_d=0\}$, this manifold is isomorphic to 
 $T^*\d\Omega$ and described by coordinates $(x',\xi')$. 
 
 The Hamiltonian vector field $H_p$ is not a 
 vector field on $\Sigma$, but for all $X$ a vector field on $T^*\R^d$, we can find
 unique fonctions $\alpha$ and $\beta$ such $X+\alpha H_\varphi+\beta H_{\{\varphi, p\} }$ is a vector 
 field on  $\Sigma$. For $H_p$ we denote the associated vector $H'_p$
 and an elementary computation leads to 
 \begin{align*}
 H'_p=H_p+\frac{ \{ p ,\{ p ,\varphi  \} }{ \{ \varphi ,\{ \varphi  , p \}}  H_\varphi .
 \end{align*}
 In coordinates  $(x',x_d)$,   $H'_p$ only depends  on $R$ and we have
  \begin{align*}
   H'_p=H'_R=\sum_{j=1}^{d-1} (\d_{\xi_j}R(x',0,\xi')\d_{x_j} - \d_{x_j}R(x',0,\xi' )  \d_{\xi_j}).
  \end{align*}
In particular the integral curves associated with $H'_p$ starting from a point into $\Sigma$ stay into $\Sigma$. In coordinates $(x',x_d)$, we denote the integral curve 
starting from $(x',\xi')$, either $\gamma_g(s; x',\xi')$, either  $\gamma_g( x',\xi')$, if $s$ is implicit or  $\gamma_g(s)$, if $(x',\xi')$ is implicit.

\medskip
\paragraph{Symplectic sub-manifold $\Sigma'$.}
The manifold $\Gamma$ can be locally defined by $\{\varphi=0, \ \psi=0 \}$, where $d\varphi\wedge d\psi\ne0$. We define the manifold $\Sigma'$ by
\begin{align*}
\Sigma'=\{ (x,\xi),\ \varphi(x)=\psi(x)=\{ \varphi ,  p\} (x,\xi)=\{ \psi , p \} (x,\xi)=0\}.
\end{align*}
This manifold is symplectic. It suffices to prove that   the following matrix
\begin{align*}
\begin{pmatrix}
 0  &  0   &   \{ \varphi ,\{ \varphi  , p \} \} & \{ \varphi ,\{ \psi  , p \} \} \\
 0   &   0  &   \{ \psi ,\{ \varphi  , p \} \} & \{ \psi ,\{ \psi  , p \}  \}\\
  \{ \{ \varphi  , p \} , \varphi \}&  \{    \{ \varphi  , p \} ,\psi  \} & 0 &  \{  \{ \varphi ,p\} ,\{ \psi  , p \} \}  \\
    \{ \{ \psi , p \} , \varphi \}&  \{    \{ \psi , p \} ,\psi  \} &  \{  \{ \psi ,p\} ,\{ \varphi  , p \} \}&  0
\end{pmatrix},
\end{align*}
is invertible.  Clearly this is true if the matrix 
\begin{align}
\begin{pmatrix} \label{matrix: on Gamma}
   \{ \varphi ,\{ \varphi  , p \} \} & \{ \varphi ,\{ \psi  , p \} \} \\
  \{ \psi ,\{ \varphi  , p \} \} & \{ \psi ,\{ \psi  , p \} \}
\end{pmatrix},
\end{align}
is invertible.  The quadratic form in $(t,z)$, 
$ \{  t\varphi+z\psi  ,\{   t\varphi+z\psi , p \} \} $ is positive definite
as $p''_{\xi \xi}$, is a positive definite  matrix and $d\varphi$ and $d\psi$ are independent as 
$d\varphi\wedge d\psi\ne0$. 
The matrix associated with  $ \{  t\varphi+z\psi  ,\{   t\varphi+z\psi , p \} \} $ 
is the one given by~\eqref{matrix: on Gamma}, then this matrix is invertible.

We can choose local coordinates $(x_1,x'',x_d)$ such that $\phi(x)= x_d$,  $\psi(x)=x_1$ and 
\begin{align*}
p(x,\xi)=\xi_d^2+\xi_1^2+R_1(x',\xi'')+x_dR_2(x,\xi')-1=\xi_d^2+R(x,\xi')-1.
\end{align*}
In these local coordinates, the manifold $\Sigma'=\{x_1=x_d=\xi_1=\xi_d=0  \}$ 
which is isomorphic to $T^*\Gamma$, and described 
by the coordinates $(x'',\xi'')$. 

For all vector field
$X$ defined on  $T^*\R^d$, we can find unique functions 
$\alpha,\ \beta, \ \gamma, \ \zeta$ such that 
$X+ \alpha H_\varphi+\beta H_\psi+\gamma H_{\{ \varphi , p\}}
  + \zeta H_{\{\psi  , p\}} $ is a vector field on $\Sigma'$. For $X=H_p$, 
  we denote the associated vector 
  field $H''_p$ and we have $H''_p=H_p+\alpha H_\varphi+\beta H_\psi$,
  as $H_p\varphi=H_p\psi=0$ on $\Sigma'$ and $H_\psi\varphi=0$. 
  We can compute $\alpha$ and $\beta$ 
  but the precise values are not useful for 
  general functions $\varphi$ and $\psi$.
  In coordinates $(x_1,x'',x_d)$, we have $H''_p=H_p+\alpha \d_{\xi_1}+\beta \d_{\xi_d}$. The 
  equations $H''_p\xi_1=H''_p\xi_d=0$, on 
  $\Sigma'$, give $\alpha=-\d_{x_1}R(0,x'',0,0,\xi'')$ and $\beta=-\d_{x_d} R_1(0,x'',\xi'')$. In particular 
  $H''_p$ only depends on $R_1$, and we have
  $H''_p=H''_R=H''_{R_1}=\sum_{j=2}^{d-1}(\d_{\xi_j}R_1(0,x'',\xi'')\d_{x_j}-(\d_{x_j}R_1(0,x'',\xi'')\d_{\xi_j})$.
  The integral curves starting from $\Sigma'$ stay on $\Sigma'$. We denote the curves starting from 
  $(x'',\xi'')$, $\gamma_{\rm sing}(s; x'',\xi'')$,  
  $\gamma_{\rm sing}(x'',\xi'')$, if $s$ is implicit and  $\gamma_{\rm sing}(s)$, if $(x'',\xi'')$ is implicit.

\medskip
\paragraph{Description and topology of $T^*_b\Omega$.}
Let $T^*_b\Omega=T^*\d\Omega\cup T^*\Omega$,
 this set is equipped with the following topology.

  First if $\rho\in T^*\Omega$, 
  a set $V$ is a \nhd of $\rho$ if $V$ contains an open set $W$ of 
  $ T^*\Omega$  such that $\rho\in W$.  
  
  Second if $\rho=(x_0',\xi_0')\in T^*\d\Omega$, a set $V$ is a 
  \nhd of $\rho$ if $V$ contains a set
  \begin{align*}
 &  \{ (x',\xi')\in T^*\d\Omega,\ |x'_0-x'|+|\xi'_0-\xi'|\le\eps \}\\
  &  \cup \{ (x,\xi)\in T^*\Omega,\ 
   |x'_0-x'|+|\xi'_0-\xi'|\le\eps \text{ and } (x_d,\xi_d)\in
  U\cap \{ x_d>0 \} \},
  \end{align*}
 where $\eps>0$ and $U$ is a \nhd of $\{ (x_d,\xi_d)\in\R^2, \ x_d=0\}$ in $\R^2$.
 
 In local coordinates where $\Omega $ is define by $x_d>0$, we define 
 $j: T^*\overline \Omega \to T^*_b\Omega$
 by $j(x,\xi)=(x,\xi)$ if $x\in\Omega$, and $j(x,\xi)=(x',\xi')$ if $x_d=0$. 
 The map $j$ is continuous for the topology given above.
  We can define more intrinsically  $j$ with the previous notation where $\Omega$ is 
  given by $\varphi(x)>0$. For 
  $(x,\xi)\in T^*\overline \Omega $, $j(x,\xi)=(x,\xi)$ if $x\in\Omega$ and 
  $j(x,\xi)=(x, \xi dx-(\{p,\varphi\}/H^2_\varphi p)d\varphi)$, if 
  $\varphi(x)=0$. We verify, in this last case that $j(x,\xi)\in \{  \varphi=\{p,\varphi\}=0 \}$, 
  as $\{p,\varphi\}(x,d\varphi)
  =\{\varphi,\{ \varphi,p\}\}$.

  As usually we define the map $\pi: T^*_b\Omega\to \overline \Omega$, in local coordinates, 
  as $\pi(x,\xi)=x$, if $(x,\xi)\in T^*\Omega$ and $\pi(x',\xi')=x'$, if $(x',\xi')\in T^*\d\Omega$.
  
\medskip
\paragraph{Bicharacteristic and generalized flow.}

\medskip

For $(x,\xi)\in T^*\R^d$, we denote by $\gamma(s;x,\xi)$ the integral curve of 
$H_p$ starting from $(x,\xi)$. We use the same short notations $\gamma(s)$ 
and $ \gamma(x,\xi)$ as above.

Now we define the generalized bicharacteristic denoted by $\Gamma(s,\rho)$ for 
$\rho\in T^*_b\Omega$. To describe this curve in a 
\nhd of the boundary we use the coordinates 
$(x',x_d,\xi',\xi_d)$ and we identify $\Sigma'$ and $T^*\d\Omega$ and locally 
$\Omega=\{x\in \R^d , \ x_d>0 \}$.  
Moreover, we assume 
$\rho\in \text{char} (P) = \{(x,\xi)\in T^*\Omega ,\ p(x,\xi)=0  \}\cup 
\{(x',\xi')\in T^*\d\Omega,\ R(x',0,\xi')-1\le0  \}$.

  Now we define the curve $\Gamma(s,\rho)$ locally for each $(s_0,\rho)$ and 
  we use the group property of the flow, namely 
  $\Gamma(s+t,\rho)=\Gamma(s,\Gamma(t,\rho))$ 
  to extend this function for every $s\in\R$.
  
  \medskip
For  $\Gamma(s_0,\rho)\in  
  T^*\Omega$, $\Gamma(s,\rho)=
  \gamma(s-s_0;\Gamma(s_0,\rho))$ if $\gamma(s-s_0;\Gamma(s_0,\rho))\in T^*\Omega$. In particular, 
  this defined $\Gamma(s,\rho)$ at least 
  for $s$ in a \nhd of $s_0$ as $\gamma(s-s_0;\Gamma(s_0,\rho))$ stay in $T^*\Omega $ for small $|s-s_0|$. 
  Observe that $p(\gamma(s-s_0; \Gamma(s_0,\rho)))=0$.
  
  \medskip
  For $\rho=(x_0',\xi'_0)\in T^*\d\Omega$, we have to distinguish different cases, first if $R(x'_0,0,\xi'_0)<1$ 
  and second if $R(x'_0,0,\xi'_0)=1$, 
  $\Gamma(s,\rho)$ depends on the properties of $\gamma(s; x'_0, 0, \xi'_0,0)$. In what follows we only 
  define the flow in a \nhd of $s=0$. We can extend the flow by the group property.
  
  If $R(x'_0,0,\xi'_0)<1$, let $\xi^\pm=\pm\sqrt{1-R(x'_0,0,\xi'_0)}$. Let 
  $\gamma(s;x_0,\xi_0)=(x(s;x_0,\xi_0), \xi(s;x_0,\xi_0))$, as $\dot x=2\xi_d$, we have 
  $x_d(s;x'_0,0,\xi'_0,\xi^+)>0$ for $s>0$ sufficiently small, and  $x_d(s;x'_0,0,\xi'_0,\xi^-)>0$ for 
  $s<0$ sufficiently small. Then we set 
  $\Gamma(0,\rho)=\rho$, $ \Gamma(s,\rho)= \gamma(s;x_0',0,\xi_0',\xi^+)$ for $s>0$ sufficiently small and 
  $\Gamma(s,\rho)= \gamma(s;x_0',0,\xi_0',\xi^-)$ for $s<0$ sufficiently small. Observe that for 
  $s\ne0$ sufficiently small,
  $\Gamma(s,\rho)\in T^*\Omega$ and $p(\Gamma(s,\rho))=0$. 
  
  Such points are called  hyperbolic points and we set 
  \(
  {\cal H}=\{ (x'_0,\xi'_0)\in T^*\d\Omega,\ R(x'_0, 0,\xi'_0)<1 \}.
  \)
%%%%%%%%%%%%%%%%%%%%%%%%%%%%%%%
%
% Definition
%
%%%%%%%%%%%%%%%%%%%%%%%%%%%%%%%
    \begin{definition}[Finite contact with the boundary]
  	\label{def: contact fini}
  Let $(x'_0,\xi'_0)$ be such that $R(x'_0,0,\xi'_0)=1$. We say that  the  bicharacteristic $ \gamma(s;x'_0,0,\xi'_0,0)=\gamma(s)$ 
  \emph{does not have  an infinite contact with the boundary} if 
there exists $ k\in \N,\ k\ge 2,\  \alpha\ne0$ such that  $ x_d(s;x'_0,0,\xi'_0,0)= x_d(s)= 
  \alpha s^k+{\cal  O}(s^{k+1})$ in a \nhd of $s=0$. We denote by ${\cal G}^k$ the set of such points.
  
For   \(k=2\) we distinguish two cases.
\begin{itemize}
\item The  diffractive points, and we denote 
\[
{\cal G}_d=  \{ (x',\xi')\in  T^*\d\Omega,  \  R(x',0,\xi')=1,\  \d_{x_d}R(x',0,\xi')<0 \}.
\]
\item The gliding points, and we denote  
\[
{\cal G}_g=  \{ (x',\xi')\in  T^*\d\Omega,  \  R(x',0,\xi')=1,\  \d_{x_d}R(x',0,\xi')>0 \}.
\]
\end{itemize}
 \end{definition}
      \begin{remark}
    By Taylor's theorem and 
 as $x_d(0)=0$ and 
  $\dot x_d(0)=2\xi_d(0)=0$, we always have $x_d(s)={\cal O}(s^2)$.
  \end{remark} 
We have four cases to treat.
\begin{description}
\item[-]   $k$ even, $\alpha>0$. In this case $x_d(s)>0$ for $s\ne0$ sufficiently small. 
We define $\Gamma(0,\rho)=\rho$ and $\Gamma(s,\rho)=\gamma(s;x_0',0,\xi'_0,0)\in T^*\Omega$ 
for $s\ne0$ sufficiently small.
\item[-] $k$ even, $\alpha<0$. In this case $x_d(s)<0$ for $s\ne0$ sufficiently small. 
We define $\Gamma(s)=\gamma_g(s,\rho)\in T^*\d\Omega$ for $s$ sufficiently small.
\item[-] $k$ odd, $\alpha>0$. In this case $x_d(s)>0$ for $s>0$ sufficiently small and   
$x_d(s)<0$ for $s<0$ sufficiently small. We define 
$\Gamma(s,\rho)=\gamma_g(s,\rho)\in T^*\d\Omega$ for $s\le0$ sufficiently small, and $\Gamma(s,\rho)=\gamma(s;x_0',0,\xi'_0,0)   \in T^*\Omega$ for $s>0$, sufficiently small.
\item[-] $k$ odd, $\alpha<0$. In this case $x_d(s)<0$ for $s>0$ sufficiently small and   $x_d(s)>0$ for $s<0$ sufficiently small. We define 
$\Gamma(s,\rho)=\gamma(s;x_0',0,\xi'_0,0)  \in T^*\Omega$ for $s<0$ sufficiently small, and $\Gamma(s,\rho)=\gamma_g(s,\rho)\in T^*\d\Omega$ for $s\ge0$, sufficiently small.
\end{description}
This local description of $\Gamma(s)$ allows to prolongate $\Gamma(s) $ for every $s\in\R$. 
The function $\Gamma(s, \rho) $ defined on $\R\times\text{char}(P)$ is 
continuous for the topology of $\R\times T^*_b\Omega$, where the topology 
of $T^*_b\Omega$ is defined above.

%%%%%%%%%%%%%%%%%%%%%%%%%%%%%%%%%%%%%%%%%%%
%
%  Hypotheses et resultats
%
%%%%%%%%%%%%%%%%%%%%%%%%%%%%%%%%%%%%%%%%%%%

\subsection{Statement of Theorems} \label{sec: Assumptions and results}
  
In the following we give the assumptions on the flows and these assumptions depend on the starting points. 
The assumptions also depend on the damping $a$ and we assume $a(x)\ge 0$ for every  $x\in\overline\Omega$. We denote by \( \omega=\{x\in\overline{\Omega},\ a(x)>0\}\).
\begin{definition}
	\label{def: mGCC}
We say that $P$, $a(x)$ and $\Omega$ satisfy the \emph{modified Geometric Control Condition (mGCC)} if
the bicharacteristic only has finite contact with the boundary (Definition~\ref{def: contact fini}) and
 the following 
assumptions are verified. Let 
\(
\rho_0\in \text{char} (P) .
\) 
\begin{itemize}
\item If \(\pi(\rho_0)\notin\Gamma\)
we assume there exist \(s_0\in\R\) be such that 
\(
\pi\Gamma(s_0,\rho_0)\in \omega 
\)
and for every 
\(
s\in [0,s_0],
\)
if 
\(
\pi\Gamma(s,\rho_0)\in\Gamma
\)
then 
\(
\Gamma(s,\rho_0)\in { \cal H}.
\)
\item  If \(\pi(\rho_0)\in\Gamma\) we assume there exist \(s_0\in \R\) be such that
\(
\pi \gamma_{\text{sing}}  (s_0,\rho_0)\in \omega.
\)
\end{itemize}
\end{definition}
\begin{remark}
This definition of mGCC is different from the usual GCC. 
We are not able to prove propagation of support of measure for generalized bicharacteristic hitting 
$\Gamma$ except for hyperbolic points. 
It is possible that singularities can be create at \(\Gamma\) but we do not know what can  happen.
 For points on $\Gamma$ we are only able to prove propagation on 
$\Gamma$ for 
integral curve of $H''_R$ if we already know that measure is supported in the fiber   above 
a point of \(\Gamma\). 
\end{remark}

We recall the assumptions on $P$,  a second order differential operator. We have 
\begin{align}
	\label{def: forme de P}
P=\sum_{1\le j ,k\le d}D_{x_j}p_{jk}(x)D_{x_k}+\sum_{1\le j \le d}p_j(x)D_{x_j}+p_0(x),
\end{align}
where 
$p_{jk}(x)$, $p_j(x)$ are real valued and  in $\Con^\infty(V)$ where $V$ is a \nhd of $\overline\Omega$.  
We assume that $P$ is formally self-adjoint. The domain of $P$ is given by  
${\cal D}(P)=\{ u\in H^1(\Omega),\ Pu\in L^2(\Omega), \  u_{|\d\Omega_D}=0 ,  
(\d_\nu u)_{|\d\Omega_N}=0 \}$, where $\d_\nu$ is the exterior normal derivative. With this domain $P$ 
is self-adjoint, and 
\begin{align}
 \label{eq: associated quadratic form to P}
(Pu|v)_{L^2(\Omega)}=\sum_{1\le j ,k\le d}( p_{jk}(x)D_{x_k}u| D_{x_j} v) _{L^2(\Omega)}+
(\sum_{1\le j \le d}p_j(x)D_{x_j}u+p_0(x) u| v)_{L^2(\Omega)},
\end{align}
where $u$ and $v$ are in $ {\cal D(P)}$. Moreover we assume $P$ positive definite, there is $\delta>0$ such that
\begin{equation}
	\label{hyp: positivity P}
(Pu|u)_{L^2(\Omega)}\ge \delta \| u\|^2_{H^1(\Omega)}, \text{ for every } u\in {\cal D}(P).
\end{equation}

To give 
a precise formulation of the wave equation we introduce $H= H^1(\Omega)\oplus 
L^2(\Omega)$, we denote by $U=(u_0,u_1)$ an element of $H$ and the operator $A$ is given by 
\begin{align}
	\label{def: forme de A}
A=
\begin{pmatrix}
0 & 1\\ -P & -a(x)
\end{pmatrix},
\end{align}
associated with the domain
\begin{align*}
{\cal D}(A)=\{ U=(u_0,u_1)\in H, P u_0\in L^2(\Omega), u_1\in H^1(\Omega), u_0=0  \text{ on } \d\Omega_D ,
\d_\nu u_0=0  \text{ on } \d\Omega_N\}.
\end{align*}
Let  $U$ be  the solution of  $\d_t U=A U$ satisfying $U(0)=(u_0,u_1)\in H$, 
we have $U(t)=(u(t),\d_t u(t))=e^{tA}(u_0,u_1)$, where 
$e^{tA}$ is the semigroup associated with $A$.
Then  $u$ satisfies the wave equation
$$
\begin{cases}
\d_t ^2u(x)+P u(x)+a(x) \d_t u(x)=0 \text{ in } \Omega,\\
u=0  \text{ on } \d\Omega_D ,\\
\d_\nu u=0  \text{ on } \d\Omega_N, \\
(u(0),\d_tu(0))=(u_0,u_1).
\end{cases}
$$
To $U(t)=(u_0(t),u_1(t))$, we associate the energy $E(t,u_0,u_1)=(Pu_0(t)|u_0(t))_{L^2(\Omega)}
+\int_\Omega |u_1(t)|^2dx$. We have $\d_t E(t,u_0,u_1)=-2(au_1|u_1)_{L^2(\Omega)}\le 0$. This implies 
$E(t,u_0,u_1)\le E(0,u_0,u_1)$. 
\begin{remark}
We have assumed that $P$ is positive definite for simplicity but if $P$ is non negative
 we can introduce $\tilde H=H/\ker P$, working on $\tilde H$ instead of $H$ 
we can obtain same results. For instance see~\cite{LR2} where this reduction is used.
\end{remark}
The main result of the paper is the following
\begin{theorem}
	\label{th: resolvent estimate}
We assume that $P$, $a(x)$ and $\Omega$ satisfy the modified Geometric Control Condition given in Definition~\ref{def: mGCC}.
We assume that $P$ has the form given in~\eqref{def: forme de P} and $P$ is self-adjoint positive definite. Let $A$ be defined 
by~\eqref{def: forme de A}, we have
\begin{description}
\item[1)] $\exists M>0 $, $\| e^{tA}\|_{{\cal L}(H)}\le M$,
\item[2)] $ A-i\mu I$, is invertible for all $\mu\in\R$,
\item[3)]   $\exists M>0 $,   $\|  (A-i\mu I)^{-1} \|_{{\cal L}(H)}\le M$.
\end{description}
\end{theorem}
\begin{remark}
There are several results when the third item is replaced by other estimates as $e^{C|\mu|}$, $|\mu|^\alpha$ in these cases 
the energy decay with a speed depending of the estimate on the resolvent. For this kind of results we refer to 
Batty-Duyckaerts~\cite{BD}, Borichev-Tomilov~\cite{BT-2010}, Burq~\cite{Burq-1997}, Lebeau~\cite{Lebeau:96}.
\end{remark}
The first item is a consequence of the energy decay. The second is given by unique continuation theorem and
also by the result given in~\cite[Proposition 1.1]{CR2014}. The goal of this article is to prove the third item.
From the Gearhart-Huang-Pr\"uss test for the exponential stability 
(see\cite{Gearhart-1978}, \cite{Pruss-1984}, 
\cite[Theorem 3]{Huang85}, \cite{EN-2000}),  the three items of the previous theorem imply that the semigroup generated by 
$A$ is exponentially stable and this 
implies the following theorem. 
%%%%%%%%%%%%%%%%%%%%%%
%
%   theorem
%
%%%%%%%%%%%%%%%%%%%%%%

\begin{theorem}\label{th: energy decay}
We assume that $P$, $a(x)$ and $\Omega$ satisfy the modified Geometric Control Condition given in Definition~\ref{def: mGCC}.
We assume that $P$ has the form given in~\eqref{def: forme de P} and $P$ is self-adjoint positive definite. Let $A$ defined 
by~\eqref{def: forme de A}, 
there exist $C,c>0$ such that 
$$
\|e^{tA}\|_{{\cal L}(H)}\le C e^{-ct}.
$$
\end{theorem}

The outline of the proof is the following. In Section~\ref{sec: Semiclassical formulation} we recall some tools on 
semiclassical pseudo-differential calculus (Section~\ref{Notations and pseudo-differential calculus}),
we reduce the third item of Theorem~\ref{th: resolvent estimate} 
to a semiclassical estimate (Section~\ref{Evolution equation and resolvent estimate} and 
Proposition~\ref{prop: equation on v spectrally localised}), and we prove a basic estimate on the trace at the boundary 
(Section~\ref{A priori estimate on traces} and Proposition~\ref{prop: a priori estimates on traces}).

In Section~\ref{Sec: Semiclassical  measure and  the characteristic set} we construct a semiclassical measure and we prove that this measure is supported on the characteristic set.
To do that at the boundary -for interior point the result is classical- we have to distinguish 
three kinds of points, hyperbolic points (see Section~\ref{Sec: Hyperbolic points} and 
Proposition~\ref{prop: estimation traces hyperbolic}) and
glancing points (see Section~\ref{Sec: Glancing points} and Proposition~\ref{prop: estimation traces glancing}).
Only for elliptic points (see Section~\ref{subsubsection: elliptic points}) we need 
to consider the boundary conditions. In a \nhd of boundary where we impose Dirichlet or Neumann boundary condition,
we prove Proposition~\ref{proposition: estimation elliptique Dirichlet Neumann} and we deduce 
Proposition~\ref{prop: convergence traces D and N}.
The proof is delicate in a \nhd of $\Gamma$ (see Proposition~\ref{lemma: trace goes to 0}).
These estimates on trace allow us to prove Proposition~\ref{prop: measure supported on the characteristic set} 
in Section~\ref{sec: Support of semiclassical measure in a nhd of boundary}. 
In Section~\ref{sec: he semiclassical  measure is not identically null} we prove that  semiclassical measure 
is not identically zero and in Section~\ref{sec: The semiclassical  measure is null on the support of a} we prove 
that semiclassical measure  is null on support of the damping.

We shall reach a contradiction if we also prove that the measure is identically null. This is done
in the  next sections.

In Section~\ref{Sec: Propagation of measure} we prove some properties of semiclassical measure. 
 In Section~\ref{sec: Action of Hamiltonian}, we obtain the action of Hamiltonian vector field on the semiclassical
measure up to the boundary. The interior result is stated in Proposition~\ref{prop: interior formula for Hp mu}.
Propositions~\ref{lemma: first propagation formula} and \ref{lemma: second propagation formula}
are analogous results at boundary.
In Section~\ref{sec: properties microlocal defect measure}, we deduce from that a decomposition of the 
 semiclassical measure
in two measure, the measure restricted in interior and a  boundary measure 
(Lemma~\ref{lem: measure on boundary}). The action of 
 Hamiltonian vector field allows us to deduce some properties of these measures. 
 Lemma~\ref{lem: Hp mu xd positive is a measure} describes the action of Hamiltonian on the interior measure, 
Lemma~\ref{lem: measure hyperbolic points} gives precisions in 
 \nhd of hyperbolic points, Lemma~\ref{lem: diffractive points non in measure support} and  
 Lemma~\ref{lem: propagation on boundary if mu0 supported on xi-d equal 0}  give properties of boundary 
 measure in a \nhd of 
Dirichlet and Neumann boundary, and Lemma~\ref{lem: propagation boundary Zaremba} is the analogous 
in \nhd of jump between Dirichlet and Neumann boundary conditions.
 
 In Section~\ref{sec: propagation proofs} we prove the propagation of support of semiclassical
measure. We have to distinguish the different cases, if a bicharacteristic hits boundary transversally
or tangentially,   \( \d\Omega_D\cup\d\Omega_N\) or
\(\Gamma\). This allows to prove the main theorem.

In Appendix~\ref{Appendix : lemma regularity H s} we prove some regularity measure needed to prove some estimate 
in a \nhd or $\Gamma$. 

In Appendix~\ref{sec: trace Neumann boundary condition} we prove some estimates 
on boundary trace in the case of Neumann boundary condition. This is useful to prove properties 
of semiclassical measure in a \nhd of a diffractive point in a \nhd of \(\d\Omega_N\). 
Appendix~\ref{sec: proof of lemmas} is devoted to prove some technical results stated in the 
previous section.

This work is based on previous results, mention  particularly, the course given by
Patrick G\'erard at IHP in 2015, the articles of Burq and Lebeau~\cite{Burq-Lebeau2001} 
and Luc Miller~\cite{Miller-2000}.
We thank Claude Zuily for the first step in this kind of problem (see \cite{Rob-Zuily-2009}), 
Belhassen Dehman, Matthieu L\'eautaud and J\'er\^ome Le~Rousseau for the working 
group where we have  together studied this subject, Nicolas Burq to draw our attention on the 
Tataru paper \cite{Tataru-1998}. That  allowed us to achieve the propagation of measure 
at boundary in the case of Neumann boundary condition.

%%%%%%%%%%%%%%%%%%%%%%%%%%%%%%%%%%%%%%%%%%%
%
%  Preuves
%
%%%%%%%%%%%%%%%%%%%%%%%%%%%%%%%%%%%%%%%%%%%

\section{Semiclassical formulation}\label{sec: Semiclassical formulation}

\subsection{Notations and pseudo-differential calculus} \label{Notations and pseudo-differential calculus}

Here we summarize some result on pseudo-differential calculus. More details, results and extension are given in the 
H\"ormander book~\cite[Chapter 18]{HormanderV3-2007}, Martinez~\cite{Martinez:02} and Le Rousseau-Lebeau~\cite{LRL-2012}. 
Essentially we follow here this last article.

To a smooth function $a(x,\xi)$, and $h\in(0,1)$ ($a$ may depend on $h$ but the constants, in the estimates given below, 
does not depend on $h$), we associate an operator by the following formula
$$
\Op_{sc} (a) u= (2\pi)^{-d}\int_{\R^d} e^{ix\xi}a(x,h\xi) \hat u(\xi)d\xi, \text{ where } 
\hat u(\xi)=\int_{\R^d} e^{-ix\xi}u(x)dx.
$$
This formula make sense under some assumption on $a$ and $u$. In this paper we mainly use symbols 
in  $S^k$. We say that $a\in S^k$ if for every $\alpha, \beta\in\N^d$ there exists $C=C_{\alpha,\beta}$ such
that 
$$
|\d_x^\alpha \d_\xi^\beta a(x,\xi)|\le C\est\xi^{k-\beta}, \text{ where }\est\xi=(1+|\xi|^2)^{1/2}.
$$
For $a\in S^k$, $\Op_{sc}(a ) u $ has a sense for $u\in \S(\R^d)$, and can be extended for $u\in H^s(\R^d)$ for
every $s\in\R$.

In a \nhd of $\d\Omega$ it is useful to use a tangential calculus. For a smooth function $a(x,\xi')$ we 
associate a tangential operator by the following formula
$$
\op_{sc} (a) u= (2\pi)^{-d+1}\int_{\R^{d-1}} e^{ix'\xi'}a(x,h\xi') \tilde u(\xi',x_d)d\xi', \text{ where } 
\tilde u(\xi',x_d)=\int_{\R^{d-1}} e^{-ix'\xi'}u(x',x_d)dx'.
$$
This formula make sense if $a\in S^k_{\rm tan}$, that is, for every $\alpha\in\N^d$, $\beta\in\N^{d-1}$, 
there exists $C=C_{\alpha,\beta}$ such
that 
$$
|\d_x^\alpha \d_{\xi'}^\beta a(x,\xi')|\le C\est{\xi'}^{k-\beta}, \text{ where }\est{\xi'}=(1+|\xi'|^2)^{1/2}.
$$
We also use this notation for pseudo-differential operator on the boundary $x_d=0$. In this case, $a$ and $u$
does not depend on $x_d$. 

For technical reason we also have to use other classes of symbols. In these cases we use the H\"ormander's
notations, for instance, $S(\est\xi^m, (dx)^2+\est{\xi'}^{-2}(d\xi')^2)$,  $S(\est {\xi' }^m, (dx)^2+(d\xi')^2)$. 
 In this case we keep the notations $\Op_{sc}(a)$, if the symbol depends on $\xi_d$ and $\op_{sc}(a)$, 
 if the symbol does not depend on $\xi_d$. We also use the notation
$\op_{sc}(b)=b(x,hD')$, in particular when we restrict a function on $x_d=0$, this allows to distinguish 
$b(x,hD')$ and $b(x',0,hD')$.

The main interest of pseudo-differential operators are the calculus of products, commutators, adjoints. We 
have for $a\in S^k$ and $b\in S^m$,
$$
\Op_{sc} (a)\Op_{sc} (b)=\Op_{sc} (c), \text{ where } c\in S^{m+k},
$$
and $c$ admits an asymptotic expansion, $c(x,\xi)= a(x,\xi)b(x,\xi)+hd(x,\xi)$, where $d\in S^{m+k-1}$.
$$
[\Op_{sc} (a),\Op_{sc} (b)]=h\Op_{sc} (c),  \text{ where } c\in S^{m+k-1},
$$
and $c$ admits an asymptotic expansion, $c(x,\xi)= -i\{a,b\}(x,\xi)+hd(x,\xi)$, where $d\in S^{m+k-2}$, and
$\{a,b\}(x,\xi)=\sum_{j=1}^d \big(\d_{\xi_j}a(x,\xi)\d_{x_j}b(x,\xi)-\d_{x_j}a(x,\xi)\d_{\xi_j}b(x,\xi)\big)$ is the 
Poisson bracket. At some point, it is useful to use that the commutator between an 
operator and  a derivative admits an exact 
formula, we have $[ hD_{x_j} ,\Op_{sc} a ]= -ih\Op_{sc}( \d_{x_j}a)$.

For  $a\in S^k$, we have
$$
\Op_{sc}(a)^*=\Op_{sc}(b) \text{ and } b(x,\xi')=\bar a(x,\xi)+hc(x,\xi)  \text{ where } b\in S^k, c\in S^{k-1}.
$$
The asymptotic expansions may be extended to all power in $h$. Analogous formulas exist for $\op_{sc}a$.

Associated with this semiclassical calculus, we introduce the semiclassical Sobolev spaces.
For $u\in \S'(\R^d)$, we define $\| u\|_{H^s_{sc}}=\| \Op_{sc}(\est{\xi}^s)u\|_{L^2(\R^d)}$, if  
$\Op_{sc}(\est{\xi}^s)u\in L^2(\R^d)$. On the boundary we define for $u\in \S'(\R^{d-1})$,  
$| u|_{H^s_{sc}(x_d=0)}=| \op_{sc}(\est{\xi'}^s)u|_{L^2(\R^{d-1})}$, if  
$\op_{sc}(\est{\xi'}^s)u\in L^2(\R^{d-1})$. On the boundary we define $(u|v)_0=\int_{\R^{d-1}} u(x')\bar v(x')dx'$.

We keep the same notation for a general $\Omega$, namely 
 $(u|v)_0=\int_{ \d\Omega }
u(x')\bar v(x')d\sigma( x') $, where $\sigma $ is the superficial measure on $\d\Omega$.

Pseudo-differential operators act on Sobolev spaces. For $a\in S^k$, there exists $C>0$ such that
$$
\| \Op_{sc}(a)u\|_{H^{s-k}_{sc}}\le C\| u\|_{H^{s}_{sc}}, \text{  for 
every } u\in H^s_{sc}.
$$
For $a\in S_{\rm tan}^k$,  there exists $C>0$ such that
$$
| \op_{sc}(a)u|_{H^{s-k}_{sc}}\le C| u|_{H^{s}_{sc}(x_d=0)},  \text{  for 
every } u\in H^s_{sc}(x_d=0).
$$
When we consider Sobolev spaces on $x_d>0$, it is useful to consider $\op_{sc}$ and distinguish 
variable $x_d\in (0,\infty)$ and variables $x'\in\R^{d-1}$. Let $L^2((0,\infty),H^{s}_{sc})$ be the space 
such that $u\in L^2((0,\infty),H^{s}_{sc})$ if $\int_{\R^{d-1}}\int_0^\infty | \op_{sc} (\est{\xi'}^s)u(x',x_d)|^2dx_ddx'=\|  u \|_{L^2( (0,\infty),H^{s}_{sc})}^2<\infty $. We have the 
following estimate, let $a\in S_{\rm tan}^k$, there exists $C>0$, such that
$$
\| \op_{sc}(a)u\|_{L^2((0,\infty),H^{s-k}_{sc})}\le C\| u\|_{L^2( (0,\infty),H^{s}_{sc})}, , \text{  for 
every } u\in L^2((0,\infty),H^{s}_{sc}).
$$
In the context of semiclassical Sobolev spaces we have the following trace formula. Let $s>0$, 
there exists $C>0$, such that
\begin{align} 
 \label{eq: trace formula Hs}
| u_{|x_d=0} |_{H^s_{sc}}\le Ch^{-1/2}\| u \|_{H^{s+1/2}_{sc}}, \text{ for every } u\in H^{s+1/2}_{sc}.
\end{align}

We recall the G\aa rding inequality for semiclassical  Sobolev spaces.
 Let $a\in S^0$  be such that 
$a(x,\xi)\ge 0$, there exists $C>0$, such that
\begin{align}  \label{eq: Garding inequality R d}
\Re (\Op_{sc}(a)u|u)_{L^2( \R^{d})}+Ch\| u\|_{L^2( \R^{d}) }^2\ge 0.
\end{align}
Here and in what follows $(w|v)_{L^2(K)}$ means the inner product in $K$.

For tangential symbol we have the analogous result.
Let $a\in S^0_{\rm tan}$ be such that 
$a(x,\xi')\ge 0$, there exists $C>0$, such that
\begin{align} 
	 \label{eq: Garding inequality}
\Re (\op_{sc}(a)u|u)_{L^2((0,\infty)\times \R^{d-1})}+Ch\| u\|_{L^2((0,\infty)\times \R^{d-1}) }^2\ge 0.
\end{align}
We use consequences of this result. Let $a\in S^0_{\rm tan}$, such that there exists $K>0$ such that 
$|a(x,\xi')|\le K$, then
\begin{align} \label{eq: estimation L2 sharp with Garding}
\| \op_{sc}(a)u\|_{L^2((0,\infty)\times \R^{d-1} )}\le 2K\| u \|_{L^2((0,\infty)\times \R^{d-1} )} 
+Ch\| u \|_{L^2((0,\infty)\times \R^{d-1} )} ,
\end{align}
where   $C>0$, depends on a finite number of seminorm of $a$.
We have the same estimate at the boundary, if $a\in S(1,(dx')^2+(d\xi')^2)$, and $|a(x',\xi')|\le K$, we have 
\begin{align} \label{eq: estimation L2 sharp with Garding boundary}
|\op_{sc}(a)u|_{L^2(\R^{d-1} )}\le 2K | u |_{L^2( \R^{d-1} )} 
+Ch| u |_{L^2(\R^{d-1} )} ,
\end{align}
where   $C>0$, depends on a finite number of seminorm of $a$.
In particular we use this estimate if $a$ depends on a parameter but $K$ is uniform with 
respect this parameter. In this case, in the previous estimate $C$ depends on the parameter.
For $w\in L^2(\Omega)$, we extends $w$ for $x\in \R^d\setminus\Omega$ by 0, and we use the 
following notations
\begin{align}  \label{def: extension omega}
 \underline{w} (x)= 1_{\Omega} w(x)=
\begin{cases}
w(x)  \text{ if } x\in \Omega , \\
0   \text{ if } x\in \R^d\setminus\Omega.
\end{cases}
\end{align}
If $w\in L^2(\R^{d-1}  \times  (0,\infty) )$, we extend $w$ by 0 for $x_d<0$ and  we use the 
notation 
 $ \underline{w} (x)= 1_{{x_d>0}} w(x)$.

 In this article we use the symbol $\lesssim$: $A\lesssim B$  means, there exists $C>0$, $A\le CB$, where $C$ is 
 independent of parameters.  
 
 We denote $z^s=\exp(s \log(z))$, where $\log z$ is the principal value of the logarithm, where $z\in \C\setminus \R^-$.
 
\subsection{Evolution equation and resolvent estimate} \label{Evolution equation and resolvent estimate}
We begin the proof of the third item of Theorem~\ref{th: resolvent estimate}. We may assume 
$|\mu|\ge 1$ as $A-i\mu I$ is invertible for all $\mu\in\R$, by second item and  
$\mu\mapsto (A-i\mu I)^{-1}$ is continuous from $\R$ to 
${\cal L}(H,H)$. 

Let $F=(f_0,f_1) \in H$ and let $U=(u_0,u_1)\in {\cal D}(A) $ be such that 
$ AU-i\mu U=F$, we have
\begin{equation}
	\label{eq: system on U}
\begin{cases}
u_1-i\mu u_0=f_0\\
-P u_0-a u_1-i\mu u_1=f_1.
\end{cases}
\end{equation}
\begin{lemma}
Assume that there exists $C_1>0$ such that 
\begin{equation}
	\label{eq: equation on u_0}
|\mu|\| u_0\|_{L^2(\Omega)}+\|\nabla  u_0\|_{L^2(\Omega)}\le C_1(\| f_0\|_{H^1(\Omega)}+\| f_1\|_{L^2(\Omega)}) 
\text{ for all } (f_0,f_1)\in H,
\end{equation}
where $(u_0,u_1)$ are the solution of~\eqref{eq: system on U}. Then  there exists $C_2>0$ such that 
\begin{align*}
\| U\|_{H}\le C\| F\|_H, \text{ for all } F\in H,
\end{align*}
where $ AU-i\mu U=F$.
\end{lemma}
\begin{proof}
By~\eqref{eq: system on U}, $u_1=i\mu u_0+f_0$, so that
\begin{equation*}
\| u_1\|_{L^2(\Omega)}\le \| f_0\|_{L^2(\Omega)}+|\mu|\| u_0\|_{L^2(\Omega)}  \lesssim
\| f_0\|_{H^1(\Omega)}+\| f_1\|_{L^2(\Omega)},
\end{equation*}
and \eqref{eq: equation on u_0} gives $ \| u_0\|_{H^1(\Omega)}\lesssim \| f_0\|_{H^1(\Omega)}
+\| f_1\|_{L^2(\Omega)}$, for $|\mu|\ge 1$. This gives the result.
\end{proof}

Formula~\eqref{eq: system on U} implies the following equation on $u_0$
\begin{align*}
-P u_0+\mu^2u_0-i\mu au_0=af_0+i\mu f_0+f_1.
\end{align*}

To use semiclassical tools, we set 
$h=1/\mu$, we multiply \eqref{eq: equation on u_0} by  $h^2$, we obtain the following equation on $u_0$
\begin{equation*}
-h^2P u_0+u_0-ih au_0=ah^2f_0+ihf_0+h^2f_1,
\end{equation*}
and \eqref{eq: equation on u_0} is equivalent to
\begin{equation}
	\label{eq: semiclassical estimation}
\| u_0\|_{L^2(\Omega)}+\|h\nabla  u_0\|_{L^2(\Omega)}\le Ch(\| f_0\|_{H^1(\Omega)}+\| f_1\|_{L^2(\Omega)}).
\end{equation}
We shall prove this inequality by contradiction. If \eqref{eq: semiclassical estimation} is false, 
up to a normalization, there 
exist a sequence $h_n\to 0$ as $n\to \infty$ denoted for sake of simplicity by $h$, 
$(u_h)_h \in H^1(\Omega)$ and $(f_0^h,f_1^h)_h\in H$ satisfying 
\begin{align}
	\label{eq: semiclassical equation on u_h}
&-h^2P u_h+u_h-ih au_h=ah^2f_0^h+ihf_0^h+h^2f_1^h  \notag\\
&\| u_h\|_{L^2(\Omega)}+\|h\nabla  u_h\|_{L^2(\Omega)}=1 \notag \\
&h(\| f_0^h\|_{H^1(\Omega)}+\| f_1^h\|_{L^2(\Omega)})\to 0 \text{ as } h\to 0.
\end{align}
Let $g_0^h=ahf_0^h+hf_1^h$ and  $g_1^h=ihf_0^h$, \eqref{eq: semiclassical equation on u_h} is equivalent to 
\begin{align}
	\label{eq: semiclassical equation on u_h-2}
&-h^2P u_h+u_h-ih au_h=hg_0^h+g_1^h  \notag\\
&\| u_h\|_{L^2(\Omega)}+\|h\nabla  u_h\|_{L^2(\Omega)}=1 \notag \\
&\| g_0^h\|_{L^2(\Omega)}+\| g_1^h\|_{H^1(\Omega)}\to 0 \text{ as } h\to 0.
\end{align}
%%%%%%%%%%%%%%%%%%%%
%
%  Proposition
%
%%%%%%%%%%%%%%%%%%%%
\begin{proposition} \label{prop: equation on v spectrally localised}
There exist $\beta>\alpha>0$,
there exists $\theta\in\Con_0^\infty(\R)$,  supported in $[\alpha,\beta]$, there exists $(\tilde u_h)_h$, 
satisfying   $\|\tilde u_h\|_{L^2(\Omega)}+\|h\nabla 
 \tilde u_h\|_{L^2(\Omega)}\le C$,  for some $C>0$,
such that
\begin{align}
	\label{eq: semiclassical equation on v_h}
&-h^2P v_h+v_h-ih av_h=hq_h , \notag\\
&\| v_h\|_{L^2(\Omega)}=1    \text{ and }\|h\nabla  v_h\|_{L^2(\Omega)}\le 2 ,\notag \\
&\| q_h\|_{L^2(\Omega)}\to 0 \text{ as } h\to 0,
\end{align}
where $v_h=\theta(h^2P)\tilde u_h$.
\end{proposition}
\begin{proof}
Let $\psi\in\Con_0^\infty(\R)$, $0\le \psi\le 1$ such that 
$$
\psi(s)=
\begin{cases}
1 \text{ if } s\le 1\\
0   \text{ if } s\ge 2.
\end{cases}
$$
Let $A>0$ be sufficiently large to be fixed below. We have 
$$
1=\lim_{n\to \infty}\psi(2^{-n}A^{-1}s)=\psi(sA^{-1})+\sum_{k=1}^\infty\big(\psi(2^{-k}A^{-1}s)   -\psi(2^{-k+1}A^{-1}s)  \big).
$$
Setting $\phi(s)=\psi(s)-\psi(2s)$, we have 
\begin{equation}
	\label{eq: decomposition Littlewood-Paley}
1=\psi(sA^{-1})+\sum_{k=1}^\infty\phi(2^{-k}A^{-1}s) \text{ and }  \phi  \text{  is supported in } [1/2, 2].
\end{equation}
By functional calculus for auto-adjoint operators, we have 
\begin{equation}  \label{formula: Littlewood-Paley}
I=\big( \psi(A^{-1}h^2P)- \psi(Ah^2P)   \big)+\psi(Ah^2P)+\sum_{k=1}^\infty\phi(2^{-k}A^{-1}h^2P).
\end{equation}

%%%%%%%%%%%%%%%%%%%%
%
%  Lemma
%
%%%%%%%%%%%%%%%%%%%%
\begin{lemma} \label{Lemma: petites et hautes frequences}
There exists $C>0$ such that 
\begin{align}
&\| \psi(Ah^2P)u_h\|_{L^2(\Omega)}\le C\big( A^{-1}+h
+h\|g_0^h\|_{L^2(\Omega)}+\| g_1^h\|_{L^2(\Omega)}
\big),
\label{eq: Lemma: petites et hautes frequences-1}\\
&\big\| \sum_{k=1}^\infty\phi(2^{-k}A^{-1}h^2P)u_h\big\|_{L^2(\Omega)}\le CA^{-1} 
\big(   
 1+h\| g_0^h \|_{L^2(\Omega)}+\| g_1^h \|_{L^2(\Omega)}\big).
\label{eq: Lemma: petites et hautes frequences-2}
\end{align}
\end{lemma}
\begin{proof} 
We apply $\psi(Ah^2P)$ to equation~\eqref{eq: semiclassical equation on u_h-2}, we obtain 
$$
-h^2P \psi(Ah^2P) u_h+\psi(Ah^2P)u_h-ih\psi(Ah^2P) (au_h)
=h\psi(Ah^2P)g_0^h+\psi(Ah^2P)g_1^h.
$$
Let $\tilde\psi (s)=s\psi(s)$ we have 
$$
\psi(Ah^2P)u_h=A^{-1} \tilde\psi(Ah^2P) u_h+ih\psi(Ah^2P) (au_h)
+h\psi(Ah^2P)g_0^h+\psi(Ah^2P)g_1^h.
$$
As $ |\psi(s)|\le 1$ and $|\tilde\psi(s)|\le C$, we obtain~\eqref{eq: Lemma: petites et hautes frequences-1}.

To prove~\eqref{eq: Lemma: petites et hautes frequences-2} first we estimate $\phi(2^{-k}A^{-1}h^2P)u_h$.
Let $\tilde \phi\in\Con_0^\infty((0,\infty))$ to be fixed below. We apply $\tilde \phi(2^{-k}A^{-1}h^2P)$
to equation~\eqref{eq: semiclassical equation on u_h-2}, we obtain 
\begin{align*}
-h^2P \tilde \phi(2^{-k}A^{-1}h^2P) u_h&=-\tilde \phi(2^{-k}A^{-1}h^2P)u_h
+ih \tilde \phi(2^{-k}A^{-1}h^2P)( au_h) \\
&\quad +h\tilde \phi(2^{-k}A^{-1}h^2P)g_0^h+\tilde \phi(2^{-k}A^{-1}h^2P)g_1^h .
\end{align*}
Let $\tilde \phi(s)=-s^{-1}\phi(s)$ be supported in  $[1/2, 2]$. We obtain
\begin{align*}
2^kA \phi(2^{-k}A^{-1}h^2P) u_h&=-\tilde \phi(2^{-k}A^{-1}h^2P)u_h
+ih \tilde \phi(2^{-k}A^{-1}h^2P)( au_h) \\
&\quad +h\tilde \phi(2^{-k}A^{-1}h^2P)g_0^h+\tilde \phi(2^{-k}A^{-1}h^2P)g_1^h .
\end{align*}
This yields
$$
2^kA\| \phi(2^{-k}A^{-1}h^2P) u_h\|_{L^2(\Omega)}\le
 C\big(   
 1+h\| g_0^h \|_{L^2(\Omega)}+\| g_1^h \|_{L^2(\Omega)}
\big).
$$
Summing over $k$ we obtain~\eqref{eq: Lemma: petites et hautes frequences-2}.
\end{proof}
Let $\theta(s)= \psi(A^{-1}s)-\psi(As)$, by Lemma~\ref{Lemma: petites et hautes frequences} and   
\eqref{formula: Littlewood-Paley} choosing $A$ sufficiently large and $h\in(0,h_0]$ for $h_0>0$ 
sufficiently small, we can have 
$\| u_h- \theta(h^2P)u_h   \|_{L^2(\Omega)}$ as small as we want.
From equation~\eqref{eq: semiclassical equation on u_h-2} multiplying by $\overline{u}_h$ 
and integrating by parts, we obtain 
$$
-(Pu_h|u_h)_{L^2(\Omega)}
+  \int_\Omega |u_h|^2  dx =  ih(au_h|u_h)_{L^2(\Omega)}
+h(g_0^h|u_h)_{L^2(\Omega)}+(g_1^h|u_h)_{L^2(\Omega)}.
$$
Taking $h_0$ sufficiently small, we have
\begin{equation}
	  \label{eq: norm L2 and H1 are equivalent}
h^2(Pu_h|u_h)_{L^2(\Omega)}
= \int_\Omega |u_h|^2  dx +\varepsilon_h \text{ where } \varepsilon_h\to 0
\text{  as } h\to 0.
\end{equation}
We now observe that
$h^2(Pu_h|u_h)_{L^2(\Omega)}$ is equivalent to $H^1_{sc}(\Omega)$-norm, uniformly with respect to $h\in(0,1)$. From~\eqref{eq: semiclassical equation on u_h-2}, the assumption on norms consequently gives
$\| u_h\|_{L^2(\Omega)}
\ge C_0+\tilde\varepsilon_h$, where $C_0>0$ and $\tilde\varepsilon_h\to 0$ as $h\to0$.  

Let $v_h= \theta(h^2P)u_h  / \| \theta(h^2P) u_h\|_{L^2(\Omega)}$. 
 Now we prove that  $v_h $ satisfies the equation. We apply $\theta(h^2P)$ 
to~\eqref{eq: semiclassical equation on u_h-2}, we have
\begin{align*}
h^2P  \theta(h^2P)u_h+ \theta(h^2P)u_h-ih \theta(h^2P)( au_h)
=h \theta(h^2P)g_0^h+ \theta(h^2P)g_1^h
\end{align*}
which is equivalent to 
\begin{align}   \label{eq: definition of qh}
h^2P v_h+v_h-ih av_h=  \| \theta(h^2P) u_h\|_{L^2(\Omega)}^{-1}\big(
  ih[\theta(h^2P), a]u_h+h \theta(h^2P)g_0^h+ \theta(h^2P)g_1^h   \big)=hq_h.
\end{align}
To obtain the estimate on $q_h$ we have to prove
\begin{align}
&\| [\theta(h^2P), a]u_h  \|_{L^2(\Omega)}\to 0 \text{ as } h\to 0  ,
\label{eq: commutator estimate}\\
&h^{-1}\|   \theta(h^2P)g_1^h     \|_{L^2(\Omega)}\to 0 \text{ as } h\to 0 .
\label{eq: estimate second member}
\end{align}  
To do that we need the following result proved below.
%%%%%%%%%%%%%%%%%%%%
%
%  Lemma
%
%%%%%%%%%%%%%%%%%%%%
\begin{lemma}   \label{lemma: commutator estimation}
Let $\phi \in \Con_0^\infty(0,\infty)$, and $a\in\Con_0^\infty(\Omega)$, there exists $C>0$ such that
\begin{align*}
& \| [\phi (h^2P), a]w  \|_{L^2(\Omega)}\le C h\| w  \|_{L^2(\Omega)}, \\
& \| [\d_{x_j}, \phi (h^2P)]w  \|_{L^2(\Omega)}\le C \| w  \|_{L^2(\Omega)} ,\\
& \| \d_{x_j}\phi (h^2P)w  \|_{L^2(\Omega)}\le C \| w  \|_{H^1(\Omega)} ,
\end{align*}
for $j=1,\dots,d$.
\end{lemma}
Estimate~\eqref{eq: commutator estimate} is a direct consequence 
of Lemma~\ref{lemma: commutator estimation}.
To prove~\eqref{eq: estimate second member}, let $\tilde \theta(s)=s^{-1}\theta(s)\in\Con_0^\infty(0,\infty)$, we compute
\begin{align*}
h^{-2}\|  \theta(h^2P)g_1^h  \|_{L^2(\Omega)}^2&=h^{-2}(  \theta(h^2P)g_1^h |  \theta(h^2P)g_1^h )_{L^2(\Omega)}, \\
&= h^{-2}( h^2P \tilde\theta(h^2P)g_1^h |  \theta(h^2P)g_1^h )_{L^2(\Omega)}, \\
&=\sum_{1\le j,k\le d} (p_{jk}\d_{x_j}\tilde\theta(h^2P) g_1^h|   \d_{x_k }  \theta(h^2P)g_1^h )_{L^2(\Omega)}  \\
&\quad + \sum_{1\le j\le d} (-ip_{j}\d_{x_j}\tilde\theta(h^2P) g_1^h|   
\theta(h^2P)g_1^h )_{L^2(\Omega)}
+ (p_{0}\tilde\theta(h^2P) g_1^h|   
\theta(h^2P)g_1^h )_{L^2(\Omega)}
\end{align*}
as $\tilde\theta(h^2P)g_1^h$ and $
 \theta(h^2P)g_1^h $  are in domain of  $ P$ we can apply \eqref{eq: associated quadratic form to P}.  By Lemma~\ref{lemma: commutator estimation}, we obtain
 $$
 h^{-2}\|  \theta(h^2P)g_1^h  \|_{L^2(\Omega)}^2\le C \| g_1^h  \|_{H^1(\Omega)}^2.
 $$
 This completes the proof of Proposition~\ref{prop: equation on v spectrally localised} because 
 arguing as  we did to obtain Formula~\eqref{eq: norm L2 and H1 are equivalent}, we get 
 $\| v_h\|_{L^2(\Omega)}=1$ so that $\|h D v_h\|_{L^2(\Omega)}=1+\eps_h$ with $\eps_h\to 0$ as $h\to 0$.
\end{proof}
\begin{proof} {\it Lemma~\ref{lemma: commutator estimation}.} We start with the following lemma which 
will be proved below.
%%%%%%%%%%%%%%%%%%%%
%
%  Lemma
%
%%%%%%%%%%%%%%%%%%%%
\begin{lemma}  \label{lem: resolvent estimates}
Let $C_1>0$. There exists $C>0$  such that 
\begin{align*}
&\| (-h^2P+z)^{-1}f\|_{L^2(\Omega)}\le C|\Im z|^{-1} \| f\|_{L^2(\Omega)} , \\
&\| h\d_{x_j}(-h^2P+z)^{-1}f\|_{L^2(\Omega)}\le C|\Im z|^{-1} \| f\|_{L^2(\Omega)},  \\
&\|(-h^2P+z)^{-1} h\d_{x_j} f\|_{L^2(\Omega)}\le C|\Im z|^{-1} \| f\|_{L^2(\Omega)} ,
\end{align*}
for all $|z|\le C_1$ and $f\in L^2(\Omega)$.
\end{lemma}
To prove Lemma~\ref{lemma: commutator estimation} we use the Helffer-Sj\"ostrand formula,
$$
\phi(h^2P)=-\frac1\pi \int \bar\d_z\tilde\phi(x,y) (-h^2P +z)^{-1}dxdy,
$$
where $z=x+iy\in\C$ and $\tilde\phi$  is an almost analytic extension of 
$\phi$ (see \cite[Proposition 7.2]{HS:1989} and \cite{DA}):
 $\tilde\phi$  
is compactly supported 
and satisfies
\begin{align*}
&\tilde\phi(x,0)=\phi(x), \\
&|\bar\d_z\tilde\phi(x,y)|\le C_N|y|^N, \text{ for every } N.
\end{align*}
We recall that $\bar\d_z= (1/2)(\d_x+i\d_y)$.  
The Helffer-Sj\"ostrand formula gives 
\begin{align*} 
[\phi(h^2P),a]&=-\frac1\pi \int \bar\d_z\tilde\phi(x,y)[ (-h^2P +z)^{-1},a]dxdy \\
&=\frac1\pi \int \bar\d_z\tilde\phi(x,y)    (-h^2P +z)^{-1}[- h^2P ,a] (-h^2P +z)^{-1}dxdy .
\end{align*}
As $ [- h^2P ,a] $ is a sum of terms of following type $ bh^2\d_{x_j}$ and $ch^2$, where $b$ and $c$ are into 
$\Con^\infty(\overline\Omega)$, the first estimate of the lemma is given by the two following estimates
\begin{align}   \label{eq: estimation order 0}
\|    (-h^2P +z)^{-1}ch^2(-h^2P +z)^{-1}  w \|_{L^2(\Omega)}
&\lesssim h^2 |\Im z|^{-1} \|    (-h^2P +z)^{-1}  w \|_{L^2(\Omega)}  \notag \\
& \lesssim h^2 |\Im z|^{-2} \|   w \|_{L^2(\Omega)},
\end{align}
by the first estimate of Lemma~\ref{lem: resolvent estimates}.
\begin{align}
\|    (-h^2P +z)^{-1}   bh^2\d_{x_j} (-h^2P +z)^{-1}  w \|_{L^2(\Omega)}
&\lesssim  h |\Im z|^{-1} \|h \d_{x_j}    (-h^2P +z)^{-1}  w \|_{L^2(\Omega)}   \notag\\
&\lesssim  h |\Im z|^{-2} \|   w \|_{L^2(\Omega)}, \label{eq: estimation order 1}
\end{align}
by the two first estimates of Lemma~\ref{lem: resolvent estimates}.

For the second estimate we have  by the Helffer-Sj\"ostrand formula
\begin{align*}
 [\d_{x_j}, \phi (h^2P)]&=-\frac1\pi \int \bar\d_z\tilde\phi(x,y)[\d_{x_j}, (-h^2P +z)^{-1}]dxdy  \\
 &=\frac1\pi \int \bar\d_z\tilde\phi(x,y) (-h^2P +z)^{-1}[\d_{x_j}, -h^2P +z] (-h^2P +z)^{-1}dxdy,
\end{align*} 
and $[\d_{x_j}, -h^2P +z] $ is a sum of term of type $ h^2\d_{x_j} b \d_{x_k}$, $ch^2\d_{x_j}$ and  $dh^2$ 
where $b, c$ and $d$ are in $\Con^\infty(\overline\Omega)$. 
The terms with $ch^2\d_{x_j}$ and  $dh^2$ were estimated in~\eqref{eq: estimation order 0} 
and \eqref{eq: estimation order 1}. For the term $ h^2\d_{x_j} b \d_{x_k}$, we have
\begin{align*}
\|    (-h^2P +z)^{-1}  h^2\d_{x_j} b \d_{x_k} (-h^2P +z)^{-1}  w \|_{L^2(\Omega)}
&\lesssim  |\Im z|^{-1} \|h \d_{x_j}    (-h^2P +z)^{-1}  w \|_{L^2(\Omega)} , \\
&\lesssim   |\Im z|^{-2} \|  w \|_{L^2(\Omega)} ,
 \label{eq: estimation order 2}
\end{align*}
By the second and third estimate of Lemma~\ref{lem: resolvent estimates}.  This gives the second estimate of 
Lemma~\ref{lemma: commutator estimation}.

To prove the third estimate of Lemma~\ref{lemma: commutator estimation} we write
$$
\d_{x_j}\phi (h^2P)= \phi (h^2P)\d_{x_j} + [\d_{x_j},\phi (h^2P)].
$$
The first term is clearly estimated by $H^1$-norm and second term is estimated by the second inequality of Lemma~\ref{lemma: commutator estimation}.
\end{proof}
\begin{proof}{\it Lemma~\ref{lem: resolvent estimates}.} 
Let  $u= (-h^2P+z)^{-1}f$, we have $u\in H^1(\Omega) $ and $u$ satisfies $(-h^2P+z)u=f$ 
and the Zaremba boundary condition. Multiplying the equation by $\overline u$, integrating over $\Omega$  and
an performing integration by parts,
we get
\begin{equation} \label{eq: formula for resolvent}
-(h^2P  u|u)_{L^2(\Omega)}+z\|  u\|_{L^2(\Omega)}^2=(f|u).
\end{equation}
Taking the imaginary part of equation we have $ |\Im z|\|  u\|_{L^2(\Omega)}^2
\le \|  f\|_{L^2(\Omega)}\|  u\|_{L^2(\Omega)}$, which gives the first estimate.

Taking the real part of~\eqref{eq: formula for resolvent} and from  \eqref{hyp: positivity P} 
we have
$\|  \nabla u\|_{L^2(\Omega)}^2\lesssim (P  u|u)_{L^2(\Omega)}$, we obtain
$$
\|  h\nabla u\|_{L^2(\Omega)}^2\lesssim \|  f\|_{L^2(\Omega)}\|  u\|_{L^2(\Omega)} +\|  u\|_{L^2(\Omega)}^2.
$$
This gives the second estimate with the previous result.

It is sufficient to prove the third estimate of Lemma~\ref{lem: resolvent estimates} for $f\in\Con_0^\infty(\Omega)$ 
as $\Con_0^\infty(\Omega)$ is dense in $L^2(\Omega)$.
We have 
\begin{align*}
\|(-h^2P+z)^{-1} h\d_{x_j} f\|_{L^2(\Omega)}^2&=( (-h^2P+z)^{-1} h\d_{x_j} f | (-h^2P+z)^{-1} h\d_{x_j} f )\\
&=- (  f | h\d_{x_j}(-h^2P+z)^{-1}(-h^2P+z)^{-1} h\d_{x_j} f ) \\
&\le C |\Im z|^{-1} \| f\|_{L^2(\Omega)}  \| (-h^2P+z)^{-1} h\d_{x_j}  f\|_{L^2(\Omega)} ,
\end{align*}
by the previous result. This gives the third estimate of Lemma~\ref{lem: resolvent estimates}.
\end{proof}
 
\subsection{A priori estimate on traces} \label{A priori estimate on traces}
In this section we assume that $\Omega$ is locally given near a point of the boundary by $x_d>0$ 
and we denote
 the variables by $x=(x',x_d)$ where $x'\in\R^{d-1}$,  and we set  $\R^d_+=\{(x',x_d)\in\R^d, \  x_d>0  \}$. 
 In these local coordinates\footnote{Rigorously, the laplace-Beltrami operator has a term 
 \(h^2D_{x_d}\). We can eliminate this term after a conjugaison of  operator  by a function non 
 null everywhere but this has no influence on the proof given here. For simplicity we choose to keep 
 \(v_h)\) instead of the conjugated function.} we have $h^2D_{x_d} ^2v_h
+(R(x,hD')-1+iha)v_h=hq_h $. When there are no ambiguity we only write
$R$ instead of $R(x,hD')$.
\begin{proposition} 
	\label{prop: a priori estimates on traces}
Let $v_h$ given in Proposition~\ref{prop: equation on v spectrally localised}. Then, there exists $C>0$ such that 
\begin{align}
&|(v_h)_{|x_d=0}|_{H^{1/2}_{sc}}\le Ch^{-1/2}  , \label{estimation apriori trace}\\
&|(hD_{x_d}v_h)_{|x_d=0}|_{H^{-1/2}_{sc}}\le Ch^{-1/2} .
\end{align}
\end{proposition}
\begin{proof}
As $v_h\in H^1_{sc}(\Omega)$, the trace formula \eqref{eq: trace formula Hs} 
gives the first estimate.

Let $\chi\in\Con^\infty(\R)$ be such that $\chi(x_d)=1$ if $x_d\le \delta$,  $\chi(x_d)=0$ if $x_d\ge 2 \delta$ 
and $0\le \chi\le 1$. Firstly
\begin{align*}
&h| \op_{sc}(\est{\xi'}^{-1/2})hD_{x_d}v_h(x',0)|_{L^2}^2 = -i\int_0^\infty hD_{x_d}\big( \chi(x_d) 
|  \op_{sc}(\est{\xi'}^{-1/2})hD_{x_d} v_h(x',x_d)|^2 \big)dx_d  \\
&= -\int_0^\infty h\chi'(x_d) 
|  \op_{sc}(\est{\xi'}^{-1/2})hD_{x_d} v_h(x',x_d)|^2 dx_d  \\
&\quad  -2i\int_0^\infty   \chi(x_d) 
\Re \big( \op_{sc}(\est{\xi'}^{-1/2})h^2D_{x_d}^2 v_h(x',x_d)   |   \op_{sc}(\est{\xi'}^{-1/2})hD_{x_d} v_h(x',x_d)  \big)dx_d \\
&=-\int_0^\infty h\chi'(x_d) 
|  \op_{sc}(\est{\xi'}^{-1/2})hD_{x_d} v_h(x',x_d)|^2 dx_d  \\
&\quad  -2i\int_0^\infty   \chi(x_d) 
\Re \big( \op_{sc}(\est{\xi'}^{-1/2})(-R+1)v_h(x',x_d)   |   \op_{sc}(\est{\xi'}^{-1/2})hD_{x_d} v_h(x',x_d)  \big)dx_d  \\
&\quad  -2i\int_0^\infty   \chi(x_d) 
\Re \big( \op_{sc}(\est{\xi'}^{-1/2})h(q_h (x',x_d)-i(av_h)(x', x_d))   |   \op_{sc}(\est{\xi'}^{-1/2})hD_{x_d} v_h(x',x_d)  \big)dx_d\\
&= I_1+I_2+I_3.
\end{align*}
Secondly, we have
\begin{align*}
&|I_1|\lesssim h\|  hD_{x_d}v_h\|_{L^2(\Omega)}^2\le C,  \\
&|I_3|\lesssim h(\|  q_h  \|_{L^2(\Omega)}   +  \|   v_h\|_{L^2(\Omega)}\ ) \|   hD_{x_d}v_h\|_{L^2(\Omega)}\le C.
\end{align*}
Finally, we obtain
$$
I_2= -2i\int_0^\infty   \chi(x_d) 
\Re \big( \op_{sc}(\est{\xi'}^{-1})(-R+1)v_h(x',x_d)   | hD_{x_d} v_h(x',x_d)  \big)dx_d .
$$
By tangential semiclassical pseudo-differential calculus, $\op_{sc}(\est{\xi'}^{-1})(-R+1)$ is of order 1, then
$$
|I_2|\lesssim \|  \op_{sc}(\est{\xi'})  v_h \|_{L^2(\Omega)} \|  hD_{x_d}v_h\|_{L^2(\Omega)}\le C.
$$
This achieves the proof of the proposition.
\end{proof}
\section{Semiclassical  measure and  the characteristic set}
	\label{Sec: Semiclassical  measure and  the characteristic set}

%%%%%%%%%%%%%%%
%
% Subsection
%
%%%%%%%%%%%%%%%%%%

\subsection{Support of the semiclassical  measure}
\label{subsec: semiclassical defect measure is null on characteristic set}

We now define a semiclassical measure associated with $(v_h)_h$. It is classical, as $(v_h)_h$ is bounded in
 $L^2(\Omega)$, that there exists $\mu$ a measure on $T^\ast (\R^d)$ such that, up to extraction of subsequence 
 of $(v_h)_h$, we have for all $b(x,\xi) \in \Con_0^\infty(\R^{2d})$,
\begin{equation}
	 \label{def: semiclassical defect measure}
 (\Op_{sc}(b) \underline{v_h}| \underline{v_h} )_{L^2(\R^d)}\to \est{\mu , b}\text{ as } h\to 0.
\end{equation}
 For first expositions on microlocal defect measure or H-measure see~\cite{Gerard-1991,Tartar-1990}. 
 For semiclassical  measure or Wigner measure 
 see~\cite{Burq-1997,Gerard-1990,Miller-2000,Rob-Zuily-2009}.
The goal of this section is to prove the following result.
\begin{proposition}  
	\label{prop: measure supported on the characteristic set}
The measure $\mu$ is supported in $\overline \Omega \times \R^d$ and $p\mu=0$, where $p$ is the 
semiclassical symbol of $h^2P-1$ 
(see \eqref{def: symbol h2P-1}).
\end{proposition}
To prove this proposition we consider four sets in $T^\ast(\R^d)$, exterior or interior points 
(i.e. $x\in\R^d\setminus \overline \Omega$ or $x\in\Omega$), hyperbolic points  (i.e., $x\in\d\Omega$ and $R(x,\xi')-1<0$),
glancing points   (i.e., 
$x\in\d\Omega$ and $R(x,\xi')-1=0$),
and elliptic points  (i.e., 
$x\in\d\Omega$ and $R(x,\xi')-1>0$). 

The proofs for exterior and interior points are classical, we give the proofs for the sake of completeness. 
The proofs for hyperbolic and glancing points are similar to the proofs given by Burq and 
Lebeau~\cite{Burq-Lebeau2001} in the context of defect measures. 
The proof for elliptic points is 
specific to Zaremba boundary condition. Of course in $\d\Omega_D\cap\d\Omega_N$ the 
proof is also classical.
%%%%%%%%%%%%%%%
%
% Subsubsection
%
%%%%%%%%%%%%%%%%%%
\subsubsection{Interior and exterior points}

Clearly, for $\chi(x)\in \Con_0^\infty(\R^d\setminus \overline \Omega)$ and $\phi(\xi)\in \Con_0^\infty(\R^d)$, 
$(\Op_{sc}( \chi(x)\phi(\xi)  ) \underline{v_h}| \underline{v_h} )\to 0$ as $h\to 0$. Then $\mu$ is supported 
in  $\overline \Omega \times \R^d$.

 Let  $\chi(x)\in \Con_0^\infty(\Omega)$,  $\phi(\xi)\in \Con_0^\infty(\R^d)$
and  $\chi_1(x)\in \Con_0^\infty(\Omega)$ be such that $\chi_1\chi=\chi$, we have, by symbol calculus as 
the symbol  $p(x,\xi) -1 $ is in $S^2(\R^d\times \R^d)$ and $ \chi(x)\phi(\xi) \in S^{-\infty}(\R^d\times \R^d)$,

\begin{align*}
\Op_{sc}( \chi(x)\phi(\xi)  )  (h^2P -1)&= \Op_{sc}( \chi(x)\phi(\xi)  p(x,\xi) )  +h\Op_{sc}(r_0)\\
&=\Op_{sc}( \chi(x)\phi(\xi)  ) \chi_1(x)  (h^2P -1) +h\Op_{sc}(r_0'),
\end{align*}
where $r_0, \tilde r_0 , r_0' \in S^0(\R^d\times \R^d)$.
As $\chi_1$ is compactly supported in $\Omega$ and $P$ is a local operator 
$\chi_1(x) (h^2P-1) \underline{v_h}= \chi_1(x) (h^2P-1) v_h$, we have
\begin{align}  \label{eq: mesure nulle a l'interieur}
( \Op_{sc}( \chi(x)\phi(\xi)  )  (h^2P -1 )   \underline{v_h}| \underline{v_h} ) &=
( \Op_{sc}( \chi(x)\phi(\xi)  )   \chi_1(x) (h^2P -1 )  {v_h}| \underline{v_h} ) 
+  h{\cal O} (\| v_h\|_{L^2(\Omega)}  ^2)\notag \\
&=(\Op_{sc}( \chi(x)\phi(\xi)  )   \chi_1(x) (-iha -hq_h) {v_h}| \underline{v_h} ) +  h{\cal O} (\| v_h\|^2_{L^2(\Omega)}  ) \notag\\
&\quad \to 0 \text{ as } h\to 0 .
\end{align}
We also have
\begin{align}
( \Op_{sc}( \chi(x)\phi(\xi)  )  (h^2P -1 )   \underline{v_h}| \underline{v_h} ) &
=  (\Op_{sc}( \chi(x)\phi(\xi)  p(x,\xi) )   \underline{v_h}| \underline{v_h} ) +  h{\cal O} (\| v_h\|_{L^2(\Omega)} ^2 ) \notag\\
&\quad \to \langle \mu ,  p(x,\xi) \chi(x)\phi(\xi) \rangle \text{ as } h\to 0 . \notag
\end{align}
This and \eqref{eq: mesure nulle a l'interieur} give that $ \langle \mu ,  p(x,\xi) \chi(x)\phi(\xi) \rangle =0$. This proves
Proposition~\ref{prop: measure supported on the characteristic set} for the interior points, as the space spanned 
by functions $\chi(x)\phi(\xi)$ is dense in $\Con_0^\infty(\Omega\times\R^{d})$.

Before proving Proposition~\ref{prop: measure supported on the characteristic set} in a 
\nhd of the boundary, we have to 
prove estimates  more precise than Proposition~\ref{prop: a priori estimates on traces}.
We have to distinguish microlocally hyperbolic, glancing and elliptic points. 
For hyperbolic and glancing points the boundary condition play no role. 
For elliptic points we have to distinguish points in  \(\d\Omega_D \cup \d\Omega_N \), 
 and points in  \(\Gamma\). The results are stated  
in  Proposition~\ref{prop: convergence traces D and N}
and in Proposition~\ref{lemma: trace goes to 0}.

%%%%%%%%%%%%%%%
%
% Subsubsection
%
%%%%%%%%%%%%%%%%%%
\subsubsection{Hyperbolic points}
	\label{Sec: Hyperbolic points}
Let $\delta >0$ be sufficiently small, let $\chi=\chi_\eps\in\Con^\infty(\R^{2d})$ be such that 
\begin{equation} \label{def: cut-off for hyperbolic region}
\chi(x,\xi')=
\begin{cases} 
1 \text{ if } R(x,\xi')-1\le -\eps \text{ and } x_d\le \delta  \\
0 \text{ if } R(x,\xi')-1\ge -\eps / 2\text{ or  } x_d\ge 2\delta  ,
\end{cases}
\end{equation}
and $0\le \chi \le 1$. Observe that $\chi\in S^{-\infty}_{\rm tan}$ as supported for $| \xi'| \le C$, 
where $C$ depends on $R$.
\begin{proposition} \label{prop: estimation traces hyperbolic}
For any $\eps>0$, there exists $C_\eps>0$ such that 
\begin{align}
& |  \op_{sc}( \chi_{|x_d=0}) (v_h ) _{|x_d=0}  |_{H^1_{sc}}\le C_\eps \notag \\
&   |  \op_{sc}( \chi_{|x_d=0}) ( h D_{x_d}v_h  ) _{|x_d=0}  |_{L^2}\le C_\eps ,  
	\label{eq: normal derivative estimation, hyperbolic}
\end{align}   
for all $h\in (0,1]$.
\end{proposition}
\begin{remark}
The estimate on traces, in the hyperbolic region, are better than the one proved in 
Proposition~\ref{prop: a priori estimates on traces}.
\end{remark}
We begin the proof by giving a localization result which is useful in each region defined in what follows.
\begin{lemma}
	\label{lem: localization in each region}
Let $\chi(x, \xi'), \chi_1(x, \xi') \in S^0_{\rm tan}$, be such that $ \chi_1(x, \xi')=1$, for $(x,\xi')$ in the 
support of $\chi$. We assume $0\le \chi\le 1$ and $0\le\chi_1\le1$. 
Let $w_h=\op_{sc} (\chi)v_h$, there exists $q_2^h$ such that
\begin{align} \label{eq: localization symbol Hyp case}
\big(h^2D_{x_d}^2+\op_{sc}(\chi_1^2(R(x,\xi')-1))\big)w_h=hq_2^h , \text{ where } \|q_2^h  \|_{L^2(\Omega)}\le C,
\end{align}
for some $C>0 $ depending on semi-norms of $\chi$ and $\chi_1$.
\end{lemma}
\begin{proof}
We have, for $x_d>0$,  \begin{align*}
(h^2 D_{x_d}^2+ R(x,hD')-1) w_h&= h\op_{sc}(\chi)( q_h-iav_h)+ [  R(x,hD') , \op_{sc}(\chi)  ] v_h 
+[  h^2 D_{x_d}^2 ,\op_{sc}(\chi)   ]\\
& =hq^h_1.
\end{align*}
The symbol of $[  R(x,hD') , \op_{sc}(\chi)  ] $ is in $hS^1_{\rm tan}$ and by exact symbol 
calculus with $D_{x_d}$, we have
\begin{align*}
[  h^2 D_{x_d}^2 ,\op_{sc}(\chi)   ]&=
-i h^2D_{x_d} \op_{sc}(\d_{x_d}\chi ) -i h \op_{sc}(\d_{x_d}\chi )h D_{x_d} \\
&= -2i h \op_{sc}(\d_{x_d}\chi )h D_{x_d} + h^2\op_{sc}(r_0),
\end{align*}
where $r_0\in S^0_{\rm tan}$. By the properties of $v_h$ given in Proposition~\ref{prop: equation on v spectrally localised}
and the previous formulas, there exists $C >0$ such that $\| q_1^h\|_{L^2}\le C $.

We have  
$$
R(x,hD')-1 =\op_{sc}(\chi_1^2(R(x,\xi')-1)) + \op_{sc}((1-\chi_1^2)(R(x,\xi')-1)),
$$
thus 
\begin{align*}
\big(h^2 D_{x_d}^2+ \op_{sc}(\chi_1^2(R(x,\xi')-1))\big)w_h= hq^h_1 -  \op_{sc}((1-\chi_1^2)(R(x,\xi')-1)) \op_{sc}(\chi)v_h=hq_2^h.
\end{align*}
As $(1-\chi_1^2)\chi=0$ by assumption, we deduce from symbol calculus that 
$$
\|   \op_{sc}((1-\chi_1^2)(R(x,\xi')-1)) \op_{sc}(\chi)v_h \|_{L^2(\Omega)}\le Ch.
$$
This gives \eqref{eq: localization symbol Hyp case}. 
\end{proof}
\begin{proof}[Proof of Proposition~\ref{prop: estimation traces hyperbolic}]

Let $\chi_1\in\Con^\infty(\R^{2d})$ be such that 
$$
\chi_1(x,\xi')=
\begin{cases}
1 \text{ if } R(x,\xi')-1\le -\eps/2 \text{ and } x_d\le 2\delta  \\
0 \text{ if } R(x,\xi')-1\ge -\eps / 4\text{ or  } x_d\ge 3\delta  ,
\end{cases}
$$
in particular $\chi_1$ is 1 on the support of $\chi$.
Let $b(x,\xi')=\chi_1(x,\xi')\big(  1-R(x,\xi') \big)^{1/2}$, observe that $1-R(x,\xi')>0$ on the support of $\chi_1$. We denote   $w_h=\op_{sc} (\chi)v_h$, where $\chi$, is defined by~\eqref{def: cut-off for hyperbolic region}.

%%%%%%%%%%%%%%%%%%%%
%
%  Lemma
%
%%%%%%%%%%%%%%%%%%%%
\begin{lemma} \label{lem: factorization hyperbolic}
There exist $C_\eps>0$, $q_3^h$ and $q_4^h$ such that 
\begin{align*}
&  \big( hD_{x_d}  - \op_{sc}(b)   \big)   \big( hD_{x_d}  +   \op_{sc}(b)   \big) w_h=hq^h_3  \\
&   \big( hD_{x_d}  + \op_{sc}(b)   \big)   \big( hD_{x_d}  - \op_{sc}(b)   \big)  w_h =hq^h_4,
\end{align*}
where $ \| q_j^h  \|_{L^2(\Omega)}\le C_\eps$ for $j=3,4$.
\end{lemma}
As the semi-norms of $\chi, \chi_1$ depend on $\eps$ the estimates depend on $\eps$ and to keep in mind that, we denote these constants by $C_\eps$.
\begin{proof} 
Let $k=1,2 $. We have
\begin{align*}
 \big( hD_{x_d}  - (-1)^k \op_{sc}(b)   \big)   \big( hD_{x_d}  +  (-1)^k \op_{sc}(b)   \big) 
& =  h^2 D_{x_d}^2+ (-1)^k  hD_{x_d}   \op_{sc}(b) \\
& \quad  -  (-1)^k  \op_{sc}(b)  hD_{x_d}  - \op_{sc}(b)^2  \\
 &=  h^2 D_{x_d}^2  +  \op_{sc}(\chi_1^2(R(x,\xi')-1))+ ±h\op_{sc}(r_1),
\end{align*}
as $  \op_{sc}(b)^2 =\op_{sc} (b^2)+h\op_{sc} (\tilde r_1)$ and exact symbol calculus, 
we have $r_1=(-1)^k D_{x_d}b+ \tilde r_1$. 
As $\chi_1$ is compactly supported, we have $r_1, \tilde r_1\in S^0_{\rm tan}$. Then 
$\| \op_{sc}(r_1)w_h\|_{L^2(\Omega)}\le  C_\eps \| \op_{sc}(\chi )  v_h\|_{L^2(\Omega)}\le C_\eps $.  
With~\eqref{eq: localization symbol Hyp case}, this proves the lemma.
\end{proof}
Let $z_k^h= \big( hD_{x_d}  -  (-1)^k \op_{sc}(b)   \big)  w_h$, by Lemma~\ref{lem: factorization hyperbolic},
we have, for $k=1,2$,  
$$
  \big( hD_{x_d}  + (-1)^k \op_{sc}(b)   \big) z_k^h=h\tilde q_k^h,
$$
 where
$\tilde q_1^h=  q_3^h$ and $\tilde q_2^h=  q_4^h$. By Proposition~\ref{prop: equation on v spectrally localised} and
Lemma~\ref{lem: factorization hyperbolic},
we have $\| z_k^h \|_{L^2(\Omega)}\le C_\eps $ and  $\| \tilde q_k^h\|_{L^2(\Omega)}\le 
C_\eps $ for $k=1,2$.
%%%%%%%%%%%%%%%%%%%%
%
%  Lemma
%
%%%%%%%%%%%%%%%%%%%%
\begin{lemma} \label{lem: trace estimation hyperbolic}
There exists $C_\eps>0$ such that  $ |  (z_k^h)_{|x_d=0}|_{L^2}\le C_\eps$.
\end{lemma}
\begin{proof}
We have
\begin{align*}
h\d_{x_d}\int_{\R^{d-1}} | z_k^h(x',x_d) |^2dx' &= 2\Re\int_{\R^{d-1}} ihD_{x_d }
\big(  z_k^h(x',x_d) \big) \overline{ z_k^h(x',x_d)}dx'  \\
&=  2\Re\int_{\R^{d-1}} i  (-1)^{k+1}\op_{sc}(b) \big(  z_k^h(x',x_d) \big) \overline{ z_k^h(x',x_d)}dx' \\
&\quad +  2\Re\int_{\R^{d-1}} ih\tilde q_k^h \overline{ z_k^h(x',x_d)}dx' .
\end{align*}
Integrating with respect to $x_d$ from $0$ to $\infty$, we obtain
\begin{equation} \label{eq: estimation energy hyperbolic}
h\int_{\R^{d-1}} | z_k^h(x',0) |^2dx'\lesssim \big| \Re ( i  \op_{sc}(b) z_k^h|  z_k^h)_{L^2(  \Omega)} \big| 
+ h \|  \tilde q_k^h \|_{L^2(\Omega)} \|  z_k^h \|_{L^2(\Omega)}.
\end{equation}
As $ ( i \op_{sc}(b) z_k^h|  z_k^h)_{L^2(\Omega)} = (  z_k^h| -i \op_{sc}(b)^\ast  z_k^h)_{L^2(\Omega)} $  and  
$\op_{sc}(b)^\ast =\op_{sc}(b) +h\op_{sc}(r_0)$ where $r_0\in S^0_{\rm tan}$, we have 
$ \big|  2\Re ( i  \op_{sc}(b) z_k^h|  z_k^h)_{L^2(\Omega)} \big|  \le C_\eps  h \|  z_k^h \|_{L^2(\Omega)}^2$. 
Then \eqref{eq: estimation energy hyperbolic} implies Lemma~\ref{lem: trace estimation hyperbolic}.
\end{proof}
The definition of $z_k^h$,  implies
\begin{align} \label{estimations comb lin deux traces}
& (hD_{x_d} w_h )_{|x_d=0} + \op_{sc}(b_0)  ( w_h )_{|x_d=0}= (z_1^h )_{|x_d=0} \notag \\
& (hD_{x_d} w_h )_{|x_d=0} - \op_{sc}(b_0)  ( w_h )_{|x_d=0}=( z_2^h  )_{|x_d=0},
\end{align}
where $b_0=b_{|x_d=0} $. Then we have 
$$
2  (hD_{x_d} w_h )_{|x_d=0} = (z_1^h )_{|x_d=0}+(z_2^h )_{|x_d=0}.
$$
As
 $$
  (hD_{x_d} \op_{sc} (\chi)v_h)_{|x_d=0} =
(\op_{sc} (\chi)_{|x_d=0}(hD_{x_d}v_h)_{|x_d=0} +h \op_{sc} (D_{x_d} \chi))_{|x_d=0}(v_h)_{|x_d=0},
$$
 and from~\eqref{estimation apriori trace} we deduce that  $  |   (hD_{x_d} w_h )_{|x_d=0} |_{L^2}  
 \le C_\eps  $ by 
Lemma~\ref{lem: trace estimation hyperbolic}.
This gives the second 
 estimate of~\eqref{eq: normal derivative estimation, hyperbolic}.
 
 From~\eqref{estimations comb lin deux traces} we also have 
 \begin{equation} \label{estimation trace hyperbolic b0z}
   |  \op_{sc}(b_0)  ( w_h )_{|x_d=0} |_{L^2}  \le C_\eps  .
 \end{equation}
 Let $\chi_2\in \Con_0^\infty(\R^d\times \R^{d-1})$, be such that $ (\chi_2)_{|x_d=0}=1$, on the support 
 of $\chi_{|x_d=0}$ and  $\supp \chi_2\subset \{ \chi_1=1 \}$.
By symbol calculus we have 
$$
 \op_{sc}\big(\chi_2(1-R(x,\xi'))^{-1/2} _{|x_d=0} \big) \op_{sc}(b_0)
= \op_{sc}(\chi_2)_{|x_d=0}+h\op_{sc}(r_0), 
$$
where $r_0\in S^{-1}_{\rm tan}$,
 and 
$$
\op_{sc} (\chi_2)_{|x_d=0}   \op_{sc} (\chi_{|x_d=0}) =  \op_{sc} (\chi_{|x_d=0}) +h \op_{sc} (r_1) ,
$$
where $ r_1\in S^{-1}_{\rm tan}$. We can write
\begin{align*}
(w_h)_{|x_d=0} &=\op_{sc}(\chi)_{|x_d=0}  (v_h)_{|x_d=0} =\op_{sc}(\chi_2)_{|x_d=0} (w_h)_{|x_d=0} 
 -h\op_{sc}(r_1)(w_h)_{|x_d=0}  \\
&  = \op_{sc}\big(\chi_2(1-R(x,\xi'))^{-1/2} _{|x_d=0} \big) \op_{sc}(b_0) (w_h)_{|x_d=0} -h \op_{sc} (r_0)(w_h)_{|x_d=0}\\
 & \quad -h\op_{sc}(r_1)(w_h)_{|x_d=0}.
\end{align*}
As $\chi_2 $ is compactly supported, $\op_{sc}(\est{\xi'}{} ) \op_{sc}\big(\chi_2(1-R(x,\xi'))^{-1/2} _{|x_d=0} \big)$ 
has a symbol 
in $S^0_{\rm tan}$ and  $ \op_{sc}(\est{\xi'}{} ) \op_{sc} (r_k) $, $k=0,1$ have symbols in $S^0_{\rm tan}$, then 
$$
|  \op_{sc}( \chi_{|x_d=0}) (v_h ) _{|x_d=0}  |_{H^1_{sc}}\le C_\eps |  \op_{sc}(b_0) (w_h)_{|x_d=0}  |_{L^2}
+h|  \op_{sc}( \chi_{|x_d=0}) (v_h ) _{|x_d=0}   |_{L^2}\le C_\eps,
$$
applying~\eqref{estimation trace hyperbolic b0z} and~\eqref{estimation apriori trace}. 
This achieves the proof of Proposition~\ref{prop: estimation traces hyperbolic}.
\end{proof}

%%%%%%%%%%%%%%%
%
% Subsubsection
%
%%%%%%%%%%%%%%%%%%
\subsubsection{Glancing points}
	\label{Sec: Glancing points}
Let $\chi=\chi_\eps\in\Con^\infty(\R^{2d})$ be such that
\begin{align}
	\label{eq: cutoff glancing region}
\chi(x,\xi')=
\begin{cases}
1 \text{ if } |R(x,\xi')-1|\le 2 \eps \text{ and } x_d\le 2\delta  \\
0 \text{ if } |R(x,\xi')-1|\ge 3\eps \text{ or  } x_d\ge 3\delta  ,
\end{cases}
\end{align}
and $0\le \chi \le 1$. Observe that $\chi\in S^{-\infty}_{\rm tan}$ as supported for $\xi'$ in a compact set, as $R(x,\xi')$ is bounded.
Let $w_h=\op_{sc} (\chi)v_h$. We have the following estimate on the traces on $w_h$.
\begin{proposition} \label{prop: estimation traces glancing}
There exists $C>0$ such that, for any $\eps>0$,
\begin{align}
& |  \op_{sc}( \chi_{|x_d=0}) (v_h ) _{|x_d=0}  |_{H^1_{sc}}\le C \eps^{1/4} h^{-1/2}+C_\eps h^{-3/8}
\label{est prop: first trace glancing case}  \\
&   |  \op_{sc}( \chi_{|x_d=0}) ( h D_{x_d}v_h  ) _{|x_d=0}  |_{L^2}\le C \eps^{3/4} h^{-1/2}
+C_\eps h^{-3/8}   ,  
	\label{eq: normal derivative estimation, glancing}
\end{align}
for $C_\eps>0$ .
\end{proposition}
\begin{remark}
Compared with the estimates stated in Proposition~\ref{prop: a priori estimates on traces}, 
we have the same power of
$h$ but with a power of $\eps$ in front of $h^{-1/2}$. The term $C_\eps h^{-3/8} = C_\eps h^{-1/2}h^{1/4}$ is a 
remainder. This gives 
a gain in this microlocal  region.
\end{remark}
\begin{proof}
In this proof $C$ is a constant independent of $\eps$ and we denote by $C_\eps$ a constant depending on
$\eps$.
Let 
$\chi_1\in\Con^\infty(\R^{2d})$ be such that 
$$
\chi_1(x,\xi')=
\begin{cases}
1 \text{ if } |R(x,\xi')-1|\le \eps \text{ and } x_d\le \delta  \\
0 \text{ if } |R(x,\xi')-1|\ge -\eps / 2\text{ or  } x_d\ge \delta  ,
\end{cases}
$$
and $0\le \chi_1\le 1$. The symbol $\chi_1\in S^{-\infty}_{\rm tan}$ since it is supported in $| \xi'| \le 2$. 
We have by Lemma~\ref{lem: localization in each region}     
\begin{align*}   
(h^2D_{x_d}^2+\op_{sc}(\chi_1^2(R(x,\xi')-1)))w_h=hq_2^h , \text{ where } \|q_2^h  \|_{L^2(\Omega)}\le C_\eps,
\end{align*}
for some $C_\eps>0$.

Then we have  
\begin{equation}  \label{eq: estimation second derivative glancing points}
\int_{\R^{d}_+} | h^2D^2_{x_d}w_h |^2dx\le 2  \int_{\R^{d}_+} | \op_{sc}(\chi_1^2(R(x,\xi')-1))w_h |^2dx  
 +  2\int_{\R^{d}_+} | hq_2^h  |^2dx.
\end{equation}
First, by symbol calculus, we have
\begin{align*}
\op_{sc}(\chi_1^2(R(x,\xi')-1))w_h&=\op_{sc}(\chi_1^2(R(x,\xi')-1)) \op_{sc} (\chi)v_h \notag \\
&= 
\op_{sc}(\chi(R(x,\xi')-1))v_h +h \op_{sc}(r_0) v_h,
\end{align*}
where $r_0\in S^0_{\rm tan}$. 
 Observe  that the semi-norms of $r_0$ depend on $\eps$.
This gives
\begin{equation}   \label{eq: estimate from  wh to vh}
\| \op_{sc}(\chi_1^2(R(x,\xi')-1))w_h  \|_{L^2(\Omega)} \le 
\|  \op_{sc}(\chi(R(x,\xi')-1))v_h   \|_{L^2(\Omega)}   +  C_\eps  h   \| v_h \|_{L^2(\Omega)} .
\end{equation}
On the support of $\chi$, we have $|  R(x,\xi')-1 | \le 2\eps$ then 
$4\eps^2-\chi^2(x,\xi')( R(x,\xi')-1 )^2\ge 0$ and $\chi^2(x,\xi')( R(x,\xi')-1 )^2\in S^0_{\rm tan}$ 
as $\chi$ is compactly supported.  By 
G\aa rding inequality \eqref{eq: Garding inequality} and by symbol calculus,
we have  
\begin{equation*}
4\eps^2\|  v_h   \|_{L^2(x_d>0)} ^2  - \|   \op_{sc}(\chi(R(x,\xi')-1))v_h     \|_{L^2(x_d>0)} ^2 
 \ge -C_\eps h \|     v_h \|_{L^2(x_d>0)} ^2  .
\end{equation*}
We deduce from this equation, \eqref{eq: estimation second derivative glancing points} 
and \eqref{eq: estimate from  wh to vh}
\begin{align*}
\| h^2 D_{x_d}^2  w_h   \|_{L^2(x_d>0)} \le (C\eps  + C_\eps h^{1/2})\|     v_h \|_{L^2(x_d>0)} 
+C h \|    q^h_2 \|_{L^2(x_d>0)}.
\end{align*}
By estimates on $v_h$ and $q^h_2$, we obtain
\begin{align} \label{est: L2 norm seconde derivative glancing case}
\| h^2 D_{x_d}^2  w_h   \|_{L^2(x_d>0)} \le C \eps+C_\eps h^{1/2}.
\end{align}
%%%%%%%%%%%%%%%%%%%%
%
%  Lemma
%
%%%%%%%%%%%%%%%%%%%%
\begin{lemma} 
	\label{lem: estimation trace by interpolation formula}
Let $g\in L^2(x_d>0)$   be supported in $\R^{d-1}\times [0, 1] $. We assume   
$D_{x_d}g\in L^2(x_d>0)$ then
\begin{align} \label{eq: estimation trace by interpolation formula}
 h| g_{|x_d=0} |_{L^2}^2\le 2\| hD_{x_d}  g \|_{L^2(x_d>0)}\|  g \|_{L^2(x_d>0)}.
\end{align}
\end{lemma}
\begin{proof}
Since
\begin{align*}
h\int_{\R^{d-1}}| g(x',0 )|^2 dx' &= - i \int_{\R^{d-1}}  \int_0^\infty hD_{x_d} | g(x',x_d) |^2 dx'dx_d \\
&=-2i \Re \int_{\R^{d-1}}  \int_0^\infty hD_{x_d}  g(x',x_d ) \overline{g(x',x_d )}dx'dx_d ,
\end{align*}
we obtain the lemma by Cauchy-Schwarz inequality.
\end{proof}
%%%%%%%%%%%%%%%%%%%%
%
%  Lemma
%
%%%%%%%%%%%%%%%%%%%%
\begin{lemma} 
	\label{lemma : interpolation first derivative}
There exists $C>0$, such that for all $g\in L^2(x_d>0)$   be supported in $\R^{d-1}\times [0, 1] $. Moreover we assume  $D_{x_d}g\in L^2(x_d>0)$ and $D_{x_d}^2g\in L^2(x_d>0)$, we have
\begin{align}  \label{estimation : interpolation first derivative}
\| hD_{x_d}g\|^2_{L^2(x_d>0)}\le C \| h^2D_{x_d}^2g\|_{L^2(x_d>0)}   \| g\|_{L^2(x_d>0)}
\end{align}
\end{lemma}
\begin{proof}
We have 
\begin{align*}
 \int_{\R^{d-1}}  \int_0^\infty |  hD_{x_d}g(x',x_d) |^2 dx'dx_d &= 
  \int_{\R^{d-1}}  \int_0^\infty hD_{x_d}\big( hD_{x_d} g(x',x_d)  \overline{ g(x',x_d) }  \big)  dx'dx_d \\
  &\quad  -   \int_{\R^{d-1}}  \int_0^\infty h^2D_{x_d}^2 g(x',x_d)  \overline{ g(x',x_d) }   dx'dx_d ,
\end{align*}
we yields
\begin{align}\label{eq: interpolation 1}
\| hD_{x_d}g\|^2_{L^2(x_d>0)}\le h|  hD_{x_d} g(x',0) |_ {L^2} |  g(x',0)  |_ {L^2} 
+ \| h^2D_{x_d}^2g\|_{L^2(x_d>0)}\| g\|_{L^2(x_d>0)}.
\end{align}
As $g\in H^2(x_d>0)$, we can apply Lemma~\ref{lem: estimation trace by interpolation formula} to $hD_{x_d}g$ to obtain
\begin{align} \label{eq: estimation first trace}
 h| (hD_{x_d}  g)_{|x_d=0}  |_{L^2}^2\le 2\| h^2D_{x_d}^2  g \|_{L^2(x_d>0)}\|   hD_{x_d} g \|_{L^2(x_d>0)}.
\end{align}
This estimate and~\eqref{eq: estimation trace by interpolation formula} yield 
\begin{align*}
 h| g_{|x_d=0} |_{L^2}| (hD_{x_d}  g)_{|x_d=0}  |_{L^2}\le 2\| h^2D_{x_d}^2  g \|_{L^2(x_d>0)}^{1/2}
 \|   hD_{x_d} g \|_{L^2(x_d>0)}  \|  g \|_{L^2(x_d>0)} ^{1/2} .
\end{align*}
From this estimate and~\eqref{eq: interpolation 1} we obtain~\eqref{estimation : interpolation first derivative}.
\end{proof}
%%%%%%%%%%%%%%%%%%%%
%
%  Lemma
%
%%%%%%%%%%%%%%%%%%%%
\begin{lemma}
There exists $C>0$ such that, for any $g\in H^2(x_d>0)$ supported in $\R^{d-1}\times [0, 1] $, we have
\begin{align} 
&h| g_{|x_d=0} |_{L^2}^2  \le C \| h^2D_{x_d}^2g\|_{L^2(x_d>0)}  ^{1/2} \| g\|_{L^2(x_d>0)}  ^{3/2}
 \label{est: trace by seconde derivative}\\
&h| (hD_{x_d}  g)_{|x_d=0}  |_{L^2}^2    \le C \| h^2D_{x_d}^2g\|_{L^2(x_d>0)} ^{3/2}  \| g\|_{L^2(x_d>0)} ^{1/2}
 \label{est:  seconde trace by seconde derivative}
\end{align}
\end{lemma}
\begin{proof}
The first estimate is obtained from~\eqref{eq: estimation trace by interpolation formula} 
and~\eqref{estimation : interpolation first derivative}. The second estimate is obtained 
from~\eqref{eq: estimation first trace} and~\eqref{estimation : interpolation first derivative}.
\end{proof}
Before applying the previous lemma to $w_h$, we have to estimate uniformly this function. As $\chi$ depend
on $\eps$ the norm of $\op_{sc}(\chi) $  as an operator on $L^2$ depends on $\eps$. Nevertheless applying 
\eqref{eq: estimation L2 sharp with Garding} a consequence of G\aa rding inequality and as $|\chi|\le 1$ we have
\begin{align*}
\|  w_h \|_{L^2(x_d>0)} =\| \op_{sc}(\chi) v_h  \|_{L^2(x_d>0)} \le C \|  v_h \|_{L^2(x_d>0)}  
+ C_\eps h^{1/2} \| v_h   \|_{L^2(x_d>0)}\le C+C_\eps h^{1/2}.
\end{align*}
From~\eqref{est: L2 norm seconde derivative glancing case} 
and~\eqref{est: trace by seconde derivative}, we  moreover have
\begin{align*}
h|  (w_h)_{|x_d=0}  |_{L^2}^2    \le (C \eps+C_\eps h^{1/2})^{1/2}  
 ( C+C_\eps h^{1/2} ) ^{3/2}\le C\eps^{1/2}+C_\eps h^{1/4},
\end{align*}
which gives~\eqref{est prop: first trace glancing case}.

From~\eqref{est: L2 norm seconde derivative glancing case} 
and~\eqref{est:  seconde trace by seconde derivative}, we also have 
\begin{align*}
h|  (hD_{x_d}w_h)_{|x_d=0}  |_{L^2}^2    \le (C \eps+C_\eps h^{1/2})^{3/2}  
 ( C+C_\eps h^{1/2} ) ^{1/2}\le C\eps^{3/2}+C_\eps h^{1/4},
\end{align*}
which gives~\eqref{eq: normal derivative estimation, glancing} as 
\begin{align*}
(hD_{x_d}\op_{sc}(\chi)v_h)_{|x_d=0}=
\op_{sc}(\chi_{|x_d=0})(hD_{x_d}v_h)_{|x_d=0}+ h\op_{sc}(hD_{x_d}\chi) _{|x_d=0}(v_h)_{|x_d=0},
\end{align*}
and the last term can be estimated by~\eqref{estimation apriori trace}.
\end{proof}
%%%%%%%%%%%%%%%
%
% Subsubsection
%
%%%%%%%%%%%%%%%%%%
\subsubsection{Elliptic points} 
	\label{subsubsection: elliptic points}
We start with the notation introduced in 
Proposition~\ref{prop: equation on v spectrally localised}. 
Let $\chi=\chi_\eps\in\Con^\infty(\R^{2d-1})$, be such that
\begin{align}
	\label{eq: cutoff elliptic region}
\chi(x,\xi')=
\begin{cases}
1 \text{ if } R(x,\xi')-1\ge 2\eps \text{ and } x_d\le \delta  \\
0 \text{ if } R(x,\xi')-1  \le  \eps \text{ or  } x_d\ge 2\delta  ,
\end{cases}
\end{align}
and $0\le \chi \le 1$. We have  $\chi\in S^{0}_{\rm tan}$. In this region the support of $\chi$ is not bounded, we have to 
take care of the symbol classes we use.

Let $\rho(x,\xi')=\big(  R(x,\xi')-1 \big)^{1/2}$ if $(x, \xi')$ satisfies $ R(x,\xi')-1>0$. Let $\chi_1(x,\xi')\in S^0_{\rm tan}$ be such that, $\chi_1=1$ on the 
support of $\chi$, and $\supp  \chi_1\subset \{ R(x,\xi')-1\ge \eps/2 \}\cup \{ x_d\le 3\delta \}$. 
Observe that $\chi_1\rho\in  S^{1}_{\rm tan}$.
We have
\begin{align*}
\big(hD_{x_d}  +  i\op_{sc}(\chi_1\rho) \big)  \big(hD_{x_d}  -  i\op_{sc}(\chi_1\rho) \big)
&= h^2D_{x_d}^2-  i hD_{x_d} \op_{sc}(\chi_1\rho)  +i   \op_{sc}(\chi_1\rho) hD_{x_d}   \\
&\quad +\op_{sc}(\chi_1\rho)^2 \\
&=   h^2D_{x_d}^2-  i [hD_{x_d} ,\op_{sc}(\chi_1\rho)  ] +\op_{sc}(\chi_1\rho)^2,
\end{align*}
and
\begin{align*}
& [hD_{x_d} ,\op_{sc}(\chi_1\rho)  ] = h\op_{sc}\big( D_{x_d} (\chi_1\rho)  \big), \text{  where } 
 D_{x_d} (\chi_1\rho) \in S^1_{\rm tan}, \\
& \op_{sc}(\chi_1\rho)^2 =  \op_{sc}\big( \chi_1^2( R(x,\xi')-1)  \big) +  h\op_{sc}(r_1),  \text{  where } 
 r_1\in  S^1_{\rm tan} .
\end{align*}
 Then 
\begin{align*}
\big(hD_{x_d}  +  i\op_{sc}(\chi_1\rho) \big)  \big(hD_{x_d}  -  i\op_{sc}(\chi_1\rho) \big)
&= h^2D_{x_d}^2 +    \op_{sc}\big( \chi_1^2( R(x,\xi')-1)  \big)  +h \op_{sc}(\tilde r_1)    ,%\\
\end{align*}
where $\tilde r_1\in S^1_{\rm tan}$.

Applying Lemma~\ref{lem: localization in each region} we obtain
\begin{align}
	\label{eq: factorization elliptic points localized}
\big(hD_{x_d}  +  i\op_{sc}(\chi_1\rho) \big)  \big(hD_{x_d} 
 -  i\op_{sc}(\chi_1\rho) \big) \op_{sc}(\chi) v_h=hq_2^h,
\end{align}
where $  \|    q_2^h    \|_{L^2(x_d>0)}\le C_\eps$.

Let $z= \big(hD_{x_d}  -  i\op_{sc}(\chi_1\rho) \big) \op_{sc}(\chi) v_h$; 
$z$ depends on $h$ but for the
sake of simplicity we prefer to denote it $z$. 
From \eqref{eq: factorization elliptic points localized} we have 
\[
\big(hD_{x_d}  +  i\op_{sc}(\chi_1\rho) \big) z=hq_2^h,
\]
in $x_d>0$.
We then have 
\begin{equation}
	\label{eq: ellipt. int. parts}
2\Re\big(  (hD_{x_d}  +i   \op_{sc}(\chi_1\rho)   ) z|i \op_{sc}(\chi_1\rho)  z\big)_{L^2(x_d>0)}
\lesssim h\| q_2^h\|_{{L^2(x_d>0)}} \| \op_{sc}(\chi_1\rho)  z \|_{{L^2(x_d>0)}} .
\end{equation}
We recall the integration  by parts formula in semiclassical context,
\begin{equation}
	\label{form: integ. parts}
(u|h D_{x_d}w)_{L^2(x_d>0)}=(h D_{x_d}u| w)_{L^2(x_d>0)} -ih( u_{|x_d=0} |  w_{|x_d=0} )_0,
\end{equation}
for $u$ and $w$ sufficiently smooth.
Taking $w=i   \op_{sc}(\chi_1\rho) z $ and $u=z$ we have
\[
(z|ih D_{x_d}   \op_{sc}(\chi_1\rho) z)_{L^2(x_d>0)}=(h D_{x_d}z| i   \op_{sc}(\chi_1\rho) 
z)_{L^2(x_d>0)} -ih( z_{|x_d=0} |  i   \op_{sc}(\chi_1\rho) z_{|x_d=0} )_0.
\]
As
\(
ih D_{x_d}   \op_{sc}(\chi_1\rho) =  i  \op_{sc}(\chi_1\rho) h D_{x_d}  
+ h\op_{sc}(\d_{x_d}(\chi_1\rho) ), 
\)
and 
\(
 \op_{sc}(\chi_1\rho)= \op_{sc}(\chi_1\rho)^*+ h  \op_{sc}(r_0),
\)
where $r_0\in S^0_{\rm tan}$ we obtain
\begin{align}
	\label{est: ellipt. estimation}
2\Re (  hD_{x_d}  z| i \op_{sc}(\chi_1\rho)  z)_{L^2(x_d>0)}
= h\Re(  z_{|x_d=0} |    \op_{sc}(\chi_1\rho) z_{|x_d=0} )_0 +hK,
\end{align}
where 
\[
|K|\lesssim  
  \| hDz\|_{L^2(x_d>0)}  \| z\|_{L^2(x_d>0)} +  \| z\|_{L^2(x_d>0)} ^2.
\]
From equation satisfied by $z$, we have
\begin{align*}
 \|hD_{x_d}z \|_{L^2(x_d>0)} 
& \lesssim \|   \op_{sc}(\chi_1\rho)  z\|_{L^2(x_d>0)}+ h\| q_2^h\|_{{L^2(x_d>0)}}\\
 &\lesssim    \| \op_{sc}(\est{\xi'})z\|_{L^2(x_d>0)}
  + h\| q_2^h\|_{{L^2(x_d>0)}} .
\end{align*}
Then 
\begin{align}
	\label{eq: ellipt. int. reste}
|K|\lesssim
  \| \op_{sc}(\est{\xi'})z\|_{L^2(x_d>0)}\| z\|_{L^2(x_d>0)}
  + h^2\| q_2^h\|_{{L^2(x_d>0)}}^2.
\end{align}
From \eqref{eq: ellipt. int. parts},  \eqref{est: ellipt. estimation} and \eqref{eq: ellipt. int. reste}
we yield
\begin{multline}
	\label{est: ellipt. boundary}
 \| \op_{sc}(\chi_1\rho)  z\|_{L^2(x_d>0)}^2
 + h\Re(  z_{|x_d=0} |    \op_{sc}(\chi_1\rho) z_{|x_d=0} )_0    \\
 \lesssim  
 h\| \op_{sc}(\est{\xi'})z\|_{L^2(x_d>0)}
\| z\|_{L^2(x_d>0)}
+ h^2\| q_2^h\|_{{L^2(x_d>0)}}^2,
\end{multline}
as we can estimate 
\[
 h\| q_2^h\|_{{L^2(x_d>0)}} \| \op_{sc}(\chi_1\rho)  z \|_{{L^2(x_d>0)}} 
 \le \alpha  \| \op_{sc}(\chi_1\rho)  z \|_{{L^2(x_d>0)}} ^2
 + C_\alpha h^2\| q_2^h\|_{{L^2(x_d>0)}}^2,
\]
and absorb the term 
\(
  \| \op_{sc}(\chi_1\rho)  z \|_{{L^2(x_d>0)}} ^2,
\)
by the left hand side of \eqref{est: ellipt. boundary} if $\alpha$ is sufficiently small.
%%%%%%%%%%%%%%%%%%%%%%%%%%%%
%
%   LEMMA
%
%%%%%%%%%%%%%%%%%%%%%%%%%%%%
\begin{lemma}
	\label{lem: est. ellip. H1 norm}
\[
 \| \op_{sc}(\est{\xi'})z\|_{L^2(x_d>0)}\lesssim   \| \op_{sc}(\chi_1\rho)  z \|_{{L^2(x_d>0)}} 
  + h\|hD v_h \|_{L^2(x_d>0)}+ h\| v_h \|_{L^2(x_d>0)}.
 \]
\end{lemma}
\begin{proof} 
Let $\chi_2(x,\xi')\in S^0_{\rm tan}$ be such that, $\chi_2=1$ on the 
support of $\chi_1$, and 
$\supp  \chi_1\subset \{ R(x,\xi')-1\ge \eps/4 \}\cup \{ x_d\le 4\delta \}$. 
We have by symbol calculus
\[
\op_{sc}(\chi_2\rho^{-1})\op_{sc}(\chi_1\rho) =\op_{sc}(\chi_1) + h\op_{sc}(r_1),
\]
where $r_1\in S^{-1}_{\rm tan}$.
Thus we obtain
\[
 \| \op_{sc}(\est{\xi'}) \op_{sc}(\chi_1) z\|_{L^2(x_d>0)}
 \lesssim
  \| \op_{sc}(\chi_1\rho)  z\|_{L^2(x_d>0)}
  +h   \|   z\|_{L^2(x_d>0)}.
\]
From this estimate we obtain
\begin{align}
	\label{est: ellipt H1 norm}
 \| \op_{sc}(\est{\xi'}) z\|_{L^2(x_d>0)}
& \lesssim 
  \| \op_{sc}(\est{\xi'}) \op_{sc}(1-\chi_1) z\|_{L^2(x_d>0)}  \notag \\
&\quad  +  \| \op_{sc}(\chi_1\rho)  z\|_{L^2(x_d>0)}
  +h   \|   z\|_{L^2(x_d>0)}.
\end{align}
From definition of $z$, we can write
\(
z= \op_{sc}(\chi) \big(hD_{x_d}  + \op_{ sc}  (r'_1)\big)v_h,
\)
where $r'_1\in  S^{1}_{\rm tan} $.
From symbol calculus and support properties of $\chi$ and $\chi_1$
the operator $ \op_{sc}(\est{\xi'}) \op_{sc}(1-\chi_1)\op_{sc}(\chi)$  is bounded on $L^2$ by 
$Ch$. From~\eqref{est: ellipt H1 norm} we thus have
\[
\| \op_{sc}(\est{\xi'}) z\|_{L^2(x_d>0)}
 \lesssim 
  h\|hD v_h \|_{L^2(x_d>0)}+  h\|v_h \|_{L^2(x_d>0)}
 +  \| \op_{sc}(\chi_1\rho)  z\|_{L^2(x_d>0)}
  +h   \|   z\|_{L^2(x_d>0)}.
\]
We obtain the statement as we can absorb $h   \|   z\|_{L^2(x_d>0)}$ by the left hand side.
\end{proof}    

From~\eqref{est: ellipt. boundary} and Lemma~\ref{lem: est. ellip. H1 norm} we deduce
\begin{multline*}
 \| \op_{sc}(\chi_1\rho)  z\|_{L^2(x_d>0)}^2
 + h\Re(  z_{|x_d=0} |    \op_{sc}(\chi_1\rho) z_{|x_d=0} )_0  
\\
 \lesssim
   h^2\|hD v_h \|_{L^2(x_d>0)}^2+ h^2\| v_h \|_{L^2(x_d>0)}^2
  + h^2\| q_2^h\|_{{L^2(x_d>0)}}^2,
\end{multline*}
as
\(
\| z\|_{L^2(x_d>0)}\le \| \op_{sc}(\est{\xi'})z\|_{L^2(x_d>0)}.
\)

From Lemma~\ref{lem: est. ellip. H1 norm} we obtain
\begin{equation}
	\label{est: elliptic case tangential derivative}
 \| \op_{sc}(\est{\xi'})  z\|_{L^2(x_d>0)} 
 \lesssim
   h\|hD v_h \|_{L^2(x_d>0)}^2+ h\| v_h \|_{L^2(x_d>0)}
  + h\| q_2^h\|_{{L^2(x_d>0)}}.
\end{equation}
From equation satisfied by $z$ we have
\[
\| h D_{x_d}z\|_{{L^2(x_d>0)}}
\lesssim  \| \op_{sc}(\est{\xi'})  z\|_{L^2(x_d>0)} + h\| q_2^h\|_{{L^2(x_d>0)}}.
\]
From this estimate, \eqref{est: elliptic case tangential derivative} and 
 trace formula~\eqref{eq: trace formula Hs}, we deduce
\begin{equation}
	\label{eq: estimation trace z}
	| z_{|x_d=0}|_{H^{1/2}_{sc}}\le C_\eps h^{1/2}.
\end{equation}
  From definition of $z$, we have, for $x_d>0$ and by symbol calculus,
  \begin{align*}  
  z= \op_{sc}(\chi) hD_{x_d} v_h-i  \op_{sc}(\chi \rho) v_h + hz_1,
  \end{align*}
  where  $ z_1= \op_{sc}( r_0) v_h$ and $r_0\in S_{\rm tan}^0$. In particular 
  \begin{align*}
   \|  z_1 \|_{L^2(x_d>0)} +  \| hD z_1  \|_{L^2(x_d>0)}\le  C_\eps .
  \end{align*}
   Let $u_0=h(D_{x_d}v_h)_{|x_d=0}$, $u_1=(v_h)_{|x_d=0}$, $\chi_0=\chi_{|x_d=0}$ 
   and $\rho_0=\rho_{|x_d=0}$. 
   From \eqref{eq: estimation trace z}, we have
   \begin{align*}
   \op_{sc}(\chi_0) u_0 -i\op_{sc}(\chi_0 \rho_0) u_1= h^{1/2} z_4, \text{ where } |z_4|_{L^2}\le C_\eps.
   \end{align*}
 Let $\Phi\in \Con^\infty(\d\Omega) $, we have 
  \begin{equation*}
   \Phi   \op_{sc}(\chi_0) =   \op_{sc}(\chi_0)\Phi +h  \op_{sc}(r_0) 
  \text{ and }  \Phi \op_{sc}(\chi_0 \rho_0)
  = \op_{sc}(\chi_0 \rho_0) \Phi+ h  \op_{sc}(\tilde r_0) , 
  \end{equation*}
   where $r_0, \tilde r_0\in S_{\rm tan}^0$ by symbol calculus.
  From Proposition~\ref{prop: a priori estimates on traces}, we have 
  \begin{equation}  \label{eq: trace relation elliptic case localized}
   \op_{sc}(\chi_0) \Phi u_0 -i\op_{sc}(\chi_0 \rho_0) \Phi  u_1= h^{1/2} z_5, \text{ where } |z_5|_{L^2}\le C_\eps.
  \end{equation}  
  With this equation we can obtain trace estimates into 
  $\d\Omega_D$ and $\d\Omega_N$.
 %%%%%%%%%%%%%%%%%%%%%%%%
 %
 %   Proposition
 %
 %%%%%%%%%%%%%%%%%%%%%%%%%   
\begin{proposition} 
   	\label{proposition: estimation elliptique Dirichlet Neumann}
  Let $\Phi\in \Con^\infty(\d\Omega) $. 
  
  If  $\Phi$ is supported on $\d\Omega_D$. Then 
  \begin{equation*}
  |  \op_{sc}(\chi_0) \Phi u_0 |_{L^2}\le C_\eps h^{1/2}.
  \end{equation*}
  
 If $\Phi$ is supported on $\d\Omega_N$. Then 
  \begin{equation*}
  |  \op_{sc}( \chi_0) \Phi u_1 |_{H^1_{sc}}\le C_\eps h^{1/2}.
  \end{equation*}
  \end{proposition}
  \begin{proof}
  If $\Phi$ is supported on $\d\Omega_D$, then $ \Phi  u_1=0$, \eqref{eq: trace relation elliptic case localized} 
  gives the first result.
   If $\Phi$ is supported on $\d\Omega_N$,  then $ \Phi  u_0=0$. Let 
   $\chi_1\in\Con^\infty (\R^{d-1}\times \R^{d-1})$ 
   be such that  $\chi_1=1$ on the support of $\chi_0$, $\supp \chi_1\subset\{  R(x',0,\xi')-1\ge \eps/2 \}$, $\chi_1
   \in S^0_{\rm tan}$, in particular we have $\chi_0\chi_1=\chi_0$ and $\rho_0\ne 0$ on support of $\chi_1$.
   We have
   $  \op_{sc} \est{\xi'}   \op_{sc} (\chi_1 \rho_0^{-1}) \op_{sc}(\chi_0 \rho_0) =  \op_{sc} \est{\xi'}   \op_{sc}(\chi_0 )
    +h\op_{sc}(r_0)$, where $r_0\in S^0_{\rm tan}$, by symbol calculus.  From 
    Proposition~\ref{prop: a priori estimates on traces} and \eqref{eq: trace relation elliptic case localized}, we have
    \begin{align*}
   |  \op_{sc}( \chi_0) \Phi u_1 |_{H^1_{sc}}   &\le   |  \op_{sc} \est{\xi'}   \op_{sc} 
   (\chi_1 \rho_0^{-1}) \op_{sc}(\chi_0 \rho_0) \Phi u_1  |_{L^2} + h| \op_{sc}(r_0) \Phi u_1 |_{L^2}  \\
 &\le h^{1/2} |  \op_{sc} \est{\xi'}   \op_{sc}   (\chi_1 \rho_0^{-1})   z_5 |_{L^2}  + h| \op_{sc}(r_0) \Phi u_1 |_{L^2} \\
 &\le C_\eps h^{1/2}.
       \end{align*}
       This gives the second estimates.
  \end{proof}
\begin{proposition}
	\label{prop: convergence traces D and N}
Let $\Phi\in \Con^\infty(\d\Omega) $. We assume that $\Phi$ is supported either on $\d\Omega_D$ or on  
$\d\Omega_N$.
\begin{align*}
&h^{1/2}|(v_h)_{|x_d=0}|_{H^{1}_{sc}}  \to0 \text{ as } h\to 0,  \\
&h^{1/2}| (hD_{x_d}v_h)_{|x_d=0}|_{L^2}  \to0 \text{ as } h\to 0.
\end{align*} 
\end{proposition}
\begin{proof}
Let $\eps>0$, we can find $\chi_H$, $\chi_G$ and $\chi_E$ satisfying respectively the assumption 
of Propositions \ref{prop: estimation traces hyperbolic}, \ref{prop: estimation traces glancing}, 
\ref{proposition: estimation elliptique Dirichlet Neumann} and furthermore the relation 
$\chi_H+\chi_G+\chi_E=1$. Applying the results of Propositions in each region, we deduce the proposition.
\end{proof}

In what follows we shall prove estimates in a \nhd of $\overline{\d\Omega_N}\cap\overline{ \d\Omega_D}$. The result is more delicate and less classical  than these obtained for Dirichlet or Neumann boundary condition.

Now we assume that $\Phi$ is supported in a neighborhood of a point of $\overline{\d\Omega_N}\cap\overline{ \d\Omega_D}$. 
We can assume that locally this set is given by $x_1=0$ and the support of 
$\Phi$ is contained into  
a fixed domain in $x'$ and into
 $\{ |x_1|\le \mu \eps/2 \}$ where $\mu>0$ will be fixed below sufficiently small.
 Here and in what follows  $\eps$ is the one used to define 
 hyperbolic, glancing, elliptic regions 
 (see respectively~\eqref{def: cut-off for hyperbolic region}, 
 \eqref{eq: cutoff glancing region}  and \eqref{eq: cutoff elliptic region}).
  We assume that $\supp u_0\subset\{ x_1\le 0 \}$ and $\supp u_1\subset\{ x_1\ge 0 \}$.
 We can choose the local coordinates such that 
 \begin{align}  
 	\label{def: de la notation R0 sur Gamma}
  R(x',0,\xi')=\xi_1^2+R_0(x'',\xi'')+x_1r_2(x',\xi'), \text{ where } x'=(x_1,x'')
\text{  and } \xi'=(\xi_1,\xi'') , 
 \end{align}
 $R_0\in S(\est{\xi''}^2,(dx'')^2+\est{\xi''}^{-2}(d\xi'')^2)$ and $r_2\in S^2_{\rm tan}$.
 Indeed in normal geodesic coordinates we have 
 \(
 R(x',0,\xi')=\xi_1^2+R_1(x',\xi'')=\xi_1^2+R_0(x'',\xi'')+x_1r_2(x',\xi')
 \)
 and in fact 
 \(
 r_2(x',\xi')
 \)
 does not depend on $\xi_1$ but we do not use this property in what follows.
 
 Let $\alpha(x'',\xi')=\big(  \xi_1^2+R_0(x'',\xi'')-1+i\eps\big)^{1/2}$, 
and let 
 \begin{equation} 
 	 \label{def: le beta}
  \beta(x'',\xi'')=\big( R_0(x'',\xi'')-1+i\eps  \big)^{1/2}, 
 \end{equation}
 be such that $\Im \beta(x'',\xi'')>0$, for all 
 $(x'',\xi'')\in \R^{d-2}\times\R^{d-2}$. We have 
 \begin{equation*}
  \alpha(x'',\xi')= \big(\xi_1  + i \beta(x'',\xi'')  \big)^{1/2}
 \big(\xi_1 -  i \beta(x'',\xi'')  \big)^{1/2},
 \end{equation*}
 Observe that $\Re \beta(x'',\xi'')>0$, we deduce that
 \begin{align*}
 &\xi_1\mapsto \big(\xi_1  + i \beta(x'',\xi'')  \big)^{\pm 1/2}  
 \text{ are  holomorphic functions in }  \{ \Im \xi_1>0 \},\\
  &\xi_1\mapsto \big(\xi_1  - i \beta(x'',\xi'')  \big)^{\pm 1/2}  
  \text{ are  holomorphic functions on }   \{ \Im \xi_1<0\}.
 \end{align*}
 Let $v_1=\op_{sc}\big(\xi_1  - i \beta(x'',\xi'')  \big)^{ 1/2}\Phi u_1$. The operator 
 $\op_{sc}\big(\xi_1  - i \beta(x'',\xi'')  \big)^{ 1/2}$ is a convolution operator 
 with respect $x_1$ and its kernel 
 is supported in $x_1\ge0$. As $u_1$ is supported in $x_1\ge0$, 
 this implies that $v_1$ is supported in $x_1\ge0$.
 Let $v_0=\op_{sc}(\xi_1+i\beta(x'', \xi''))^{-1/2}\Phi u_0$. As $u_0$ is supported in 
 $x_1\le 0$ and $(\xi_1+i\beta(x'', \xi''))^{-1/2}$ 
is a holomorphic function in $\{ \Im \xi_1>0\}$, $v_0$ is supported in $x_1\le 0$. 
We first prove the following lemma.
\begin{lemma}\label{lemma: estimation v0 v1}
There exists $C>0$, such that for all $\eps\in(0,1)$ we have, for every $h\in(0,1)$
\begin{align*}
&|v_1|_{L^2}  \le  C\eps^{1/5}h^{-1/2}+C_\eps h^{-3/8} ,\\
&|v_0|_{L^2} \le  C\eps^{1/5}h^{-1/2}+C_\eps h^{-3/8},
\end{align*}
where $C_\eps>0$  depends on $\eps$.
\end{lemma}
\begin{proof}
Observe that $(\xi_1-i\beta(x'', \xi''))^{\pm1/2}\in S(\est{\xi'}^{\pm 1/2}, (dx')^2+(d\xi')^2)$ then
 \[\op_{sc} (\xi_1-i\beta(x'', \xi''))^{-1/2}\op_{sc} (\xi_1-i\beta(x'', \xi''))^{1/2}= Id+ h\op_{sc}(r_0), \]
 where 
 $r_0\in  S(1, (dx')^2+(d\xi')^2)$. 
 From the definition of $v_1$ we have 
 \begin{equation}
 	\label{eq: link between v1 and u1}
  \op_{sc} (\xi_1-i\beta(x'', \xi''))^{-1/2}v_1=   \Phi u_1+ h\op_{sc}(r_0) u_1.  
 \end{equation}
 As $\op_{sc}(\chi_0 \rho_0) \op_{sc}(r_0)$ has a symbol in $S(\est{\xi'},(dx')^2+(d\xi')^2)$,
we obtain 
\[
 |  \op_{sc}(\chi_0 \rho_0) \op_{sc}(r_0) v_1|_{H_{sc}^{-1/2}}\le C_\eps | u_1 |_{H_{sc}^{1/2}}
\le C_\eps h^{-1/2}
\]
 by Proposition~\ref{prop: a priori estimates on traces}.
 Then
 from~\eqref{eq: trace relation elliptic case localized} we obtain 
   \begin{equation} \label{eq: relation between u0 and v_1}
     \op_{sc}(\chi_0) \Phi u_0 -i\op_{sc}(\chi_0 \rho_0) 
     \op_{sc}\big(\xi_1  - i \beta(x'', \xi'') \big)^{ -1/2} v_1
     = h^{1/2} z_6, \text{ where } |z_6|_{H_{sc}^{-1/2}}\le C_\eps.
   \end{equation}
   To determine $u_0$ and $u_1$ we have to use the support properties of these functions. 
   To do that, we have to modify
   the operators acting on these functions, 
   note  that $\op_{sc}(\chi_0 \rho_0) $ does \emph{not} preserve the
   support of $u_0$.
   
   We introduce three cutoff functions $\chi_H$, $\chi_G$ and $\chi_E$, be such that $\chi_H +  \chi_G+\chi_E=1$ and
   \begin{align*}
 &  \supp \chi_H\subset \{ R(x', 0,\xi')-1  \le -\eps/2  \}, \ 0\le \chi_ H\le 1  \\
 &  \supp \chi_G\subset \{ | R(x', 0,,\xi')-1 |\le   \eps \}, \ 0\le \chi_G \le 1  \\
 &  \supp \chi_E\subset \{ R(x',0,\xi')-1 \ge  \eps/2 \}, \ 0\le \chi_E \le 1  .
   \end{align*}
   We then have $\chi_H$,  $\chi_G$ and $\chi_E$, the $\chi_{|x_d=0}$, defined respectively in the 
   hyperbolic, glancing and elliptic regions (after multiplying  $\eps$ by a fix factor).
   We have $\Phi u_0=  \op_{sc}( \chi_H) \Phi u_0 +   \op_{sc}( \chi_G)\Phi u_0+  \op_{sc}( \chi_E)\Phi u_0$.
   From  hyperbolic estimate given in Proposition~\ref{prop: estimation traces hyperbolic}, 
   Proposition~\ref{prop: a priori estimates on traces} and symbol calculus, we have
   \begin{equation*}
   |  \op_{sc}( \chi_H)  (\Phi u_0 ) |_{L^2}\le C_\eps.
      \end{equation*}
   From  glancing estimate given in Proposition~\ref{prop: estimation traces glancing},  
   Proposition~\ref{prop: a priori estimates on traces} and symbol calculus, we have
   \begin{equation*}
     |  \op_{sc}( \chi_G)(  \Phi u_0 ) |_{L^2}\le C \eps^{3/4} h^{-1/2}
+C_\eps h^{-3/8} .
   \end{equation*}
   We deduce that
   \begin{equation}\label{eq: estimation difference u0 and cutoff elliptic zone}
      |  \Phi u_0 - \op_{sc}( \chi_E)(  \Phi u_0 ) |_{L^2}\le   C \eps^{3/4} h^{-1/2}
+C_\eps h^{-3/8} .
   \end{equation}
   \begin{lemma} \label{lem: estimation v0 hyperbolic and glancing}
With the previously defined notations, we have
   \begin{align}
  & | \op_{sc}(\chi_H\alpha) \op_{sc}\big(\xi_1  - i \beta(x'', \xi'') \big)^{ -1/2} v_1|_{H_{sc}^{-1/2}}\le C_\eps  
  \label{est: hyperbolic on v1}
   \\
&    |\op_{sc}(\chi_G \alpha) \op_{sc}\big(\xi_1  - i \beta(x'', \xi'') \big)^{ -1/2} v_1 |_{H_{sc}^{-1/2}}  \le 
    C \eps^{3/4} h^{-1/2}
+C_\eps h^{-3/8} ,
\label{est: glancing on v1}
   \end{align}
   where $\alpha(x'',\xi')=\big(  \xi_1^2+R_0(x'',\xi'')-1+i\eps\big)^{1/2}$.
     \end{lemma}
   \begin{proof}
   Let $\tilde\chi_H$ be such that $\tilde\chi_H=1$ on the support of $\chi_H$ and $\tilde\chi_H=0$ if 
   $R_0(x',\xi') - 1\ge -\eps/4$. 
    Let $\tilde\chi_G$ be such that $\tilde\chi_G=1$ on the support of $\chi_G$ and $\tilde\chi_H=0$ if 
   $|R_0(x',\xi') - 1|\ge 2\eps$. Let $J=G$ or $H$.
   
   By symbol calculus, we have
   \begin{equation*}
   \op_{sc}(\alpha \chi_J)=\op_{sc}(\alpha \chi_J)\op_{sc}(\tilde\chi_J)+h\op_{sc}(r_0),
   \end{equation*}
  where $r_0\in S^0_{\rm tan}$. From~\eqref{eq: link between v1 and u1}, we deduce
  \begin{align*}
 &  | \op_{sc}(\chi_J\alpha) \op_{sc}\big(\xi_1  - i \beta(x'', \xi'') \big)^{ -1/2} v_1|_{H_{sc}^{-1/2}}\\
 &\quad  \le 
      | \op_{sc}(\chi_J\alpha)\Phi u_1|_{H_{sc}^{-1/2}}   +   h  | \op_{sc}(\chi_J\alpha) \op_{sc}(r_0)u_1|_{H_{sc}^{-1/2}}  \\
&   \quad      \le  | \op_{sc}(\chi_J\alpha) \op_{sc}(\tilde\chi_J)  u_1|_{H_{sc}^{-1/2}} +
    C_\eps h   | u_1|_{H_{sc}^{1/2}} .
  \end{align*}
   If $J=H$ we have
      \begin{align*}
       | \op_{sc}(\chi_H\alpha) \op_{sc}\big(\xi_1  - i \beta(x'', \xi'') \big)^{ -1/2} v_1|_{H_{sc}^{-1/2}}&\le    
       C_\eps  |\op_{sc}(\tilde\chi_H)  u_1|_{H_{sc}^{1/2}} +
C_\eps     h   |  u_1|_{H_{sc}^{1/2}} .
      \end{align*}
we obtain~\eqref{est: hyperbolic on v1} from Propositions~\ref{prop: a priori estimates on traces} 
and \ref{prop: estimation traces hyperbolic}.
   
 If $J=G$, using that $|x_1|\le \mu\eps/2$, on the support of $\chi_G$, we have $|\alpha|\le C\eps^{1/2}$.
    As 
 \begin{align*}
       | \op_{sc}(\chi_J\alpha) \op_{sc}(\tilde\chi_J)  u_1|_{H_{sc}^{-1/2}}\le
        C | \op_{sc}(\chi_J\alpha) \op_{sc}(\tilde\chi_J)  u_1|_{L^2},
 \end{align*}
we apply  G\aa rding inequality~\eqref{eq: estimation L2 sharp with Garding boundary} and we obtain
       \begin{align*}
       | \op_{sc}(\chi_G\alpha) \op_{sc}\big(\xi_1  - i \beta(x'', \xi'') \big)^{ -1/2} v_1|_{H_{sc}^{-1/2}}&\le    
       C  \eps^{1/2}|\op_{sc}(\tilde\chi_G)  u_1|_{H_{sc}^{1/2}} +
C_\eps     h   |  u_1|_{H_{sc}^{1/2}} ,
      \end{align*}
  where, at the right hand side, we have estimated the $L^2$-norm by the $H^{1/2}$-norm. 
    We obtain~\eqref{est: glancing on v1}
     from    Propositions~\ref{prop: a priori estimates on traces} and \ref{prop: estimation traces glancing}.
   \end{proof}

Following  \eqref{eq: relation between u0 and v_1} (with $\chi_0=\chi_E$), \eqref{eq: estimation difference u0 and cutoff elliptic zone} and
Lemma~\ref{lem: estimation v0 hyperbolic and glancing}, we have 
\begin{multline*}
\Phi u_0-i \op_{sc}\big( \alpha(\chi_H+\chi_G) \big)\op_{sc}(\xi_1 -i\beta(x'', \xi''))^{-1/2}v_1\\
-i \op_{sc}(\chi_E\rho_0)
\op_{sc}(\xi_1-i\beta(x'', \xi''))^{-1/2}v_1=z_7,
\end{multline*}
where $| z_7 |_{H_{sc}^{-1/2}}\le C\eps^{3/4}h^{-1/2}+C_\eps h^{-3/8}$.
 Applying $\op_{sc}( \xi_1+i\beta(x'', \xi'') )^{-1/2}$ to this equation, we obtain 
  \begin{multline} 
  \label{eq: v0 v1 z6 sic}
v_0-i\op_{sc}(\xi_1+i\beta(x'', \xi''))^{-1/2} \op_{sc}\big( \alpha(\chi_H+\chi_G) \big)\op_{sc}(\xi_1-i\beta(x'', \xi''))^{-1/2}v_1  \\
-i\op_{sc}(\xi_1+i\beta(x'', \xi''))^{-1/2} \op_{sc}(\chi_E\rho_0)
\op_{sc}(\xi_1-i\beta(x'', \xi''))^{-1/2}v_1  \\
= \op_{sc}(\xi_1+i\beta(x'', \xi''))^{-1/2}z_7.
\end{multline}
We have to precisely  estimate $z_8= \op_{sc}(\xi_1+i\beta(x'', \xi''))^{-1/2}z_7$. By the definition of $\beta(x'', \xi'')$ 
there exists $C_0>0$
such that 
$\Re \beta(x'', \xi'')\ge C_0 \eps$, as $R_0(x'',\xi'')-1\ge -1$ moreover  if $|\xi''|\ge 2$, then 
$\Re \beta(x'',\xi'')\ge C_0\est{\xi''}$.
   We  deduce that $|\xi_1+i\beta (x'',\xi'') |\ge C_0\eps \est{\xi'}$. This implies that 
   $|(\xi_1+i\beta (x'',\xi''))^{-1/2} |\le C_0\eps ^{-1/2}\est{\xi'}^{-1/2}$, where $C_0>0$.
   We have by symbol calculus
   \begin{align*}
    | z_8|_{L^2}\le |  \op_{sc}\big( (\xi_1+i\beta(x'', \xi''))^{-1/2} \est{\xi'}^{1/2} \big) \op_{sc}(\est{\xi'}^{-1/2})z_7|_{L^2}
    + C_\eps  h|\op_{sc}(\est{\xi'}^{-1/2}) z_7|_{L^2},
   \end{align*}
   and by  G\aa rding inequality~\eqref{eq: estimation L2 sharp with Garding boundary} and the estimate on $z_7$, we have
   \begin{align*}
     | z_8|_{L^2}&\le C \eps^{-1/2} |  \op_{sc}(\est{\xi'}^{-1/2})z_7|_{L^2} 
     + C_\eps  h|\op_{sc}(\est{\xi'}^{-1/2}) z_7|_{L^2},\\
     &\le C\eps^{1/4}h^{-1/2}+C_\eps h^{-3/8}.
   \end{align*}
By symbol calculus, as $(\xi_1+i\beta(x'', \xi''))^{-1/2}  \alpha(\xi_1-i\beta(x'', \xi''))^{-1/2}=1$,    we have
   \begin{multline}
   	\label{eq: zone hyperbolic glancing v1}
 \op_{sc}(\xi_1+i\beta(x'', \xi''))^{-1/2} \op_{sc}\big( \alpha(\chi_H+\chi_G) \big)\op_{sc}(\xi_1-i\beta(x'', \xi''))^{-1/2}  \\
 = \op_{sc}(\chi_H+\chi_G) +h\op_{sc}(r_0),
    \end{multline}
     where $r_0\in S(1,(dx')^2+(d\xi')^2)$. 
     
  By the same argument, we have
  \begin{equation}\label{eq: difference form rho times inverse alpha}
  \op_{sc}(\xi_1+i\beta(x'', \xi''))^{-1/2} \op_{sc}(\chi_E\rho_0)
\op_{sc}(\xi_1-i\beta(x'', \xi''))^{-1/2}=\op_{sc}(\chi_E\rho_0 \alpha^{-1})+h\op_{sc}(r_0),
  \end{equation}
  where $r_0\in S(1,(dx')^2+(d\xi')^2)$. Indeed $(\xi_1\pm i\beta(x'', \xi''))^{-1/2} \in S(\est{\xi'}^{-1/2},(dx')^2+(d\xi')^2)$ and 
  $\chi_E\rho_0\in S(\est{\xi'},(dx')^2+(d\xi')^2)$.

The following lemma gives a precise estimate on $\chi_E(\rho_0\alpha^{-1}-1)$. We shall 
exploit that $\alpha$ and $\rho_0$ are close.
\begin{lemma}  \label{lem: difference rho alpha elliptic}
We have 
\begin{align*}
\big|    \op_{sc}  (\chi_E \rho_0\alpha^{-1} ) v_1 -      \op_{sc}  (\chi_E  ) v_1 \big|_{L^2}
\le C \eps^{1/5} h^{-1/2}+C_\eps h^{-3/8}.
\end{align*} 
\end{lemma}
\begin{proof}
Let $\tilde\Phi(x_1)$ supported in  $\{ |x_1|\le \mu \eps \}$ and $\tilde\Phi=1$ on the support of $\Phi$.
Let $b$ be either the symbol $\chi_E \rho\alpha^{-1} $ or $\chi_E $, we have $b\in S(1,(dx')^2+(d\xi')^2)$. 

By symbol calculus 
we have $\op_{sc}(b (1-\tilde\Phi  ))v_1= h\op(r)u_1$, where $r\in S(\est{\xi'}^{1/2},(dx')^2+(d\xi')^2)$.
By Proposition~\ref{prop: a priori estimates on traces}  we have 
$| \op_{sc}(b (1-\tilde\Phi  ))v_1|_{L^2}\le C_\eps h|  u_1|_{H^{1/2}_{sc}}\le C_\eps h^{1/2}$.
Then we can considerate   $\op_{sc}(\tilde\Phi\chi_E \rho\alpha^{-1}) $ and 
$\op_{sc}( \tilde\Phi  \chi_E )$, instead of respectively
 $\op_{sc}( \chi_E \rho\alpha^{-1}) $ and $\op_{sc}(\chi_E )$.
 
We introduce three cutoff functions $\chi_j\in \Con^\infty(\R^{d-2}\times\R^{d-1})$ be such that
\begin{align*}
& \chi_1=
\begin{cases}
 1 \text{ if }  \quad   \xi_1^2+R_0(x'',\xi'')-1 \le  \eps^{4/5} , \\
 0 \text{ if }   \quad \xi_1^2+R_0(x'',\xi'')-1  \ge 3 \eps^{4/5} ,
\end{cases} \\
& \chi_2=
\begin{cases}
 1 \text{ if }   \quad  1\ge \xi_1^2+R_0(x'',\xi'')-1 \ge 2  \eps^{4/5} ,\\
 0 \text{ if }   \quad  \xi_1^2+R_0(x'',\xi'')-1    \le   \eps^{4/5}  \text{ and } | \xi' |^2-1\ge 3,
\end{cases} \\
& \chi_3=
\begin{cases}
 1 \text{ if }   \quad  \xi_1^2+R_0(x'',\xi'')-1 \ge 2  ,\\
 0 \text{ if }  \quad   \xi_1^2+R_0(x'',\xi'')-1\le 1,
\end{cases}
\end{align*}  
$\chi_1 +\chi_2 +\chi_3 =1$ and $0\le \chi_j\le 1$ for $j=1,2,3$.

\medskip
{\bf Estimation on the support of $\chi_E\chi_1$.  } 

\smallskip

On the support of $\chi_E\chi_1$ we have $\xi_1^2+R_0(x'',\xi'')-1  \le 3 \eps^{4/5} $ and 
$R(x',0,\xi')-1 \ge \eps/2 $, in particular $|\xi'|$ is bounded.
We compute on this domain
\begin{align}\label{eq: difference between rho and alpha}
\rho_0(x',\xi')-\alpha(x'',\xi')= D(x',\xi')\big( x_1r_2(x',\xi')- i\eps \big),
\end{align} 
where $D(x',\xi')= \Big( (  R(x', 0,\xi') -1)^{1/2} + (\xi_1^2+ R_0(x'',\xi'') -1 +i\eps )^{1/2}  \Big)^{-1}$.
Observe that we have $\Re  (\xi_1^2+ R_0(x'',\xi'') -1 +i\eps )^{1/2} >0$ and 
$(  R(x', 0,\xi'') -1)^{1/2}\ge 2^{-1/2}\eps^{1/2}$, then $|D(x',\xi')| \le C_0 \eps^{-1/2}$.
We deduce that $|  \rho_0 -\alpha|\le C_0\eps^{1/2}$ if $|x_1|\le \mu \eps$ where $\mu $ was introduced 
in the definition of $\Phi$.
As 
\begin{equation*}
| \xi_1^2+R_0(x'',\xi'')-1-(R(x',0,\xi')-1) |\le C |x_1|\le C\mu \eps,
\end{equation*}
if $\mu$ is chosen sufficiently small, on the support of $\chi_E$, we have 
$\xi_1^2+R_0(x'',\xi'')-1\ge C_0\eps$, for $C_0>0$. We deduce that $|\alpha(x'',\xi')|\ge C_1\eps^{1/2}$, for $C_1>0$. 

This implies that
\begin{equation}
	\label{eq: estimation rho0 moins alpha}
|  (\rho_0-\alpha)\alpha^{-1}|\le C_2 \text{ on the support of } \chi_E\chi_1, \text{ for } C_2>0 \text{ and } |x_1|\le \mu \eps .
\end{equation}
Let 
\begin{equation*}
\tilde\chi_1(x'\xi')=
\begin{cases}
1 \text{ on the support of } \chi_1 , \\
0 \text{ if }\quad \xi_1^2+R_0(x'',\xi'')-1\ge 4\eps^{4/5} \text{ or }  \quad R(x',0,\xi')-1\le \eps/4.
\end{cases}
\end{equation*}
By symbol calculus in classes  $S(\est{\xi'}^s, (dx')^2+ (d\xi')^2)$,
we have
\begin{align*}
\big(  \op_{sc}  (\tilde\Phi  \chi_E\chi_1  \rho_0\alpha^{-1} )   -      \op_{sc}  (\tilde\Phi\chi_E\chi_1   )  \big)v_1=
 \op_{sc} (\tilde\Phi\chi_E\chi_1(  \rho_0-\alpha)\alpha^{-1} )  \op_{sc}  (\tilde\chi_1)v_1+h \op_{sc}  (r_0)v_1,
\end{align*}
where $r_0\in S(1, (dx')^2+ (d\xi')^2)$.
By G\aa rding inequality~\eqref{eq: estimation L2 sharp with Garding boundary}  and~\eqref{eq: estimation rho0 moins alpha} 
we have
\begin{align}\label{est: difference rho alpha zone un}
\big|  \big(  \op_{sc}  (\tilde\Phi  \chi_E\chi_1  \rho_0\alpha^{-1} )   -      \op_{sc}  (\tilde\Phi  \chi_E\chi_1   )  \big)v_1  \big|_{L^2}
\le C |  \op_{sc}  (\tilde\chi_1) v_1 |_{L^2}  +  C_\eps h | v_1  |_{L^2}.
\end{align}
Observe that
\begin{align*}
\op_{sc}  (\tilde\chi_1)v_1=\op_{sc}\big(\xi_1  - i \beta(x'',\xi'')  \big)^{ 1/2}\Phi\op_{sc}  
(\tilde\chi_1) u_1+ h\op_{sc}(r_{1/2})u_1,
\end{align*}
by semiclassical symbol calculus in $S(\est{\xi'}^s,(dx')^2+ (d\xi')^2)$, where $r_{1/2}\in  S(\est{\xi'}^{1/2},(dx')^2+ (d\xi')^2)$.
As $\tilde \chi_1$ is supported in $| R(x',0,\xi')-1 |\le C \eps^{4/5}$, we can apply 
Proposition~\ref{prop: estimation traces glancing} with $\eps^{4/5}$ instead of $\eps$. 
From~\eqref{est: difference rho alpha zone un}, Proposition~\ref{prop: a priori estimates on traces} and as 
 $ | v_1  |_{L^2} \le C_\eps  | u_1  |_{H^{1/2}_{sc}}$,   we have
\begin{align}  \label{est: zone 1 elliptic diference rho alpha}
\big|  \big(  \op_{sc}  (\chi_E\chi_1  \rho_0\alpha^{-1} )   -      \op_{sc}  (\chi_E\chi_1   )  \big)v_1  \big|_{H^{1/2}_{sc}}
& \le  C |  \op_{sc}  (\tilde\chi_1) u_1 |_{H^{1}_{sc}}  +  C_\eps h | u_1  |_{H^{1/2}_{sc}}.  \notag\\
& \le C\eps^{1/5}h^{-1/2}+C_\eps h^{-3/8} .
\end{align}

{\bf Estimation on  the support of $\chi_E\chi_2$. }

\smallskip
Equation~\eqref{eq: difference between rho and alpha} is valid in the supports of $\chi_E\chi_2$ and 
$ \tilde\Phi  $. As 
$ \xi_1^2+R_0(x'',\xi'')-1   \ge   \eps^{4/5} $, on this domain, we have $ \Re (\xi_1^2+ R_0(x'',\xi'') -1 
+i\eps )^{1/2}  \ge C_0\eps^{2/5}$, for $C_0>0$. We deduce  that $|D(x',\xi')|\le C\eps^{-2/5}$ and 
$|\alpha^{-1}(x'',\xi')|\le C\eps^{-2/5}$,
then $|(\rho_0-\alpha) \alpha^{-1}  | \le C\eps^{1/5}$.
We conclude by semiclassical symbol calculus in 
$S(\est{\xi'}^s, (dx')^2+(d\xi')^2)$, G\aa rding 
inequality~\eqref{eq: estimation L2 sharp with Garding boundary} and 
Proposition~\ref{prop: a priori estimates on traces} that
\begin{align}\label{est: zone 2 elliptic diference rho alpha}
&\big|    \op_{sc}  (\tilde\Phi  \chi_E\chi_2  \rho_0\alpha^{-1} )  -      \op_{sc}  (\tilde\Phi  \chi_E\chi_2   ) 
v_1 \big|_{L^2}    \notag  \\ 
& \quad = 
\big|    \op_{sc}  (\tilde\Phi  \chi_E\chi_2 ( \rho_0-\alpha)\alpha^{-1} )\op_{sc}\big(\xi_1  - i \beta(x'',\xi'')  \big)^{ 1/2}
\Phi u_1 \big|_{L^2} 
\le C \eps^{1/5} h^{-1/2}+C_\eps.
\end{align}

{\bf Estimation on the support of $\chi_E\chi_3$.}

On the supports of $\chi_E\chi_3$  and $ \tilde\Phi  $,
we have $ \Re (\xi_1^2+ R_0(x'',\xi'') -1 +i\eps )\ge C_0\est{\xi'}^2$ for $C_0>0$. 
We deduce that
$ \Re (\xi_1^2+ R_0(x'',\xi'') -1 +i\eps )^{1/2}\ge C_0\est{\xi'}$ for $C_0>0$ and 
$|(\rho_0 -\alpha)\alpha^{-1}|\le C\eps$.
By semiclassicalsymbol calculus in $S(\est{\xi'}^s, (dx')^2+(d\xi')^2)$, 
G\aa rding inequality and 
Proposition~\ref{prop: a priori estimates on traces} we have
\begin{align}\label{est: zone 3 elliptic diference rho alpha}
&\big|    \op_{sc}  (\tilde\Phi \chi_E\chi_3  \rho_0\alpha^{-1} )  -      
\op_{sc}  (\tilde\Phi \chi_E\chi_3   ) v_1 \big|_{L^2 }    \notag  \\
&\qquad= 
\big|    \op_{sc}  (\tilde\Phi \chi_E\chi_3 ( \rho_0-\alpha)\alpha^{-1} )\op_{sc}\big(\xi_1  - i \beta(x'',\xi'') 
 \big)^{ 1/2}\Phi u_1 \big|_{L^2}
\le C \eps h^{-1/2}+C_\eps.
\end{align}
As $\chi_1 +\chi_2 +\chi_3 =1$, Formulas~\eqref{est: zone 1 elliptic diference rho alpha}, 
\eqref{est: zone 2 elliptic diference rho alpha} and \eqref{est: zone 3 elliptic diference rho alpha} 
give the conclusion of Lemma~\ref{lem: difference rho alpha elliptic}.
\end{proof}
From~\eqref{eq: v0 v1 z6 sic}---%
\eqref{eq: difference form rho times inverse alpha} and Lemma~\ref{lem: difference rho alpha elliptic} we have
\begin{align}\label{eq: on traces after conjugaison}
v_0-iv_1=z_9, \text{ where } |z_9|_{L^2}\le C\eps^{1/5}h^{-1/2}+C_\eps h^{-3/8}.
\end{align}
As $v_0$ is supported in $x_1\le 0$ and $v_1$ is supported in $x_1\ge 0$ , if we 
restrict~\eqref{eq: on traces after conjugaison} on $x_1>0$, we obtain
\begin{align*}
|v_1|_{L^2(x_1>0)}=|v_1|_{L^2}\le C\eps^{1/5}h^{-1/2}+C_\eps h^{-3/8},
\end{align*}
and we deduce $|v_0|_{L^2} \le  C\eps^{1/5}h^{-1/2}+C_\eps h^{-3/8}$. This proves Lemma~\ref{lemma: estimation v0 v1}.
\end{proof}
\begin{proposition}\label{lemma: trace goes to 0}
With the previously defined notation, we have  
\begin{align*}
&h^{1/2}|(v_h)_{|x_d=0}|_{H^{1/2}_{sc}}=h^{1/2}|u_1|_{H^{1/2}_{sc}}\to0 \text{ as } h\to 0,  \\
&h^{1/2}| (hD_{x_d}v_h)_{|x_d=0}|_{H^{-1/2}_{sc}}=h^{1/2}| u_0|_{H^{-1/2}_{sc}}\to0 \text{ as } h\to 0.
\end{align*} 
\end{proposition}
\begin{proof}
We have to introduce another small parameter $\nu>0$ chosen below such that $\nu>>\eps$. Let $\chi_H$,  $\chi_G$ and $\chi_E$ 
in $\Con^\infty(\R^{d-1}\times \R^{d-1})$   such that
\begin{align*}
&\chi_H(x',\xi') \text{ is supported in }  R(x',0,\xi')-1\le -\nu,   \\
&\chi_G(x',\xi')  \text{ is supported in }  |R(x',0,\xi')-1|\le 2\nu   ,\\
&\chi_E(x',\xi')  \text{ is supported in }    R(x',0,\xi')-1\ge \nu ,\\
&\chi_H+\chi_G+\chi_E=1.
\end{align*}
Let $\psi_D$, $\psi_Z$ and $\psi_N$ in $\Con^\infty(\d\Omega)$ such that
\begin{align*}
&\psi_D  \text{ is supported in } x_1\le -\eps\mu/4 \\
&\psi_Z  \text{ is supported in } |x_1|\le \eps\mu/2 \\
& \psi_N \text{ is supported in } x_1\ge \eps\mu/4  \\
& \psi_D + \psi_Z +\psi_N =1.
\end{align*}
To be clear, the $\eps$ is the one used in elliptic region. We recall the estimates obtained in previous sections.
From Proposition~\ref{prop: estimation traces hyperbolic} we have 
\begin{align}
& |  \op_{sc}( \chi_H)u_1  |_{H^1_{sc}}\le C_\nu   \notag \\
&   |  \op_{sc}( \chi_H)u_0 |_{L^2}\le C_\nu .  \label{est: hyperbolic }
\end{align}
From Proposition~\ref{prop: estimation traces glancing} we have
\begin{align}
& |  \op_{sc}( \chi_G)  u_1 |_{H^1_{sc}}\le C \nu^{1/4} h^{-1/2}+C_\nu h^{-3/8}
\notag
 \\
&   | ( \op_{sc}( \chi_G) u_0 ) _{|x_d=0}  |_{L^2_{sc}}\le C \nu^{3/4} h^{-1/2}
+C_\nu h^{-3/8}  \label{est: Glancing} .
\end{align}
In the elliptic region, we estimate $v_j$ and we have to estimate $u_j$ for $j=0,1$. To be precise,
$v_0=   \op_{sc}(\xi_1+i\beta(x'', \xi''))^{-1/2}\psi_Z u_0$ and $v_1=  \op_{sc}\big(\xi_1  - i \beta(x'',\xi'')  
\big)^{ 1/2}\psi_Z u_1 $, where $\beta$ is defined by formula~\eqref{def: le beta} and 
$\Phi=\psi_Z$. By Lemma~\ref{lemma: estimation v0 v1}, we have
\begin{align*}
&|v_1|_{L^2}  \le  C\eps^{1/5}h^{-1/2}+C_\eps h^{-3/8} ,\\
&|v_0|_{L^2} \le  C\eps^{1/5}h^{-1/2}+C_\eps h^{-3/8}.  
\end{align*}
Thus $\op_{sc}(\chi_E(\xi_1-i\beta(x'', \xi''))^{-1/2})v_1= \op_{sc}(\chi_E) \psi_Z u_1+h\op_{sc}(r_0)u_1$, where 
$r_0\in S(1,(dx')^2+(d\xi')^2)$. 
%%%%%%%%%%%%%%%%%%%%%
%
%   Lemmma
%
%%%%%%%%%%%%%%%%%%%%%
\begin{lemma}
	\label{lem: estimation beta}	
On the support of $\chi_E$, we have $| (\xi_1-i\beta(x'', \xi''))^{-1/2} |\le C\nu^{-1/4}\est{\xi'}^{-1/2}
=C_\nu \est{\xi'}^{-1/2}$, where $C_\nu$ does not depend on $\eps$.
\end{lemma}
\begin{proof}
We have to consider different cases.

 If $| \xi''|\ge C$, where $C$ sufficiently large to have $R_0(x'',\xi'')\ge 2$, we have $\Re\beta(x'', \xi'')\ge C_1\est{\xi''}$, for $C_1>0$. 
 Then $|\xi_1-i\beta(x'', \xi'')|\ge C_2\est{\xi'}$,  for $C_2>0$. 

If $| \xi''|\le C$ and $|\xi_1|$ sufficiently large, we have $|\xi_1-i\beta(x'', \xi'')|\ge C_3\est{\xi_1}\ge C_4\est{\xi'}$. 

If $| \xi' |$ bounded, on the support of $\chi_E$, if $\eps$ is sufficiently small with respect $\nu$, 
we have $\xi_1^2+R_0(x'',\xi'')-1\ge \nu/2$. If $R_0(x'',\xi'')-1\ge \delta_1 \nu$, for 
$\delta_1>0$, then $\Re\beta(x'', \xi'')\ge \delta_2\nu^{1/2}$, for $\delta_2>0$. 
If $\delta_1$ is sufficiently small and $R_0(x'',\xi'')-1\le \delta_1 \nu$, then 
$\xi_1^2\ge \nu/4$ and $|\xi_1-i\beta(x'', \xi'')|\ge \delta_3\nu^{1/2}$. 

In all cases, we get that  $|\xi_1-i\beta(x'', \xi'')|\ge \delta_3\nu^{1/2}\est{\xi'}$. 
This implies  the result.
\end{proof}

By symbol calculus, we have
    \begin{align*}
  | \op_{sc}(\chi_E) \psi_Z u_1|_{H^{1/2}_{sc}} & \le    |\op_{sc}(\est{\xi'}^{1/2}) \op_{sc}
  (\chi_E(\xi_1-i\beta(x'', \xi''))^{-1/2})v_1|_{L^2}  + C_{\nu,\eps} h |   u_1 |_{H^{1/2}_{sc}}   \notag \\
  &\le  |\op_{sc}(\est{\xi'}^{1/2}\chi_E(\xi_1-i\beta(x'', \xi''))^{-1/2})v_1|_{L^2}  + C_{\nu,\eps} h |   u_1 |_{H^{1/2}_{sc}}  ,
    \end{align*}
    and by  G\aa rding inequality~\eqref{eq: estimation L2 sharp with Garding boundary} and 
    Lemma~\ref{lemma: estimation v0 v1} we have
   \begin{align}   
  | \op_{sc}(\chi_E) \psi_Z u_1|_{H^{1/2}_{sc}}  &\le C_\nu | v_1|_{L^2}+C_{\nu,\eps}h^{1/2} 
   + C_{\nu,\eps} h |   u_1 |_{H^{1/2}_{sc}}  \notag \\
  &\le C_\nu\eps^{1/5}h^{-1/2}+C_{\nu,\eps} h^{-3/8}  + C_{\nu,\eps} h |   u_1 |_{H^{1/2}_{sc}}  . 
   \label{est: trace u_1 Z zone}
  \end{align}
  For $v_0$ we have \[\op_{sc}(\chi_E(\xi_1+i\beta(x'', \xi''))^{1/2})v_0=   \op_{sc}(\chi_E)\psi_Z u_0+h\op_{sc}(r_0)u_0, \
  \text{ where }r_0\in S(1,(dx')^2+(d\xi')^2).  \]
 A proof analogous  to the one of Lemma~\ref{lem: estimation beta} gives
  $| (\xi+i\beta(x'', \xi''))^{1/2} |\le C_\nu\est{\xi'}^{1/2}$,
  we have by G\aa rding inequality and symbol calculus
  \begin{align}
  |   \op_{sc}(\chi_E)\psi_Z u_0 |_{H^{-1/2}}&   \le C_\nu | v_0|_{L^2} +C_{\nu,\eps}h^{1/2}
  +C_{\nu,\eps} h |   u_0 |_{H^{-1/2}_{sc}} 
   \notag \\
  &\le  C_\nu\eps^{1/5}h^{-1/2}+C_{\nu,\eps} h^{-3/8}
  +C_{\nu,\eps} h |   u_0 |_{H^{-1/2}_{sc}} .    \label{est: trace u_0 Z zone}
  \end{align}
  From Proposition~\ref{proposition: estimation elliptique Dirichlet Neumann}, we have 
  \begin{align}
  &|  \op_{sc}(\chi_E) \psi_Du_0 |_{L^2}\le C_{\nu,\eps} h^{1/2} \text{ and }\psi_Du_1=0 \notag\\
  & |  \op_{sc}( \chi_E) \psi_N u_1 |_{H^1_{sc}}\le C_{\nu,\eps}  h^{1/2}  \text{ and }\psi_Nu_0=0  .
  \label{est: elliptic Dirichet Neumann zone}
  \end{align}
  As $u_j= \op_{sc}(\chi_H)  u_j   + \op_{sc} (\chi_G)  u_j  +   \op_{sc}  (\chi_E) \psi_N u_j  +   \op_{sc}  (\chi_E)   
   \psi_Z u_j +   \op_{sc}(\chi_E)  \psi_D u_j $, we have, by~\eqref{est: hyperbolic }---\eqref{est: elliptic Dirichet Neumann zone},
  \begin{align*}
h^{1/2} | u_j  |_{H^{-1/2+j}_{sc}}\le C\nu^{(3-2j)/4} + C_\nu\eps^{1/5}+C_{\nu,\eps}h^{1/8}
+ C_{\nu,\eps} h^{3/2} | u_j  |_{H^{-1/2+j}_{sc}}  .
  \end{align*}
  Choosing first $\nu$ sufficiently small, second $\eps$ sufficiently small, we can absorb the right hand side term 
   $ C_{\nu,\eps} h^{3/2} | u_j  |_{H^{-1/2+j}_{sc}} $ by the left hand side term taking $h$ sufficiently small. 
   The limit superior with respect to $h$ of the left hand side can be estimated by any positive number. This proves Proposition~\ref{lemma: trace goes to 0}.
  \end{proof}

  %%%%%%%%%%%%%%%
%
% Subsection
%
%%%%%%%%%%%%%%%%%%

\subsection{Support of semiclassical measure in a \nhd of boundary}
\label{sec: Support of semiclassical measure in a nhd of boundary}

We can now prove Proposition~\ref{prop: measure supported on the characteristic set}, 
that is,  the measure $\mu$ is supported 
on $p=0$, in a \nhd of $x_0\in\d\Omega$.
\begin{proof} 
The proof is based on the results obtained by Proposition~\ref{prop: equation on v spectrally localised} 
and Proposition~\ref{lemma: trace goes to 0}. 
We recall that in local coordinates $p(x,\xi)=\xi_d^2+R(x,\xi')-1$. We have with the 
notation~\eqref{def: extension omega}
\begin{align}\label{eq: equation obtained on extension by 0}
[h^2D_{x_d}+R(x,hD_{x'})-1]\underline{v_h}=h\underline{q_h}-ih(hD_{x_d}v_h)_{|x_d=0}
\otimes \delta_{x_d=0}-ih(v_h)_{|x_d=0}\otimes hD_{x_d}\delta_{x_d=0}.
\end{align}
Let $\varphi\in\Con_0^\infty(\R_{\xi_d})$ and $\chi\in\Con_0^\infty(\R_{x_d})$ be such that $\chi $
 is supported in a \nhd of $0$. 
Let $\ell\in\Con_0^\infty(\R^{d-1} _{x'} \times  \R^{d-1}_{\xi'} )$. By symbol calculus we have
\begin{align*}
\Op_{sc}\big(  \chi(x_d)\varphi(\xi_d)\ell(x',\xi') \big)[h^2D_{x_d}+R(x,hD_{x'})-1] \qquad \qquad \qquad \qquad \qquad \qquad \qquad \qquad  \   \notag \\
=\Op_{sc}\big(  \chi(x_d)\varphi(\xi_d)\ell(x',\xi') (\xi_d^2+R(x,\xi')-1)\big) +h\Op_{sc}(r_0),
\end{align*}
where $r_0\in S^0$. Then
\begin{align*}
I&:=\Big(   
  \Op_{sc}\big(  \chi(x_d)\varphi(\xi_d)\ell(x',\xi') \big)[h^2D_{x_d}+R(x,hD_{x'})-1]\underline{v_h}|\underline{v_h}  
\Big)    \\
&=\underbrace{\Big(\Op_{sc}\big(  \chi(x_d)\varphi(\xi_d)\ell(x',\xi') (\xi_d^2+R(x,\xi')-1)\big) \underline{v_h}| \underline{v_h}  \Big)}_{=:A}  +\underbrace{h (\Op_{sc}(r_0)\underline{v_h}|\underline{v_h}  )}_{=:B}.
\end{align*}
By definition of the semiclassical  measure, the  term  $A$ converges to 
$\est{\mu|  \chi(x_d)\varphi(\xi_d)\ell(x',\xi') (\xi_d^2+R(x,\xi')-1)}$ as $h$ to $0$.
The  term  $B$ is estimated by $Ch\| v_h \|^2_{L^2(\Omega)}$ and this converges to 0 as $h$ to 0 by Proposition~\ref{prop: equation on v spectrally localised}.

By~\eqref{eq: equation obtained on extension by 0} we also have  
\begin{align*}
I&=h \big(   \Op_{sc}\big(  \chi(x_d)\varphi(\xi_d)\ell(x',\xi') \big)  \underline{ q_h }|  \underline{v_h}   \big) \\
&\quad -ih   \big(   \Op_{sc}\big(  \chi(x_d)\varphi(\xi_d)\ell(x',\xi') \big)  (hD_{x_d}v_h)_{|x_d=0}\otimes \delta_{x_d=0} |  \underline{v_h}   \big) \\
&\quad -ih \big(    \Op_{sc}\big(  \chi(x_d)\varphi(\xi_d)\ell(x',\xi') \big)  (v_h)_{|x_d=0}\otimes hD_{x_d}\delta_{x_d=0}  | \underline{v_h}    \big)  =I_1+I_2+I_3.
\end{align*}
 Obviously we have
 \begin{equation*}
 |I_1|\le C h \| q_h \|_{L^2(\Omega}  \| v_h \|_{L^2(\Omega},
 \end{equation*}
 as $ \chi(x_d)\varphi(\xi_d)\ell(x',\xi') \in S^0$. Then $I_1\to 0$ as $h\to 0$ by 
 \eqref{eq: semiclassical equation on v_h}.
 By  exact calculus, we have
 \begin{align*}
 \Op_{sc}\big(   \chi(x_d)\varphi(\xi_d)\ell(x',\xi') \big)
 =   \chi(x_d) \op_{sc}\big(  \ell(x',\xi') \big) \Op_{sc}( \varphi(\xi_d)).
 \end{align*}
 Let $w_j=((hD_{x_d})^{1-j}v_h)_{|x_d=0}$, we have for $j=0,1$
 \begin{align} \label{est: term traces for p mu =0}
& h  | \big(   \Op_{sc}\big(  \chi(x_d)\varphi(\xi_d)\ell(x',\xi') \big)  w_j\otimes (h D_{x_d})^j \delta_{x_d=0} |  \underline{v_h}   \big)|  \notag  \\
& \quad  = h|    \big(     \chi(x_d)  \op_{sc}  \big(     \ell(x',\xi') \big) w_j    \otimes   \Op_{sc}   
\big(  \varphi(\xi_d)\big)(h D_{x_d})^j \delta_{x_d=0}  |    \underline{v_h}   \big) |  \notag\\
& \quad \le  h | w_j |_{H^{j-1/2}}\|  \Op_{sc}   \big(  \varphi(\xi_d)\big)(h D_{x_d})^j \delta_{x_d=0}  \|_{L^2(\R)}  \|  v_h\|_{L^2(\Omega)},
 \end{align}
 where we have used, to estimate $w_0$,  that $\est{\xi'}^{1/2}\ell(x',\xi')$ is bounded on $L^2$, as $\ell$ is compactly supported. 
 A direct computation in Fourier variable gives
 that $\|  \Op_{sc}   \big(  \varphi(\xi_d)\big)(h D_{x_d})^j \delta_{x_d=0}  \|_{L^2(\R)}  \lesssim h^{-1/2}$. From~\eqref{est: term traces for p mu =0}, we obtain
 \begin{align*}
 | I_2|+ | I_3|\le C ( |  (v_h)_{|x_d=0}   |_{H^{1/2}} +   |  (D_{x_d}v_h)_{|x_d=0}   |_{H^{-1/2}} )h^{1/2}\to 0   \text{ as }   h  \to 0,
 \end{align*}
 by Proposition~\ref{lemma: trace goes to 0}. We conclude that $\est{\mu|  \chi(x_d)\varphi(\xi_d)\ell(x',\xi') (\xi_d^2+R(x,\xi')-1}=0$ and by density of functions spanned by
 $ \chi(x_d)\varphi(\xi_d)\ell(x',\xi')$ in $\Con_0^\infty(\R^d\times \R^d)$, we have that $\est{ p(x,\xi) \mu | \phi(x,\xi)}=0$, for all $\phi\in \Con_0^\infty(\R^d\times \R^d)$. 
 This gives the conclusion of Proposition~\ref{prop: measure supported on the characteristic set}.
 \end{proof}
 %%%%%%%%%%%%%%%
%
% Subsection
%
%%%%%%%%%%%%%%%%%%
\subsection{The semiclassical  measure is not identically null}
\label{sec: he semiclassical  measure is not identically null}

\begin{proposition}
The measure $\mu$ constructed at the beginning of 
Section~\ref{subsec: semiclassical defect measure is null on characteristic set}
for the sequence $(v_h)_h$ satisfying~\eqref{eq: semiclassical equation on v_h} 
is not identically 0, i.e. $\mu\not\equiv 0$.
\end{proposition}
\begin{proof}
Let $\phi\in\Con_0^\infty(\R)$ be such that $\phi=1$ in a \nhd of 0. Let $s\in(0,1/2)$. Let $\phi_R(\xi)=\phi(|\xi|/R)$, we have 
\begin{align*}
\| \Op_{sc}(1-\phi_R ) \underline{v_h}\|_{L^2(\R^d)}& \le C R^{-s}\|  \Op_{sc}(\est{\xi}^s )\underline{v_h}\|_{L^2(\R^d)}
\le CR^{-s}\| {v_h}\|_{H^s_{sc}(\Omega)}\\
&\le  C R^{-s}\| {v_h}\|_{H^1_{sc}(\Omega)}
\le C  R^{-s},
\end{align*}
as $\| w\|_{H^s_{sc}(\R^d)}$ is equivalent to $\| w\|_{H^s_{sc}(\Omega)}$  (uniformly with respect to $h\in(0,1)$) if $w$ is supported in $\overline \Omega$.
Then for $R$ sufficiently large, $\| \Op_{sc}(1-\phi(|\xi|/R) \underline{v_h}\|_{L^2(\R^d)}\le 1/2$.
We thus have
\begin{align*}
(\Op_{sc}(\phi_R)  \underline{v_h}| \underline{v_h})_{L^2(\R^d)}&=\| \underline{v_h}  \|_{L^2(\R^d)}^2
-(  \Op_{sc}(1-\phi_R ) \underline{v_h}|  \underline{v_h}  )_{L^2(\R^d)}\\
&\ge 1-\| \Op_{sc}(1-\phi(|\xi|/R) \underline{v_h}\|_{L^2(\R^d)}\ge 1/2.
\end{align*}
Let $\chi\in \Con_0^\infty(\R^d)$ be such that $\chi(x)=1$ for $x\in\Omega$ and $\chi\ge0$, we have
\begin{align*}
(\Op_{sc}(\phi_R)  \underline{v_h}| \underline{v_h})_{L^2(\R^d)}=(\Op_{sc}(\phi_R(\xi)\chi(x))  \underline{v_h}| \underline{v_h})_{L^2(\R^d)}\to \est{\mu, \chi(x)\phi_R(\xi)} \text{ as } h\to0,
\end{align*}
 we obtain
$\est{\mu,  \chi(x)\phi_R(\xi)}\ge 1/2$. Then $\mu$ is not identically null.
\end{proof}

\subsection{The semiclassical  measure is null on the support of $a$}
\label{sec: The semiclassical  measure is null on the support of a}

Before proving the result we need to extend the space of test functions acting on $\mu$. We have the following lemma.

Let $b(x,\xi')\in\Con^\infty(\R^d\times\R^{d-1})$,  we can give a sense to the expression 
$\est{\mu,b(x,\xi')\xi_d^j}$ for all $j\in\N$. 
\begin{lemma}\label{lem: extension measure mu}
Let $\Phi\in\Con_0^\infty(\R)$, be such that $\Phi(\sigma)=1$, for $|\sigma|\le 1$. Let $j\in\N$, then  the quantity 
$\est{\mu,b(x,\xi')\xi_d^j \Phi(|\xi|/R )}$ does not depend on $R$ for sufficiently large $R$. 

By definition, we denote $\est{\mu,b(x,\xi')\xi_d^j}=\lim_{R\to\infty}\est{\mu,b(x,\xi')\xi_d^j \Phi(|\xi|/R )}$.
\end{lemma}
\begin{proof}
As $p\mu=0$, $\mu $ is supported in $|\xi'|^2+\xi_d^2\le C_0$, for $C_0>0$, sufficiently large. 
 If $R$ is sufficiently large and $R'>R$, $\Phi(|\xi|/R)-\Phi(|\xi|/{R'} ) =0$, if $|\xi |\le R$, in particular if $R^2>C_0$. Then
 $b(x,\xi')\big(  \Phi(|\xi|/R)-\Phi(|\xi|/{R'} )  \big)=0$, on the support of $\mu$. This proves that 
 $\est{\mu,b(x,\xi')\xi_d^j \Phi(|\xi|/R )}$, does not depend on $R$ if $R$ is sufficiently large.
\end{proof}
\begin{proposition}
We have $a\mu=0$.
\end{proposition}
\begin{proof}
From Proposition~\ref{prop: equation on v spectrally localised} we have 
$-h^2P v_h+v_h-ih av_h=hq_h$. The inner product with $v_h$ lead to
$( -h^2P v_h+v_h-ih av_h| v_h)=h(q_h|v_h)$. Taking the imaginary part of this equation, as $(Pv_h|v_h)$ is real, we have
$-(av_h|v_h)=\Im (q_h|v_h)$. As $|\Im (q_h|v_h)|\le \| q_h\| \| v_h\|\to 0$, as $h\to 0$, we have 
$(av_h|v_h)\to 0$, as $h\to 0$.
Let $\Phi\in\Con_0^\infty(\R)$, be such that $\Phi=1$ in a \nhd of 0, and $0\le \Phi\le 1$. 

By Lemma~\ref{lem: extension measure mu}, 
the limit when $h$ goes to 0 of 
$(\Op_{sc}(a(x)\Phi(|\xi|/R))v_h|v_h)$ does not depend on $R$ for $R$ large enough. 
By G\aa rding inequality~\eqref{eq: Garding inequality R d}, as $a\ge0$, 
$(\Op_{sc}(a(x)\Phi(|\xi|/R))v_h|v_h)$ and $(\Op_{sc}(a(x)(1-\Phi(|\xi|/R)))v_h|v_h)$ are non negative 
modulo ${\cal O}(h)$. 

Consequently, we have 
$\lim_{h\to0} (\Op_{sc}(a(x)\Phi(|\xi|/R))v_h|v_h)\ge0$, $\lim_{h\to0}(av_h|v_h)=0$ moreover 
$(\Op_{sc}(a(x)(1-\Phi(|\xi|/R)))v_h|v_h)$ has a limit and 
$\lim_{h\to0} (\Op_{sc}(a(x)(1-\Phi(|\xi|/R)))v_h|v_h)\ge0$. Since 
\begin{align*}
\lim_{h\to0} (\Op_{sc}(a(x)\Phi(|\xi|/R))v_h|v_h)+\lim_{h\to0} 
(\Op_{sc}(a(x)(1-\Phi(|\xi|/R)))v_h|v_h)=0,
\end{align*}
we deduce that  $\lim_{h\to0} (\Op_{sc}(a(x)\Phi(|\xi|/R))v_h|v_h)=0$.
 This implies that $\est{\mu,a}=0$, and as $a$ and $\mu $ are non negative, 
 we deduce $a\mu=0$.
\end{proof}
%%%%%%%%%%%%%%%x
%
% Section
%
%%%%%%%%%%%%%%%%%%
\section{Measure properties}
	\label{Sec: Propagation of measure}
%%%%%%%%%%%%%%%x
%
% Subsection
%
%%%%%%%%%%%%%%%%%%
\subsection{Action of Hamiltonian} \label{sec: Action of Hamiltonian}
We first recall the main results proved in previous sections. 
There exists a sequence $(v_h)_h$ satisfying the following properties
\begin{align}\label{properties: sequence and measure}
&h^2Pv_h-v_h+ih av_h=hq_h , \notag\\
&\| v_h\|_{L^2(\Omega)}=1    \text{ and }\|h\nabla  v_h\|_{L^2(\Omega)}\le 2 , \notag \\
&\| q_h\|_{L^2(\Omega)}\to 0 \text{ as } h\to 0,\notag \\
&h^{1/2}|(v_h)_{|x_d=0}|_{H^{1/2}_{sc}}\to0 \text{ as } h\to 0,    \notag \\
&h^{1/2}| (hD_{x_d}v_h)_{|x_d=0}|_{H^{-1/2}_{sc}}\to0 \text{ as } h\to 0,    \notag \\
& p\mu=0  .
\end{align}
We also proved that $a\mu=0$ but in the following we do not  systematically use this property.
These results was stated in Proposition~\ref{prop: equation on v spectrally localised},  
Proposition~\ref{prop: measure supported on the characteristic set} and
Proposition~\ref{lemma: trace goes to 0}.

\subsubsection{Interior  formula}

We begin by stated that the measure $\mu$ is propagated along the $H_p$ flow in interior of domain. This property is classical but 
the proof is simpler in this case than in a \nhd of boundary, even if the main ideas are used.
\begin{proposition}
	\label{prop: interior formula for Hp mu}
Let $b\in\Con^\infty_0(\Omega\times\R^d)$. We have $\est{H_p\mu-2a\mu,b}=0$.
\end{proposition}
\begin{proof}
We consider the following quantity
\begin{align*}
A&=ih^{-1} \big( b(x,hD)\big(  h^2P-1 +iha\big)  v_h   |  v_h   \big)_{L^2(\Omega)}  \\
&\quad 
- ih^{-1} \big( b(x,hD) v_h   |\big(  h^2P-1 +iha\big)   v_h   \big)_{L^2(\Omega)} \\
&=   ih^{-1} \big( b(x,hD)hq_h |  v_h   \big)_{L^2(\Omega)}  
- ih^{-1} \big( b(x,hD) v_h   |h q_h  \big)_{L^2(\Omega)} .
\end{align*}
We have $|  A|\lesssim \|v_h  \|_{L^2(\Omega)} \| q_h  \|_{L^2(\Omega)}$, then $A$ goes to 0 as $h$.
As $b$ is supported far away the boundary of $\Omega$, we have 
$
 \big( b(x,hD) v_h   |\big(  h^2P-1 +iha\big)   v_h   \big)_{L^2(\Omega)}
 = \big( \big(  h^2P-1 -iha\big) b(x,hD) v_h   |  v_h   \big)_{L^2(\Omega)}$. As the principal symbol of  $b(x,hD) a-ab(x,hD) $
 is ${\cal O} (h)$, we have 
\begin{align*}
A&=ih^{-1} \big(  \big[ b(x,hD),   h^2P-1 \big]  v_h   |  v_h   \big)_{L^2(\Omega)} 
 -2\big(a  b(x,hD)   v_h   |  v_h   \big)_{L^2(\Omega)}
 +{\cal O} (h)\| v_h  \|_{L^2(\Omega)}^2 ,
\end{align*}
and by symbol calculus the principal symbol of $ \big[ b(x,hD),   h^2P-1 \big]  $ is $- ih \{ b , p \}$. Then 
\begin{align*}
A=  \big( (\Op_{sc}(\{ b , p \} )-2a(x) b(x,hD) )  v_h|v_h\big)_{L^2(\Omega)}+{\cal O} (h)\| v_h  \|_{L^2(\Omega)}^2,
\end{align*}
then $A\to \est{\mu, \{ b , p \} -2ab}$ as $h\to 0$, which gives the result.
\end{proof}

\subsubsection{Limit computations}

In the following section  quantities as $(\op_{sc}(a) h^jD_{x_d}^j v_h| v_h)$ for $j=0,1,2$ appear.  We need to evaluate their  
limits in term of the  measure $\mu$. To do so we shall now state some technical results.
\begin{proposition} \label{lem: measure on xi d power 0}
Let $b_0(x,\xi')\in S(1, (dx)^2+(d\xi')^2)$,  and $b_1(x,\xi''),\ b_{2}(x,\xi'')\in \Con_0^\infty(\R^d\times\R^{d-2})$.  
Let $b(x,\xi')= b_0(x,\xi')+b_1(x,\xi'') \xi_1+b_{2}(x,\xi'') \xi_1^2$. 
We have
$\big(  \op_{sc}(b) v_h| v_h \big)_{L^2(x_d>0)} \to \est{ \mu , b} \text{ as } h\to 0$.
\end{proposition}

For the proof we need the following lemma.
\begin{lemma} \label{lem: estimation H s for extension by 0}
Let $\Phi\in\Con_0^\infty(\R)$, be such that $\Phi(\sigma)=1$, for $\sigma$ in a \nhd of $0$. 
For all $s\in (0,1/2)$, there exists $C>0$, such that
\begin{align*}
&\Big\|  \Big(1-\Op_{sc}\big( \Phi(\xi_d/R)\big) \Big) 1_{x_d>0}v \Big\|_{L^2(\R^d)}\le {C}{R^{-s}}\|  v\|_{H^s_{sc}(x_d>0)}\\
&\Big\|  \Big(1-\Op_{sc}\big( \Phi(\xi_d/R)\Phi(|\xi'|/R)\big) \Big) 1_{x_d>0}v \Big\|_{L^2(\R^d)}\le {C}{R^{-s}}\|  v\|_{H^s_{sc}(x_d>0)},
\end{align*}
for all $v\in H^s(x_d>0)$, for all $h\in(0,1)$ and all $R>1$.
\end{lemma}
We recall that 
$$H^s_{sc}(x_d>0)= \{  u\in \D'(x_d>0): \exists w\in H^s_{sc}(\R^d), \ w_{|x_d>0}=u \},$$ 
and for $u\in H^s_{sc}(x_d>0)$, we define
$\| u\|_{H^s_{sc}(x_d>0)}=\inf \{ \| w\|_{H^s_{sc}(\R^d)},  w_{|x_d>0}=u  \}$.
We recall that for $s\in[0,1/2)$ and $u\in H^s_{sc}(x_d>0)$, one has $1_{x_d>0} u\in H^s_{sc}(\R^d)$ 
and $\| u\|_{H^s_{sc}(x_d>0)}$ and $  \| 1_{x_d>0} u\|_{H^s_{sc}(\R^d)}$ define two equivalent norms (uniformly with respect to $h\in(0,1)$).
\begin{proof}
Let $w=1_{x_d>0}v$. We have
\begin{align}
\Big\|  \Big(1-\Op_{sc}\big( \Phi(\xi_d/R)\big) \Big) 1_{x_d>0}v \Big\|^2_{L^2(\R^d)} &\lesssim \int 
\big(1- \Phi(h\xi_d/R)\big) ^2 |\hat w(\xi',\xi_d)|^2d\xi  \notag\\
& \lesssim \int \Big(  \big(1- \Phi(h\xi_d/R)\big)/(|h\xi_d|^{s}) \Big)^2   \est{h\xi_d}^{2s}|\hat w(\xi',\xi_d)|^2d\xi . \notag
\end{align}
But 
\begin{align*}
  \big|1- \Phi(h\xi_d/R)\big)/(|h\xi_d|^{s})\big|\le R^{-s}   \big|1- \Phi(h\xi_d/R)\big)/(|h\xi_d/R|^{s})\big| \lesssim  R^{-s} ,
\end{align*}
then, as $s<1/2$,  we obtain that 
\begin{align} \label{est: high frequencies in Hs}
\Big\|  \Big(1-\Op_{sc}\big( \Phi(\xi_d/R)\big) \Big) 1_{x_d>0}v \Big\|^2_{L^2(\R^d)}\lesssim 
R^{-2s} \int   \est{h\xi}^{2s}|\hat w(\xi',\xi_d)|^2d\xi 
\lesssim  R^{-2s} \|  v\|_{H^s(x_d>0)}^2,
\end{align}
which is the first estimate of statement. 

By the same method, we  prove that (for any $s>0$)
\begin{align} \label{est: high frequencies in Hs bis}
\Big\|  \Big(1-\Op_{sc}\big( \Phi(\xi'/R)\big) \Big) 1_{x_d>0}v \Big\|^2_{L^2(\R^d)}\lesssim R^{-2s} \int   \est{h\xi'}^{2s}|\hat w(\xi',\xi_d)|^2d\xi 
\lesssim  R^{-2s} \|  v\|_{H^s(x_d>0)}^2,
\end{align}
As $1- \Phi(\xi_d/R)\Phi(|\xi'|/R)=1-\Phi(\xi_d/R)+ \Phi(\xi_d/R)\big( 1-   \Phi(|\xi'|/R) \big)$, we have 
\begin{align*}
\Big\|  \Big(1-\Op_{sc}\big( \Phi(\xi_d/R)\Phi(|\xi'|/R)\big) \Big) 1_{x_d>0}v \Big\|_{L^2(\R^d)}  &\le 
 \Big\|  \Big(1-\Op_{sc}\big( \Phi(\xi_d/R) \big) 1_{x_d>0}v \Big) \|_{L^2(\R^d)}   \\
 &\quad +  \Big\|  \Big(1-\op_{sc}\big(\Phi(|\xi'|/R)\big) \Big) 1_{x_d>0}v \Big\|_{L^2(\R^d)},
\end{align*}
as $\Op_{sc} \big( \Phi(\xi_d/R) \Big)$ is bounded by 1 on $L^2$.
From~\eqref{est: high frequencies in Hs} and \eqref{est: high frequencies in Hs bis}, we obtain
 the second estimate of statement.
\end{proof}
\begin{proof}  [Proof of Proposition~\ref{lem: measure on xi d power 0}]

Let $\Phi\in\Con_0^\infty(\R)$, be such that $\Phi(\sigma)=1$, for $\sigma$ in a \nhd of $0$. We treat the terme $b_0$. We have
as $\op_{sc}(b_0)$ is a tangential operator
\begin{align*}
\big(  \op_{sc}(b_0) v_h| v_h \big)_{L^2(x_d>0)} &= \big(  \op_{sc}(b_0)1_{x_d>0} v_h| 1_{x_d>0}  v_h \big)_{L^2(\R^d)}  \notag \\
& =   \big(  \op_{sc}(b_0) \Op_{sc}\big( \Phi(\xi_d/R )\Phi(|\xi'|/R)\big)1_{x_d>0} v_h| 1_{x_d>0}  v_h \big)_{L^2(\R^d)} \notag \\
&\  \ +   \big(  \op_{sc}(b_0) \Op_{sc}\big( ( 1-   \Phi(|\xi'|/R)  )     \Phi(\xi_d/R ) \big)1_{x_d>0} v_h| 1_{x_d>0}  v_h \big)_{L^2(\R^d)} \notag \\
&\ \ +   \big(  \op_{sc}(b_0) \Op_{sc}\big(  1-\Phi(\xi_d/R ))\big)1_{x_d>0} v_h| 1_{x_d>0}  v_h \big)_{L^2(\R^d)} =A_1+A_2+A_3 .\notag
\end{align*}
By definition of semiclassical  measure and from Lemma~\ref{lem: extension measure mu} 
we have 
\begin{align*}
 A_1 \to \est{\mu, b_0(x,\xi') \Phi(\xi_d/R )\Phi(|\xi'|/R) }=\est{\mu, b_0}\text{ as }h\to0. 
 \end{align*}
 We have 
  \begin{align*}
|A_2|&  \lesssim \| 1_{x_d>0}    \Op_{sc} ( 1-   \Phi(|\xi'|/R)  )v_h\|_{L^2(\R^d)} \| 1_{x_d>0} v_h\|_{L^2(\R^d)}\\
&\lesssim R^{-1}\| v_h\|_{L^2(\R,H^1_{sc}(\R^{d-1}) )}  \|v_h\|_{L^2(x_d>0)},
  \end{align*}
 and, for $s\in(0,1/2)$ we have
 \begin{align*}
 |A_3|\lesssim \|   \Op_{sc}\big(  ( 1-\Phi(\xi_d/R )  )1_{x_d>0} v_h\|_{L^2(\R^d)} \| 1_{x_d>0} v_h\|_{L^2(\R^d)}\lesssim R^{-s}
  \|  v\|_{H^s_{sc}(x_d>0)}\|  v\|_{L^2(x_d>0)}, 
 \end{align*}
 as $ \op_{sc}(b_0)$ is bounded on $L^2(\R^d)$ and from Lemma~\ref{lem: estimation H s for extension by 0}. 
 As $A_1+A_2+A_3$ does not depend on $R$ we obtain the result for $b_0$.
 
In the following we only consider the term $b_{2}\xi_1^2$, the  term $b_{1}\xi_1$ can be managed as 
the previous with some minor modifications. 
\begin{align*}
\big(  \op_{sc}(b_2\xi_1^2) v_h| v_h \big)_{L^2(x_d>0)} &= \big(  \op_{sc}(b_2\xi_1^2)
1_{x_d>0} v_h| 1_{x_d>0}  v_h \big)_{L^2(\R^d)}   \\
& =   \big(  \op_{sc}(b_2\xi_1^2) \Op_{sc}\big( \Phi(\xi_d/R )\Phi(|\xi_1|/R)\big)1_{x_d>0} v_h| 1_{x_d>0}  v_h \big)_{L^2(\R^d)} \notag \\
&\quad +   \big(  \op_{sc}(b_2\xi_1^2) \Op_{sc}\big( ( 1-\Phi(\xi_d/R ))
 \Phi(|\xi_1|/R)\big)1_{x_d>0} v_h| 1_{x_d>0}  v_h \big)_{L^2(\R^d)}   \\
 &\quad +   \big(   \op_{sc}(b_2\xi_1^2) \Op_{sc}\big( ( 1- \Phi(|\xi_1|/R)\big)  
 1_{x_d>0} v_h| 1_{x_d>0}  v_h \big)_{L^2(\R^d)}\\
 & =B_1+B_2+B_3 
\end{align*}
As previously $B_1\to \est{\mu, b_2\xi_1^2  \Phi(\xi_d/R )\Phi(|\xi_1|/R)}  =    \est{\mu, b_2\xi_1^2}$ as $h\to 0$.
We need to prove a regularity result on $v_h$ given by the following lemma which is proven in 
Appendix~\ref{Appendix : lemma regularity H s}.
%%%%%%%%%%%%%%%%%%%%%%%%%%%%%
%
%  Lemma
%
%%%%%%%%%%%%%%%%%%%%%%%%%%%%%%
\begin{lemma}[Zaremba regularity result]
	\label{Lem: regularity H s}
Let  $s\in (0,1/2)$. There exists $C>0$, such that for any $v_h$ satisfying~\eqref{properties: sequence and measure},
we have
$\| v_h\|_{H^{1+s}_{sc}(\Omega)}\le C$.
\end{lemma}
We have after an integration by parts and symbol calculus, for $s\in(0,1/2)$
\begin{align*}
| B_2|&\lesssim \| \Op_{sc} \big(( 1- \Phi(|\xi_d|/R)\big)  1_{x_d>0} hD_{x_1}v_h  \|_{L^2(x_d>0)} \|  v_h  \|_{L^2(x_d>0,H^1_{sc}(\R^{d-1}))} 
+h  \|  v_h  \|_{H^1_{sc}(x_d>0)}^2 \\
&\lesssim R^{-s} \|  hD_{x_1}v_h  \|_{H^s(x_d>0)} \|  v_h  \|_{L^2(x_d>0,H^1_{sc}(\R^{d-1})} +h  \|  v_h  \|_{H^1_{sc}(x_d>0)}^2
\end{align*}
and
\begin{align*}
|B_3|&\lesssim \| \big( ( 1- \Phi(|\xi_1|/R)\big)   v_h  \|_{L^2(x_d>0,H^1_{sc}(\R^{d-1}))} \|  v_h  \|_{L^2(x_d>0,H^1_{sc}(\R^{d-1}))} 
+h\|  v_h  \|_{H^1_{sc}(x_d>0)}^2  \\
&\lesssim R^{-s}\| v_h  \|_{L^2(x_d>0,H^{1+s}_{sc}(\R^{d-1}))} \|  v_h  \|_{L^2(x_d>0,H^1_{sc}(\R^{d-1}))} +h \|  v_h  \|_{H^1_{sc}(x_d>0)}^2 
\end{align*}
where we have applied Lemma~\ref{Lem: regularity H s}.
Then we can conclude as for the term $b_0$.
\end{proof}

In the following lemma we consider the quantity  
$\big(  \op_{sc}(b) h^2D_{x_d}^2 v_h| v_h \big)_{L^2(x_d>0)}$, 
which is not clearly well defined. But as $v_h$ satisfied~\eqref{properties: sequence and measure}, 
we can prove that $h^2D_{x_d}^2v_h\in L^2(x_d>0, H^{-1}_{sc}(\R^{d-1}))$. 
The inner product in 
tangential variables need to be interpreted as a duality product $H^{-1}, H^{1}$.
\begin{proposition}\label{lem: measure on xi d power 2}
Let $b(x,\xi')\in\Con_0^\infty(\R^d\times\R^{d-1})$, we have
\[ \big(  \op_{sc}(b) h^2D_{x_d}^2 v_h| v_h \big)_{L^2(x_d>0)} \to \est{ \mu , b\xi_d^2} \text{ as } h\to 0.\]
\end{proposition}
\begin{proof}
We use the equation satisfied by $v_h$ (see~\eqref{properties: sequence and measure}). 
\begin{align*}
\big(  \op_{sc}(b) h^2D_{x_d}^2 v_h| v_h \big)_{L^2(x_d>0)}&= \big(  \op_{sc}(b) h (q_h-iav_h)| v_h \big)_{L^2(x_d>0)} \\
&\quad - \big(  \op_{sc}(b)(R(x,hD_{x'})-1) v_h| v_h \big)_{L^2(x_d>0)}\\
&= A+B. 
\end{align*}
Clearly
\begin{align*}
| A|\lesssim h(\| q_h\|_{L^2(x_d>0)}  + \|  v_h\|_{L^2(x_d>0)}) \|  v_h\|_{L^2(x_d>0)}\to 0 \text{ as } h\to 0.
\end{align*}
By symbol calculus, $ \op_{sc}(b)(R(x,hD_{x'})-1)=\op_{sc}\big( b(x,\xi')(R(x,\xi')-1) \big)+h\op_{sc}(r_0)$, 
where $r_0\in S^0_{\rm tan}$. Then, by Proposition~\ref{lem: measure on xi d power 0}, we have $B\to -\est{\mu, b(R-1)}$. 
Let $\Phi$ be as given in Lemma~\ref{lem: extension measure mu}, we have for $\lambda$ sufficiently large
\begin{align*}
 -\est{\mu, b(x,\xi')(R(x,\xi')-1)}&= -\est{\mu,  b(x,\xi')(R(x,\xi')-1)\Phi(\xi_d/\lambda)}  \\
 &=  -\est{\mu,  b(x,\xi')(\xi_d^2 +R(x,\xi')-1)\Phi(\xi_d/\lambda)}+\est{\mu ,b(x,\xi')\xi_d^2 \Phi(\xi_d/\lambda)}  \\
 &= \est{\mu ,b(x,\xi')\xi_d^2 \Phi(\xi_d/\lambda)} = \est{\mu ,b(x,\xi')\xi_d^2},
\end{align*}
as $p\mu=0$. Which gives the lemma.  
\end{proof}
\begin{proposition}\label{lem: measure on xi d power 1}
Let 
$b(x,\xi')\in S(1, (dx)^2+(d\xi')^2)$,
we have
\begin{align*}
\big(  \op_{sc}(b) hD_{x_d} v_h| v_h \big)_{L^2(x_d>0)} \to \est{ \mu , b\xi_d} \text{ as } h\to 0.
\end{align*}
\end{proposition}
\begin{proof}
Let $\Phi$ be as given in Lemma~\ref{lem: extension measure mu}, we have for $\lambda>0$,
\begin{align*}
\big(  \op_{sc}(b) hD_{x_d} v_h| v_h \big)_{L^2(x_d>0)}&=  \big(  \op_{sc}(b)  1_{x_d>0}hD_{x_d} v_h| 1_{x_d>0} v_h \big)_{L^2(\R^d)}  \\
&=  \big(  \op_{sc}(b) \Op_{sc}\big(\Phi(\xi_d/\lambda)  \big)   1_{x_d>0}hD_{x_d} v_h| 1_{x_d>0} v_h \big)_{L^2(\R^d)}  \\
&\quad +   \big(  \op_{sc}(b) \Op_{sc}\big(1-\Phi(\xi_d/\lambda)  \big)   1_{x_d>0}hD_{x_d} v_h| 1_{x_d>0} v_h \big)_{L^2(\R^d)} \\
&=A+B.
\end{align*}
By symbol calculus in $S(1,(dx)^2+(d\xi)^2)$, we have  
\begin{align*}
\op_{sc}(b) \Op_{sc}\big(1-\Phi(\xi_d/\lambda) \big)  
= \Op_{sc}\big(1-\Phi(\xi_d/\lambda ) \big)^* \op_{sc}(b) +h\Op_{sc}(r_0),
\end{align*}
where $r_0\in S(1,(dx)^2+(d\xi)^2)$. Then 
\begin{align*}
|B|&\lesssim \big|      \big(  \op_{sc}(b)  1_{x_d>0}hD_{x_d} v_h|  \Op_{sc}\big(1-\Phi(\xi_d/\lambda)  \big) 
 1_{x_d>0} v_h \big)_{L^2(\R^d)}  \big| +h\| hD_{x_d}v_h \|_{L^2(\R^)}\| v_h \|_{L^2(\R^d)}  \\
&\lesssim \lambda^{-s}+ h,
\end{align*}
by Lemma~\ref{lem: estimation H s for extension by 0} and a priori estimates~\eqref{properties: sequence and measure}.

Next we treat the term $A$. We have $  1_{x_d>0}hD_{x_d} v_h= hD_{x_d}  \big( 1_{x_d>0} v_h\big)
+ih(v_h)_{|x_d=0}\otimes \delta_{x_d=0}$. Then 
we have
\begin{align*}
A&=  \big(  \op_{sc}(b) \Op_{sc}\big(\Phi(\xi_d/\lambda)  \big)  hD_{x_d}\big( 1_{x_d>0} v_h \big) | 1_{x_d>0} v_h \big)_{L^2(\R^d)}   \\
&\quad +ih\big(  \op_{sc}(b) \Op_{sc}\big(\Phi(\xi_d/\lambda)  \big)  \big( (v_h)_{|x_d=0}\otimes 
\delta_{x_d=0}   \big)|  1_{x_d>0} v_h \big)_{L^2(\R^d)}=A_1+A_2.
\end{align*}
We have $\Op_{sc}\big(\Phi(\xi_d/\lambda)  \big)  hD_{x_d}= \Op_{sc}\big(\xi_d\Phi(\xi_d/\lambda)  \big) $ and 
by symbol calculus in   $S(1,(dx)^2+(d\xi)^2)$, we have 
$  \op_{sc}(b) \Op_{sc}\big(\xi_d\Phi(\xi_d/\lambda)  \big)= \Op_{sc}\big( \xi_d\Phi(\xi_d/\lambda)  b\big)+h\Op_{sc}(r_0)$, 
where $r_0\in  S(1,(dx)^2+(d\xi)^2)$. Then
$A_1=\big(  \op_{sc}(b) \Op_{sc}\big(\Phi(\xi_d/\lambda)  \big)  hD_{x_d}\big( 1_{x_d>0} v_h \big) | 1_{x_d>0} v_h \big)_{L^2(\R^d)} 
 \to \est{\mu ,  \xi_d\Phi(\xi_d/\lambda)  b  }=  \est{\mu ,  \xi_db }$, if $\lambda$ is sufficiently large.
 
 Let $\tilde\Phi$ be such that ${\cal F} \tilde\Phi=\Phi  $, where $\cal F$ is the Fourier transform.  We have 
 $\Op_{sc}\big(\Phi(\xi_d/\lambda)\big)\delta_{x_d=0}= h^{-1}\lambda\tilde\Phi(\lambda x_d/h)$. Then 
 \begin{align*}
 |A_2|&\lesssim  \big|   \big(  \op_{sc}(b)  \big( (v_h)_{|x_d=0}     \lambda\tilde\Phi(\lambda x_d/h)    \big)   
 |  1_{x_d>0} v_h \big)_{L^2(\R^d)}   \big|  \\
 &\lesssim \lambda |  (v_h)_{|x_d=0}  |_{L^2(\R^{d-1})}  |  \tilde\Phi(\lambda x_d/h) |_{L^2(\R)} \| v_h  \|_{L^2(x_d>0)}   \\
  &\lesssim \lambda^{1/2}h^{1/2} |  (v_h)_{|x_d=0}  |_{L^2(\R^{d-1})} ,
 \end{align*}
 as $ |  \tilde\Phi(\lambda x_d/h) |_{L^2(\R)} = Ch^{1/2}\lambda^{-1/2}$. From~\eqref{properties: sequence and measure}, 
 we have $A_2\to0$, as $h\to 0$. From estimates on $A_1$, $A_2$ and $B$ we obtain the result.
 \end{proof}
 
 \subsubsection{Boundary  formulas}
 
 The two next propositions are the analogous of Proposition~\ref{prop: interior formula for Hp mu} at the boundary.
 
\begin{proposition}\label{lemma: first propagation formula}
Let 
$b\in\Con_0^\infty(\R^d\times \R^{d-1})$  $($resp. $b\in\Con_0^\infty(\R^d\times \R^{d-2})$\,$)$
be  real valued functions. We have the following formula
\begin{align}\label{formula: first propagation formula}
\est{H_p\mu -2a\mu, b}=\est{\mu, \{b,p\}-2ab}
&=\lim_{h\to 0}2\Re 
\big(  b(x',0,hD') (v_h)_{|x_d=0} | (hD_{x_d}v_h)_{|x_d=0}  \big)_0.\notag\\
&( \text{ resp.} \lim_{h\to 0}2\Re 
\big(  b(x',0,hD'') (v_h)_{|x_d=0} | (hD_{x_d}v_h)_{|x_d=0}  \big)_0.)
\end{align}
In particular $ \big(  b(x',0,hD') (v_h)_{|x_d=0} | (hD_{x_d}v_h)_{|x_d=0}  \big)_0$ 
$($resp.$ \big(  b(x',0,hD'') (v_h)_{|x_d=0} | (hD_{x_d}v_h)_{|x_d=0}  \big)_0$\,$)$, have  limits as $h\to 0$.
\end{proposition}
\begin{remark}
In what follow we write $b(x',0,hD')$ even if $b$ only depends on variables $(x,\xi'')$.
With Proposition~\ref{lemma: trace goes to 0} we can only proof that  
$ \big(  b(x',0,hD') (v_h)_{|x_d=0} | (hD_{x_d}v_h)_{|x_d=0}  \big)_0$, is a $o(h^{-1})$. 
In the proof below we show that the quantity has a limit and converges to the left hand side 
of~\eqref{formula: first propagation formula}. We are not able to prove that $(v_h)_{|x_d=0}$ 
or $(hD_{x_d}v_h)_{x_d=0}$ are bounded.
\end{remark}
\begin{proof}
We recall~\eqref{form: integ. parts} the integration  by parts formula in semiclassical context.
\begin{equation*}
(u|h D_{x_d}w)_{L^2(x_d>0)}=(h D_{x_d}u| w)_{L^2(x_d>0)} -ih( u_{|x_d=0} |  w_{|x_d=0} )_0,
\end{equation*}
for $u$ and $w$ sufficiently smooth.

To proof the lemma we compute the following quantity by two different manners.
\begin{align*}
A&=ih^{-1} \big( b(x,hD')\big(  h^2D_{x_d}^2 +R(x,hD')-1 +iha\big)  v_h   |  v_h   \big)_{L^2(x_d>0)}  \\
&\quad 
- ih^{-1} \big( b(x,hD') v_h   |\big(  h^2D_{x_d}^2 +R(x,hD')-1 +iha\big)   v_h   \big)_{L^2(x_d>0)} \\
&=   ih^{-1} \big( b(x,hD')hq_h |  v_h   \big)_{L^2(x_d>0)}  - ih^{-1} \big( b(x,hD') v_h   |h q_h  \big)_{L^2(x_d>0)} .
\end{align*}
Then $|A|\lesssim \| q_h \|_{L^2(x_d>0)}   \| v_h  \|_{L^2(x_d>0)}\to 0$ as $h\to 0$ 
from~\eqref{properties: sequence and measure}.
To compute $A$ we now integrate by parts. We have
\begin{align*}
&\big( b(x,hD') v_h   |\big( R(x,hD')-1 \big)   v_h   \big)_{L^2(x_d>0)} =\big( \big( R(x,hD')-1 \big)  
b(x,hD')v_h   |  v_h   \big)_{L^2(x_d>0)}  ,\\
&\big( ihb(x,hD')a v_h | v_h  \big)_{L^2(x_d>0)}-\big( b(x,hD') v_h |    ihav_h  \big)_{L^2(x_d>0)}\\
&\quad =2ih\big( \op_{sc}(b(x,\xi')a(x)) v_h | v_h  \big)_{L^2(x_d>0)}+{\cal O} (h^2),
\end{align*}
 as first $R(x,hD')$ is self-adjoint and does not contain derivative with respect $x_d$ and second 
 by symbol calculus $b(x,hD') a = a b(x,hD')= \op_{sc}(a(x)b(x,\xi'))$ up to ${\cal O}(h)$. We have 
 \begin{align*}
 \big( b(x,hD') v_h   | h^2D_{x_d}^2  v_h   \big)_{L^2(x_d>0)} &= \big(  hD_{x_d} b(x,hD')
  v_h   | hD_{x_d} v_h   \big)_{L^2(x_d>0)}\\
&\quad  -ih  \big(  b(x',0,hD') (v_h)_{|x_d=0} | (hD_{x_d}v_h)_{|x_d=0}  \big)_0  \\
&=   \big(  h^2D_{x_d} ^2 b(x,hD') v_h   |  v_h   \big)_{L^2(x_d>0)}\\
&\quad -ih \big(  b(x',0,hD')    ( hD_{x_d}   v_h)_{|x_d=0} | (v_h)_{|x_d=0}  \big)_0    \\
   &\quad 
  -ih  \big(  b(x',0,hD') (v_h)_{|x_d=0} | (hD_{x_d}v_h)_{|x_d=0}  \big)_0  .
 \end{align*}
 Then we have
 \begin{align*}
 A&=ih^{-1}  \big([  b(x,hD'), \big(  h^2D_{x_d}^2 +R(x,hD')-1 \big)  ] v_h 
   |  v_h   \big)_{L^2(x_d>0)}   -2\big( \op_{sc}(b(x,\xi')a(x)) v_h | v_h  \big)_{L^2(x_d>0)} \\
 &\quad 
 - \big(  b(x',0,hD')    ( hD_{x_d}   v_h)_{|x_d=0} | (v_h)_{|x_d=0}  \big)_0  %   
  -  \big(  b(x',0,hD') (v_h)_{|x_d=0} | (hD_{x_d}v_h)_{|x_d=0}  \big)_0+{\cal O} (h)  .
 \end{align*}
 From the structure of $b$ we claim that 
 \[
 [  b(x,hD'), \big(  h^2D_{x_d}^2 +R(x,hD')-1 \big)  ] 
 = -ih \Op_{sc}\big(  \{ b,p \}\big)+ h^2 \op_{sc}(r_0)
 \]
  where $r_0\in S(1, (dx)^2+(d\xi')^2)$. 
 Indeed, the assumptions imply that $b(x,\xi')\in S(1,  (dx)^2+(d\xi')^2)$,  and $h^2D_{x_d}^2+R$ is a 
 sum of terms $c(x)h^2D_{x_j}D_{x_k}$.
 By exact symbol calculus we have  
 \begin{align*}
  [    b(x,hD'),c(x)h^2D_{x_j}D_{x_k} ]&= [    b(x,hD'),c(x) ]h^2D_{x_j}D_{x_k} +  c(x) [    b(x,hD'), hD_{x_j }] h D_{x_k} \\
 &\quad  +   c(x)  hD_{x_j }[    b(x,hD'), hD_{x_k}]  \\
  &=[    b(x,hD'),c(x) ]h^2D_{x_j}D_{x_k} +  c(x) [    b(x,hD'), hD_{x_j }] h D_{x_k} \\
 &\quad  +   c(x) [    b(x,hD'), hD_{x_k}]  hD_{x_j }+c(x)  [ hD_{x_j }, [    b(x,hD'), hD_{x_k}]  ]\\
  &= -ih \Op_{sc}(\{  b , c\xi_j\xi_k  \})+ h^2\op_{sc}(  \tilde r_0),
 \end{align*}
 where $ \tilde r_0$ is in $S(1, (dx)^2+(d\xi')^2)$. 

First  $  \{ b,p \}= -2\xi_d \d_{x_d} b+\{ b,R\}$, we can apply Propositions~\ref{lem: measure on xi d power 0} and    
\ref{lem: measure on xi d power 1} 
as $\{b,R\}$ is the sum of terms in $S(1,(dx)^2+(d\xi')^2)$ or of the form $q(x,\xi'')$, $q(x,\xi'')\xi_1$ and $q(x,\xi'')\xi_1^2$, 
where $q\in \Con_0^\infty(\R^d\times\R^{d-2})$. Second 
$|(\op_{sc}(r_0) v_h|v_h)|\lesssim \| v_h \|_{L^2(x_d>0)}^2 $
we can conclude that 
 \begin{align*}
& ih^{-1}  \big([  b(x,hD'), \big(  h^2D_{x_d}^2 +R(x,hD')-1 \big)  ] v_h   |  v_h   \big)_{L^2(x_d>0)} 
 -2\big( \op_{sc}(b(x,\xi')a(x)) v_h | v_h  \big)_{L^2(x_d>0)} \\
 &\qquad  \to \est{\mu ,\{b,p\}-2ba} \text{ as }h\to0.
 \end{align*}
By symbol calculus we have $\op_{sc}(b(x',0,\xi'))^*=\op_{sc}(b(x',0,\xi')) +h\op_{sc}(r_0) $, 
where $r_0\in  S(1,(dx)^2+(d\xi')^2)$. 
 Then
 $ \big(  b(x',0,hD') (  hD_{x_d} v_h)_{|x_d=0} | (v_h)_{|x_d=0}  \big)_0 = \big(   ( hD_{x_d} v_h)_{|x_d=0} | 
 b(x',0,hD') (v_h)_{|x_d=0}  \big)_0 + B$,
 where $|B| \le h |   (v_h)_{|x_d=0}|_{H^{1/2}} |   (hD_{x_d}v_h)_{|x_d=0}  |_{H^{-1/2}}  \to 0$ as $h\to0$ 
 by~\eqref{properties: sequence and measure}. 
 This gives the conclusion of Lemma.
 \end{proof}
 %%%%%%%%%%%%%%%%%%%%%
 %
 %  PROPOSITION
 %
 %%%%%%%%%%%%%%%%%%%%%%%
 \begin{proposition}
 	\label{lemma: second propagation formula}
Let $b\in\Con_0^\infty(\R^d\times \R^{d-1})$ be a real valued we have the following formula
\begin{align}
\label{formula: second propagation formula}
&\est{-H_p\mu +2a\mu, b\xi_d }=\est{\mu, \{p, \xi_d b\}+2ab\xi_d} \notag\\
&\qquad=\lim_{h\to 0}  \Re \Big( 
\big( b(x',0,hD') \big(    R(x',0,hD')-1 \big)  (v_h)_{|x_d=0}|  (v_h)_{|x_d=0}  \big)_0  \notag\\
&\qquad \qquad  \qquad \qquad-   \big( b(x',0,hD')  (hD_{x_d}v_h)_{|x_d=0} |  (hD_{x_d}v_h)_{|x_d=0}  \big)_0
\Big).
\end{align}
In particular this means that
\begin{multline*}
\Re    \Big( \big( b(x',0,hD') \big(    R(x',0,hD')-1 \big)  (v_h)_{|x_d=0}|  (v_h)_{|x_d=0}  \big)_0   
\! \\   - \!  \big( b(x',0,hD')  (hD_{x_d}v_h)_{|x_d=0} |  (hD_{x_d}v_h)_{|x_d=0}  \big)_0  \Big) ,
\end{multline*}
 has a limit as $h\to 0$.
\end{proposition}
 \begin{remark}
 As in Proposition~\ref{lemma: first propagation formula} the right hand side of \eqref{formula: second propagation formula} 
 does not have a priori limit and we do not know if each term of the sum has a limit.
 \end{remark}
\begin{proof}
We begin by an observation on regularity of traces. 
From the definition of $q_h$ (see~\eqref{eq: definition of qh}), the terms 
$\theta(h^2P)g_j^h$ are in ${\cal D}(P)$ for $j=1,2$. The term $[\theta (h^2P)  , a]u_h=\theta (h^2P)  ( a u_h)-a\theta (h^2P)u_h$ 
is in ${\cal D}(P)$, it is clear for $\theta (h^2P)  ( a u_h)$ and $a\theta (h^2P)u_h $ is in $H^1(\Omega)$, a direct computation 
shows that $P(a\theta (h^2P)u_h )$ is in $L^2(\Omega)$ and $a\theta (h^2P)u_h$ satisfies the Zaremba trace condition 
as $\theta (h^2P)u_h$ satisfies it.
This implies that $q_h\in H^1(\Omega)$.
In particular we have $h^2D_{x_d}^2v_h= hq_h-(R(x,hD')-1)v_h$, then 
$(\op_{sc}(b)h^2D_{x_d}^2v_h)_{|x_d=0}
=(h\op_{sc}(b)q_h)_{|x_d=0}-(\op_{sc}(b)(R(x,hD')-1+iha)v_h) _{|x_d=0}\in L^2(x_d=0)$, for $b$  
compactly supported 
and using  properties~\eqref{properties: sequence and measure}. In this analysis we do not estimate
 the size of 
the norm with respect $h$ but this allows to give a sense to some terms appearing in what follows.
We introduce the following quantity which is real
\begin{align*}
A&=ih^{-1} \Big(  
\big(  ( b(x,hD') hD_{x_d}+  hD_{x_d} b(x,hD') ^*  )  v_h     |    
(h^2D_{x_d}^2+R(x,hD')-1+iha) v_h     \big)_{L^2(x_d>0)}  \\
&\quad - \big(    (h^2D_{x_d}^2+R(x,hD')-1+iha) v_h    |    ( b(x,hD') hD_{x_d}+  hD_{x_d} b(x,hD') ^*  ) 
 v_h      \big)_{L^2(x_d>0)}
\Big)  \\
&= 
\Big(    
\big(  ( b(x,hD') hD_{x_d}+  hD_{x_d} b(x,hD') ^*  )  v_h     |   q_h     \big)_{L^2(x_d>0)}     \\
   &\quad 
- \big(    q_h    |    ( b(x,hD') hD_{x_d}+  hD_{x_d} b(x,hD') ^*  )  v_h      \big)_{L^2(x_d>0)}
\Big).
\end{align*}
As $ hD_{x_d} b(x,hD') ^* = b(x,hD') ^*  hD_{x_d}+\op_{sc}(r_0)$ where $r_0\in S^0_{\rm tan}$, we obtain
\begin{equation*}
|A|\lesssim \| q_h\|_{L^2(x_d>0)}(   \|  v_h \|_{L^2(x_d>0)}  +  \| hD_{x_d} v_h  \|_{L^2(x_d>0)} ) \to 0 \text{ as } h\to 0,
\end{equation*}
 by~\eqref{properties: sequence and measure}.

Let 
\begin{align}
B=\big(    (h^2D_{x_d}^2+R(x,hD')-1+iha) v_h    |    ( b(x,hD') hD_{x_d}+  hD_{x_d} b(x,hD') ^*  )  v_h      \big)_{L^2(x_d>0)}.
\end{align}
We have by integrations by parts
\begin{align}\label{eq: second propagation formula term B}
B&=\big(    ( b(x,hD') hD_{x_d}+  hD_{x_d} b(x,hD') ^*  )  (h^2D_{x_d}^2+R(x,hD')-1+iha) v_h    |    v_h      \big)_{L^2(x_d>0)}   \notag \\
&\quad   
+\big(    ( b(x,hD') hD_{x_d}+  hD_{x_d} b(x,hD') ^*  )      iha  v_h    |    v_h      \big)_{L^2(x_d>0)}   \notag \\
&\quad  -ih\Big( \big(    ( b(x,hD')+  b(x,hD') ^*  )  (h^2D_{x_d}^2+R(x,hD')-1)   v_h \big) _{|x_d=0}  |   ( v_h  )_{|x_d=0}    \Big)_0.
\end{align}
Let 
\begin{align*}
C= \big(  ( b(x,hD') hD_{x_d}+  hD_{x_d} b(x,hD') ^*  )  v_h     |    h^2D_{x_d}^2 v_h     \big)_{L^2(x_d>0)} .
\end{align*}
We have by integration by parts
\begin{align}\label{eq: second propagation formula term C}
C&= \Big( hD_{x_d} \big( b(x,hD') hD_{x_d}+  hD_{x_d} b(x,hD') ^*  \big)  v_h     |    hD_{x_d} v_h     \Big)_{L^2(x_d>0)} \notag\\
&\quad  -ih \big(       \big(    ( b(x,hD') hD_{x_d}+  hD_{x_d} b(x,hD') ^*  )  v_h   \big)_{|x_d=0}      |  (  hD_{x_d} v_h  )_{|x_d=0}      \big)_0 \notag  \\
&=  \big( h^2D_{x_d}^2 ( b(x,hD') hD_{x_d}+  hD_{x_d} b(x,hD') ^*  )  v_h     |   v_h     \big)_{L^2(x_d>0)}  \notag \\
&\quad -ih  \Big(       \big[   hD_{x_d} \big( b(x,hD') hD_{x_d}+  hD_{x_d} b(x,hD') ^*  \big)  v_h    \big]_{|x_d=0}     |  ( v_h   )_{|x_d=0}      \Big)_0   \notag\\
&\quad  -ih \Big(       \big(    ( b(x,hD') hD_{x_d}+  hD_{x_d} b(x,hD') ^*  )  v_h   \big)_{|x_d=0}      |  (  hD_{x_d} v_h  )_{|x_d=0}      \Big)_0   .
\end{align}
The terms with damping term $a$, coming from $A$ and $B$, give a term 
\begin{align*}
& - 4\big( iha     b(x,hD') hD_{x_d}  v_h    |    v_h      \big)_{L^2(x_d>0)} +h^2 \big( \op_{sc}(r_0)  v_h    |    v_h      \big)_{L^2(x_d>0)}  \\
&\quad+h^2 \big( \op_{sc}(\tilde r_0) hD_{x_d} v_h    |    v_h      \big)_{L^2(x_d>0)}    ,
\end{align*}
where  $r_0,\tilde r_0\in S^0_{\rm tan}$.
From this,~\eqref{eq: second propagation formula term B} and \eqref{eq: second propagation formula term C} we obtain
\begin{align}\label{eq: term A commutator and boundary terms}
A&=ih^{-1} \Big(     \big[     h^2D_{x_d}^2+R(x,hD')-1   ,  b(x,hD') hD_{x_d}+  hD_{x_d} b(x,hD') ^*      \big]     v_h     |    v_h  \Big)_{L^2(x_d>0)}    \notag \\
&  +4\big( a     b(x,hD') hD_{x_d}  v_h    |    v_h      \big)_{L^2(x_d>0)} +h{\cal O}(\|  v_h\|_{L^2(x_d>0)}  ^2+\| hD v_h\|_{L^2(x_d>0)} ^2 )   \notag\\
&\quad - \Big( \big(    ( b(x,hD')+  b(x,hD') ^*  )  (h^2D_{x_d}^2+R(x,hD')-1)   v_h \big) _{|x_d=0}  |   ( v_h  )_{|x_d=0}    \Big)_0 \notag \\
&\quad + \Big(       \big[   hD_{x_d} \big( b(x,hD') hD_{x_d}+  hD_{x_d} b(x,hD') ^*  \big)  v_h    \big]_{|x_d=0}     |  ( v_h   )_{|x_d=0}      \Big)_0   \notag\\
&\quad  + \Big(       \big(    ( b(x,hD') hD_{x_d}+  hD_{x_d} b(x,hD') ^*  )  v_h   \big)_{|x_d=0}      |  (  hD_{x_d} v_h  )_{|x_d=0}      \Big)_0   .
\end{align}
By symbol calculus we have
\begin{align}
&ih^{-1}  \big[     h^2D_{x_d}^2+R(x,hD')-1   ,  b(x,hD') hD_{x_d}+  hD_{x_d} b(x,hD') ^*      \big] \\
& \qquad  =\Op_{sc} \{  \xi_d^2+R(x,\xi')-1 , 2b(x,\xi')\xi_d\}+ h\op_{sc}(r_0)+h\op_{sc}(\tilde r_0)hD_{x_d},
\end{align}
where $r_0,\tilde r_0\in S^0_{\rm tan}$.
Propositions~\ref{lem: measure on xi d power 0}, \ref{lem: measure on xi d power 2} and \ref{lem: measure on xi d power 1} imply that
\begin{align}  \label{eq: braket convergente}
\big(  \Op_{sc} \{  \xi_d^2+R(x,\xi')-1 , 2b(x,\xi')\xi_d\} v_h   |  v_h  \big)_{L^2(x_d>0)}   \to \est{\mu , \{p,2b\}}
\text{ as } h\to 0.
\end{align}
By Proposition~\ref{lem: measure on xi d power 1} the term $4\big( a     b(x,hD') hD_{x_d}  v_h    |    v_h     
 \big)_{L^2(x_d>0)} \to \est{\mu, 4ba\xi_d}$ as $h\to 0$.
And we have
\begin{align*}
&| (h\op_{sc}(r_0) v_h| v_h )_{L^2(x_d>0)} | \lesssim h \|  v_h\|^2_{L^2(x_d>0)}  \to 0 \text{ as } h\to 0,\\
&| (h\op_{sc}(\tilde r_0)hD_{x_d} v_h| v_h )_{L^2(x_d>0)} | \lesssim h \|  hD_{x_d}v_h\|_{L^2(x_d>0)}    \|  v_h\|_{L^2(x_d>0)}   \to 0 \text{ as } h\to 0.
\end{align*}
Then this and~\eqref{eq: braket convergente} imply that
\begin{align}\label{eq: convergence measure interior commutator}
&ih^{-1} \Big(     \big[     h^2D_{x_d}^2+R(x,hD')-1   ,  b(x,hD') hD_{x_d}+  hD_{x_d} b(x,hD') ^*      \big]     v_h     |    v_h  \Big)_{L^2(x_d>0)}  \notag  \\
&\quad +4\big( a     b(x,hD') hD_{x_d}  v_h    |    v_h      \big)_{L^2(x_d>0)}   \notag \\
&\qquad \to \est{\mu , \{p,2b\}   +4ab\xi_d} \text{ as } h\to 0.
\end{align}
We now treat the boundary terms coming from~\eqref{eq: term A commutator and boundary terms}. We have by symbol calculus
\begin{align*}
&b(x,hD') + b(x,hD') ^*=2 b(x,hD')+h\op_{sc}(r_0), \\
&b(x,hD') hD_{x_d}+  hD_{x_d} b(x,hD') ^* =2 b(x,hD') hD_{x_d}+h\op_{sc}(r_0), 
\end{align*}
\begin{multline*}
 hD_{x_d} ( b(x,hD') hD_{x_d}+  hD_{x_d} b(x,hD') ^*  ) \\ = ( b(x,hD') +   b(x,hD') ^*)h^2D_{x_d}^2+h\op_{sc}(r_0)+h\op_{sc}(\tilde r_0)hD_{x_d},
\end{multline*}
where $r_0,\tilde r_0\in S^0_{\rm tan}$. Then boundary terms of $A$ can be written as
\begin{multline}\label{eq: boundary terms A convergence}
 - \Big( \big(    (2 b(x,hD')  )  (R(x,hD')-1)   v_h \big) _{|x_d=0}  |   ( v_h  )_{|x_d=0}    \Big)_0 \\
 + \Big(       \big(    2 b(x,hD') hD_{x_d}   v_h   \big)_{|x_d=0}      |  (  hD_{x_d} v_h  )_{|x_d=0}      \Big)_0   ,
\end{multline}
up to terms estimated by $h| \big( (  \op(r_0) v_h)_ {|x_d=0}   | (v_h)_ {|x_d=0}  \big)_0|
+h | \big( (  \op(\tilde r_0)hD_{x_d} v_h)_ {|x_d=0}   | (v_h)_ {|x_d=0}  \big)_0|$, and these terms
 converge to 0 from~\eqref{properties: sequence and measure}. 
Recalling that  $A\to 0 $ as $h\to 0$ and $A$ real valued,
 we deduce Proposition~\ref{lemma: second propagation formula} from~\eqref{eq: term A commutator and boundary terms}, 
 \eqref{eq: convergence measure interior commutator}
  and \eqref{eq: boundary terms A convergence}.
\end{proof}

\subsection{Properties and support of semiclassical measure }
\label{sec: properties microlocal defect measure}

Here we decompose the measure into an interior measure and a measure  supported at boundary. Moreover this last measure is supported
on $\xi_d=0$.
\begin{lemma} \label{lem: measure on boundary}
There exists a non negative Radon  measure $\mu^\d$ on $x_d=\xi_d=0$ such that $\mu=1_{x_d>0}\mu +\mu^\d\otimes \delta _{x_d=0}\otimes\delta_{\xi_d=0}$. 
Furthermore $\mu^\d$ is supported on $R(x',0,\xi')-1=0$.
\end{lemma}
\begin{proof}
We apply Proposition~\ref{lemma: second propagation formula}. We observe that $H_p=2\xi_d\d_{x_d}-(\d_{x_d}R)\d_{\xi_d}+H'_R$ where $H'_R=
\sum_{j=1}^{d-1}\big((\d_{\xi_j}R)\d_{x_j}-(\d_{x_j}R)\d_{\xi_j}\big)$. We have $H_p(\xi_d b)=2\xi_d^2(\d_{x_d}b) -(\d_{x_d}R) b+\xi_dH'_R b$. Let 
$b=b^\eps=\eps\chi(x_d/\eps)\ell(x,\xi')$, where $\chi\in\Con_0^\infty(\R)$, such that $\chi(0)=0$ and $\chi'(0)=1$, $\ell\in\Con_0^\infty(\R^d\times\R^{d-1})$. We have
\begin{align*}
H_p(\xi_d b^\eps)+2ab^\eps\xi_d& =2\xi_d^2\chi'(x_d/\eps)\ell(x,\xi')+2\eps \chi(x_d/\eps)
\xi_d^2\d_{x_d}\ell-(\d_{x_d}R)\eps\chi(x_d/\eps)\ell(x,\xi')
\\
&\quad+\xi_d\eps\chi(x_d/\eps)H'_R \ell
+2\eps a\chi(x_d/\eps)\ell(x,\xi') .
\end{align*}
Clearly $H_p(\xi_db^\eps)+2ab^\eps\xi_d$ is uniformly bounded on the support of $\mu$
and $H_p(\xi_d b^\eps)+2ab^\eps\xi_d\to 
2\xi_d^2\chi'(0)\ell 1_{x_d=0}$ everywhere as $\eps\to0$.
Then by Lebesgue's dominated convergence theorem we have $\est{\mu,H_p(\xi_db^\eps)+2ab^\eps\xi_d}\to\est{\mu, 2\xi_d^2 \ell 1_{x_d=0}}$ as $\eps\to0$. As $\chi(0)=0$, the 
right hand side of~\eqref{formula: second propagation formula} is 0 for every $h$. Then $\est{\mu, 2\xi_d^2 \ell 1_{x_d=0}}=0$. 
This means that $1_{x_d=0}\mu$ is supported on $\xi_d=0$. We denote $\mu^\d$ the measure $1_{\xi_d=0}1_{x_d=0}\mu$. 
As $\mu=1_{x_d>0}\mu+1_{x_d=0}\mu$, 
we have  $\mu=1_{x_d>0}\mu+\mu^\d\otimes \delta _{x_d=0}\otimes\delta_{\xi_d=0}$. We have by 
Proposition~\ref{prop: measure supported on the characteristic set}, $(\xi_d^2+R(x,\xi')-1)1_{x_d=0}\mu=0 $, 
then $(R(x',0,\xi')-1)\mu^\d=0$. This gives the conclusion of Lemma.
\end{proof}
The Hamiltonian of the interior measure is a priori a distribution of order one supported 
on $x_d=0$. The following lemma says that this quantity is a measure if the 
Hamiltonian vector field is transverse to the boundary.
\begin{lemma}\label{lem: Hp mu xd positive is a measure}
We assume that $a\mu=0$. 
There exists a distribution (of order 1) $\mu_0$ defined on $x_d=0$, such that $H_p(\mu 1_{x_d>0})=\delta_{x_d=0}\otimes \mu_0$.
Moreover, in a \nhd where $H_p x_d>0$, $\mu_0$ is a non negative Radon measure, and in a \nhd where $H_p x_d<0$, $\mu_0$ is a non positive Radon measure.
\end{lemma}
\begin{proof}
The support of $H_p(\mu 1_{x_d>0})$ as a distribution is $x_d\ge0$ and $H_p(\mu)=0$ on $x_d>0$. Then $H_p(\mu 1_{x_d>0})$ is supported on $x_d=0$.
This implies that there exist $n\ge0$ and $\mu_j$ distributions on $x_d=0$ such that $H_p(\mu 1_{x_d>0})=\sum_{j=0}^n\delta_{x_d=0}^{(j)}\otimes \mu_k$.
Suppose that $n\ge 1$  and 
let $\chi\in\Con_0^\infty(\R)$, be such that $\chi^{(k)}(0)=0$ for $k=0,\dots n-1$ and $\chi^{(n)}(0)=1$. Let 
$b\in \Con_0^\infty(\R^d\times\R^{d-1})$,
we have 
\begin{align}\label{eq: to calculate Hp mu}
-\est{\mu 1_{x_d>0},H_p\big(\eps^n\chi(x_d/\eps)b(x,\xi')\big)}& = \est{H_p\mu, \eps^n\chi(x_d/\eps)b(x,\xi')} \notag  \\
&=\est{\sum_{j=0}^n\delta_{x_d=0}^{(j)}\otimes \mu_k, \eps^n\chi(x_d/\eps)b(x,\xi')}= \est{\mu_n, b(x',0,\xi')}.
\end{align}
We also have 
\begin{align*}
H_p\big(\eps^n\chi(x_d/\eps)b(x,\xi')\big)= 2\xi_d\eps^{n-1}\chi'(x_d/\eps) b(x,\xi')+\eps^n\chi(x_d/\eps)H_p b(x,\xi'),
\end{align*}
which is bounded uniformly with respect $\eps$ and supported in a fixed compact set.
Then if $n\ge 2$,  $H_p\big(\eps^n\chi(x_d/\eps)b(x,\xi')\big)\to0$ everywhere as $\eps\to0$ 
and~\eqref{eq: to calculate Hp mu}, Lebesgue's dominated convergence theorem imply that $\mu_n=0$ for $n\ge2$.
If $n=1$,  $H_p \big(\eps\chi(x_d/\eps)b(x,\xi')\big)\to  2\xi_d 1_{x_d=0}b(x,\xi')$ everywhere as $\eps\to 0$. Lebesgue's dominated convergence 
theorem and~\eqref{eq: to calculate Hp mu} imply that
\begin{align*}
-\est{\mu 1_{x_d>0},2\xi_d 1_{x_d=0}b(x,\xi')}=\est{\mu_1,b(x',0,\xi')},
\end{align*}
as $ \est{\mu 1_{x_d>0},2\xi_d 1_{x_d=0}b(x,\xi')}=0$ we find that $\mu_1=0$. Then we have 
$  H_p(\mu 1_{x_d>0})=\delta_{x_d=0}\otimes \mu_0$, where $\mu_0$ is a distribution of order  1.

If  $H_px_d\not=0$, let $(x(s; x',\xi), \xi(s;x',\xi))$ the solution to $(\dot x,\dot\xi)=H_p$ satisfying the initial condition 
$(x(0; x',\xi), \xi(0;x',\xi))= (x',0,\xi)$. We verify that the map
$(s,x',\xi)\to (x(s; x',\xi), \xi(s;x',\xi))$ locally is one to one and transforms $\d_s$ in $H_p$. Moreover, $s=0$ is transformed in 
$x_d=0$ and if $H_p x_d>0$, $s>0$ is transformed in $x_d>0$, if  $H_p x_d<0$, $s<0$ is transformed in $x_d>0$. In coordinates $(s,x',\xi)$ the equation
$  H_p(\mu 1_{x_d>0})=\delta_{x_d=0}\otimes \mu_0$ is transformed in $\d_s (\mu 1_{s>0})=\delta_{s=0}\otimes \mu_0$ if $H_p x_d>0$ and 
$\d_s (\mu 1_{s<0})=\delta_{s=0}\otimes \mu_0$ if $H_p x_d<0$, where we keep the notations $\mu, \mu_0$ in variables  $(s,x',\xi)$  for the images of 
$\mu, \mu_0$. If $H_p x_d>0$, we have $\mu 1_{s>0}=(1_{s>0}ds)\otimes \mu_0$ and if $H_p x_d<0$, we have $\mu 1_{s<0}=-(1_{s<0}ds)\otimes \mu_0$. As 
$\mu$ is non negative, we obtain that  $\mu_0$ is a measure and its sign.
\end{proof}

At the hyperbolic region the measure $\mu_0$ has a particular structure given by this lemma. 
\begin{lemma}
	\label{lem: measure hyperbolic points}
We assume that $a\mu=0$.
Let $(x_0',\xi_0')$ be a hyperbolic point (i.e. $R(x'_0,0,\xi'_0)<1$). Locally in a \nhd of  $(x_0',\xi_0')$, there exist $\mu^+$ and $\mu^-$ non negative 
measures on $\R^{d-1}_{x'}\times \R^{d-1}_{\xi'}$ such that $\mu_0=\mu^+\otimes \delta_{\xi_d=\sqrt{1-R(x',0,\xi')}}-\mu^-\otimes \delta_{\xi_d=-\sqrt{1-R(x',0,\xi')}}$.
Moreover, if in a \nhd of $x_0'$,  $(v_h)_h$ satisfies the Dirichlet boundary condition or the Neumann boundary condition then $\mu^+=\mu^-$.

In a \nhd  of $\Gamma$ we have the following property, if $\mu^+=0$ (resp $\mu^-=0$) then 
 $\mu^-=0$ (resp $\mu^+=0$).
\end{lemma}
\begin{proof}
As $\mu $ is supported on $\xi_d^2+R(x,\xi')-1=0$, this implies that $\delta_{x_d=0}\otimes \mu_0$ is supported on  $\xi_d^2+R(x,\xi')-1=0$, 
then $ \mu_0$ is supported on $\xi_d=\pm \sqrt{1-R(x',0,\xi')}$. Moreover $H_p=2\xi_d\d_{x_d}+X$ where $X$ is a vector field tangent to $x_d=0$, 
in particular $H_p x_d=2\xi_d$. This implies that $H_p x_d>0$ if $\xi_d= \sqrt{1-R(x',0,\xi')}$ and  $H_p x_d<0$ if $\xi_d=- \sqrt{1-R(x',0,\xi')}$. 
From Lemma~\ref{lem: Hp mu xd positive is a measure}, we obtain 
$\mu_0=\mu^+\otimes \delta_{\xi_d=\sqrt{1-R(x',0,\xi')}}-\mu^-\otimes \delta_{\xi_d=-\sqrt{1-R(x',0,\xi')}}$. From 
Proposition~\ref{lemma: first propagation formula}
if $b_{|x_d=0}$ is supported on a part of the boundary such that $(v_h)_h$ satisfies the Dirichlet boundary condition (resp. Neumann boundary condition), then the right 
hand side of~\eqref{formula: first propagation formula} is 0. Then we have $\est{H_p\mu, b}=0$ which implies $\est{\mu^+-\mu^-,b_{|x_d=0}}=0$ 
as $b_{|x_d=0}$ describes every $\Con_0^\infty$ function supported in a \nhd of $(x_0',\xi_0')$, this implies that $\mu^+=\mu^-$ in a \nhd of  $(x_0',\xi_0')$.

Now to obtain the result in a \nhd of $\Gamma$, it suffices to take  
\(
b\in\Con_0^\infty(\R^d\times \R^{d-2})
\) 
positive. 
Here we use  notation defined above Formula~\eqref{def: de la notation R0 sur Gamma}.
As $\Con_0^\infty(x_1>0)$ is dense in $H^{1/2}(x_1\ge0)$ 
(see \cite[Theorem 11.1]{LM:1968}), we can approach in $H^{1/2}_{sc}$ (for $h$ fixed)
$(v_h)_{|x_d=0}$ by a sequence $(w_n)_n$ of smooth functions supported on $x_1> 0$.
We have 
\(
\big(  b(x',0,hD'') w_n | (hD_{x_d}v_h)_{|x_d=0}  \big)_0=0,
\)
by support properties and passing to the limit we have 
\(
\big(  b(x',0,hD'') (v_h)_{|x_d=0} | (hD_{x_d}v_h)_{|x_d=0}  \big)_0=0,
\)
for each $h$. As above we obtain $\est{\mu^+-\mu^-,b_{|x_d=0}}=0$. If $\mu^-=0$ we have
$\est{\mu^+,b_{|x_d=0}}=0$ for all $b$ independent of $\xi_1$ supported on a \nhd of 
a point $\rho$ of $\Gamma\times \R^{d-1}$, as $\mu^+$ is a non negative measure 
we obtain $\mu^+=0$ in a \nhd of $\rho$.
\end{proof}

We recall   that 
 ${\cal G}_d=\{ (x',\xi'),\ R(x',0,\xi')=1 \text{ and }\d_{x_d}R(x',0,\xi')<0  \}$ is the set of diffractive points.

The following lemma states that the diffractive points whose projection belong  to  $\d\Omega_D$ are not in the support of the measure.
%%%%%%%%%%%%%%%%%%%%%%%%%%%%%%%%%%
%
%   Lemma
%
%%%%%%%%%%%%%%%%%%%%%%%%%%%%%%%%%%
\begin{lemma}
	\label{lem: diffractive points non in measure support}
We assume $a\mu=0$. We have $1_{{\cal G}_d\cap (\d\Omega_{DN}\times\R^{d-1})}\mu^\d=0$, where \(\d\Omega_{DN}=\d\Omega_{D} \cup \d\Omega_{N}\) . 
\end{lemma}
\begin{proof}
We apply Proposition~\ref{lemma: second propagation formula}.
We have to choose an adapted function $b$. Let $\chi\in\Con^\infty$ be such that $\chi(\sigma)=0$ if $|s|\ge 2$, $ \chi(\sigma)=1$ if  $|\sigma|\le 1$.
We apply Proposition~\ref{lemma: second propagation formula} with 
$b(x, \xi')= \chi((1-R(x,\xi'))/\eps)\ell(x,\xi')\chi(x_d/\eps)$, where  $\eps>0$ will be chosen in what follows and $\ell$ 
is supported  in a \nhd of a point of $\d\Omega_{DN}\times\R\times\R^{d-1}$.
We recall that  $H_p=2\xi_d\d_{x_d}-(\d_{x_d}R)\d_{\xi_d}+H'_R$ (see the proof of Lemma~\ref{lem: measure on boundary}).
We have
\begin{align}
	\label{eq: calcul de Hp a xi d}
H_p(\xi_d b)&=  2\xi_d^2\Big(    -(\d_{x_d}R  ) \chi'((1-R(x,\xi'))/\eps)\ell(x,\xi')\chi(x_d/\eps)/\eps  \notag\\
&\quad + \chi((1-R(x,\xi'))/\eps)(\d_{x_d}\ell(x,\xi')) \chi(x_d/\eps) \notag \\
&\quad + \chi((1-R(x,\xi'))/\eps) \ell(x,\xi')\chi'(x_d/\eps)/\eps   \Big) \notag \\
&\quad  -(\d_{x_d}R  )  \chi((1-R(x,\xi'))/\eps)\ell(x,\xi')\chi(x_d/\eps) \notag \\
&\quad  + \chi((1-R(x,\xi'))/\eps)\chi(x_d/\eps)  \xi_d H'_R  \ell(x,\xi').
\end{align}
We 
claim 
\begin{align}\label{property: limit Hp a xi d}
&H_p(b \xi_d) \text{ is uniformly bounded on } \xi_d^2+R(x,\xi')=1,\notag \\
&H_p(b \xi_d)\to -(\d_{x_d}R(x',0,\xi'))1_{R(x',0,\xi')=1}1_{x_d=0}\ell (x',0,\xi')  \text{ as }\eps\to 0 \notag\\ 
&\text{ for all } (x,\xi), \text{ such that } \xi_d^2+R(x,\xi')=1.
\end{align}
As $\mu$ is supported on $\xi_d^2+R(x,\xi')-1=0$, then  $\xi_d^2/\eps=(1-R(x,\xi'))/\eps$ on the support of 
$\mu$, this implies that the three first terms 
in~\eqref{eq: calcul de Hp a xi d} are bounded. It is easy to prove that they converge to 0 as $\eps$ to 0.
The fourth term is bounded and converges to $-(\d_{x_d}R(x',0,\xi'))1_{R(x',0,\xi')=1}1_{x_d=0}\ell(x',0,\xi') $. 
In the last term as $|    R(x,\xi')-1 |/\eps$ is bounded, thus 
$|\xi_d|$ is bounded by $C\sqrt \eps$ and then this term converges to 0 as $\eps\to0$. This 
proves~\eqref{property: limit Hp a xi d}.
From that we can conclude that 
$\est{\mu, H_p(\ell \xi_d)}$ converges to 
\begin{multline}
	\label{Form: diffractive case, boundary measure}
\est{\mu, -(\d_{x_d}R(x',0,\xi'))1_{R(x',0,\xi')=1}1_{x_d=0}\ell (x',0,\xi') } 
\\ =\est{\mu^\d, -(\d_{x_d}R(x',0,\xi')) 1_{R(x',0,\xi')=1}\ell(x',0,\xi') } ,
\end{multline}
as $\eps\to 0$, which is a non negative term if $l$ is 
non negative and supported in a \nhd of a point of ${\cal G}_d$.

We now assume \(\ell\) supported on a \nhd of a point of \(\d\Omega_D\).
Let 
\begin{align}
	\label{eq: diffractive case boundary terms}
A_\eps&=  \lim_{h\to 0}\Big( 
\big( b(x',0,hD') \big(    R(x',0,hD')-1 \big)  (v_h)_{|x_d=0}|  (v_h)_{|x_d=0}  \big)_0  \notag \\
 &\quad  -   \big( b(x',0,hD')  (hD_{x_d}v_h)_{|x_d=0} |  (hD_{x_d}v_h)_{|x_d=0}  \big)_0
\Big)\notag \\
&=   -\lim_{h\to 0} \big( b(x',0,hD')  (hD_{x_d}v_h)_{|x_d=0} |  (hD_{x_d}v_h)_{|x_d=0}  \big)_0  ,
\end{align}
as $v_h$ satisfies the Dirichlet boundary condition.
We want to prove that $A_\eps\le0$.
By G\aa rding inequality~\eqref{eq: Garding inequality R d} 
(in fact used for $\R^{d-1}$ instead $\R^d$)
we have 
\begin{align*}
   -   \big( b(x',0,hD')  (hD_{x_d}v_h)_{|x_d=0} |  (hD_{x_d}v_h)_{|x_d=0}  \big)_0
 \le 
  C_{\eps}h | (hD_{x_d}v_h)_{|x_d=0}   |^2 _{L^2}  .
\end{align*}

Taking the limit as $h\to 0$, Proposition~\ref{prop: convergence traces D and N}  
   implies that $A_\eps\le 0$. As \eqref{Form: diffractive case, boundary measure} is non negative,
this prove that 
\begin{align}
	\label{form: diffractive boundary measure null}
\est{\mu^\d, -(\d_{x_d}R(x',0,\xi'))1_{R(x',0,\xi')=1}\ell(x',0,\xi') } =0.
\end{align}
Then $1_{{\cal G}_d}\mu^\d=0$ on $\d\Omega_D$ as $\ell$ can be arbitrary chosen.

We now assume \(\ell\) supported on a \nhd of a point of \(\d\Omega_N\). With \(A_\eps\) defined
in \eqref{eq: diffractive case boundary terms} we have, 
as \((hD_{x_d}v_h)_{|x_d=0}=0\) on the support of \(\ell\)
\begin{align*}
A_\eps&=  \lim_{h\to 0}
\big( b(x',0,hD') \big(    R(x',0,hD')-1 \big)  (v_h)_{|x_d=0}|  (v_h)_{|x_d=0}  \big)_0  \\
&=  \lim_{h\to 0}\big(  \op_{sc}\big( \chi((1-R(x',0,\xi'))/\eps)\ell(x', 0,\xi') 
(    R(x',0,h\xi'')-1) \big)   (v_h)_{|x_d=0}|  (v_h)_{|x_d=0}  \big)_0 .
\end{align*}
From Proposition~\ref{prop: boundary term diffractive Neumann}, \(A_\eps\) goes to 0 as 
\( \eps \to 0\). Then from \eqref{Form: diffractive case, boundary measure} and 
Proposition~\ref{lemma: second propagation formula} we obtain 
\eqref{form: diffractive boundary measure null} in this case and \(1_{{\cal G}_d}\mu^\d=0\) 
on \(\d\Omega_N\).
\end{proof}

This lemma describes how the support of the boundary measure propagates along the boundary.
%%%%%%%%%%%%%%%%%%%%%%%%%%%%%%%%
%
%    Lemma
%
%%%%%%%%%%%%%%%%%%%%%%%%%%%%%%%%
\begin{lemma}  
	\label{lem: propagation on boundary if mu0 supported on xi-d equal 0}   
Let $(x'_0,\xi'_0)\in T^*\d \Omega$ be such that $x_0'\in \d\Omega_D\cup  \d\Omega_N$. 
Let $(x,\xi)\in T^*\Omega$,
 we denote 
$\gamma_{(x,\xi)}$ the integral curve of $H_p$ starting from $(x,\xi)$.
We assume that $\mu$ is locally supported in a \nhd of $(x'_0,0, \xi'_0, 0)$,  in the set $\{(x,\xi), \ \xi_d^2+ R(x,\xi')=1, \text{ such that } 
 \gamma_{(x,\xi)} \text{ hits }x_d=0,\    \ \xi_d=0    \}$. 
  In particular 
$\mu_0$ is supported 
on $\xi_d=0$ and  $\mu_0= \tilde \mu_0\otimes  \delta_{\xi_d=0}
+ \tilde \mu_1\otimes  \delta_{\xi_d=0}'$, where $\tilde \mu_0$ is a distribution
and $  \tilde \mu_1 $ is a  Radon measure.  
Then $H'_R\mu^\d-2a\mu^\d+\tilde\mu_0=0$, $\tilde\mu_1=0$ and $H_p\mu=-\d_{x_d}R(x',0,\xi') \mu^\d\otimes  \delta_{x_d=0}\otimes \delta_{\xi_d=0}'
+2a\mu^\d\otimes  \delta_{x_d=0}\otimes \delta_{\xi_d=0} $, in a \nhd of $(x'_0,0, \xi'_0, 0)$.
\end{lemma}
\begin{proof}
We have $\mu=1_{x_d>0}\mu+\mu^\d\otimes  \delta_{x_d=0}\otimes \delta_{\xi_d=0}$. Then we have
\begin{align}
H_p\mu&= \mu_0\otimes  \delta_{x_d=0} + 2\xi_d \mu^\d  \otimes\delta_{x_d=0}' \otimes \delta_{\xi_d=0} -(\d_{x_d} R(x',0,\xi')) 
\mu^\d  \otimes \delta_{x_d=0}\otimes \delta_{\xi_d=0}' \notag\\
&\quad +H'_R\mu^\d \otimes\delta_{x_d=0}\otimes \delta_{\xi_d=0} \notag\\
&=  \mu_0\otimes  \delta_{x_d=0}-(\d_{x_d} R(x',0,\xi')) 
\mu^\d  \otimes \delta_{x_d=0}\otimes \delta_{\xi_d=0}'  +H'_R\mu^\d\otimes \delta_{x_d=0}\otimes \delta_{\xi_d=0},
\label{eq: Hp form propagation boundary D and N bis}
\end{align}
as $\xi_d \mu^\d  \delta_{x_d=0}' \otimes \delta_{\xi_d=0}=0$. 
 As 
$\mu_0$ is of order 1 and supported on $\xi_d=0$,  we have $\mu_0= \tilde \mu_0\otimes  \delta_{\xi_d=0}
+ \tilde \mu_1\otimes  \delta_{\xi_d=0}'$, where $\tilde \mu_j$ are distributions. 

To prove that $\tilde \mu_1$ is a Radon measure, we test $\mu_0$ on 
$ \varphi=\xi_d\psi(x',\xi')\chi(\xi_d/\eps)$. The first derivative of $\varphi$ are estimated 
by supremum of 
$\psi(x',\xi')\chi(\xi_d/\eps)$, $\psi(x',\xi')\chi'(\xi_d/\eps)\xi_d/\eps$ and
$\xi_d\d \psi(x',\xi')\chi(\xi_d/\eps)$. When $\eps\to 0$ the supremum on $\varphi $ is 
estimated by supremum of $\psi(x',\xi')$. We also have $\est{\mu_0,\varphi}$ converging to
$\est{ \tilde \mu_1,\psi(x',\xi') } $ as $\eps\to 0$. We deduce that $\est{ \tilde \mu_1,\psi(x',\xi') } $ is estimated by the supremum of $\psi(x',\xi')$, this implies that $ \tilde \mu_1$ is a Radon measure.

 By Proposition~\ref{lemma: first propagation formula} and if 
 $b\in\Con_0^\infty(\R^d\times\R^{d-1})$ is supported in 
 $\d\Omega_D\cup \d\Omega_N$ and in a \nhd of $(x'_0,0, \xi'_0, 0)$, 
 we have $\est{H_p\mu-2a\mu, b(x,\xi')}=0$. 
 
Let $b\in\Con_0^\infty(\R^d\times \R^{d-1})$, be such that $b_{|x_d=0}=\ell\in\Con_0^\infty(\R^{d-1}\times \R^{d-1})$, and $\chi\in \Con_0^\infty(\R)$, 
be such that $\chi(\sigma)=1$ in a \nhd of 0.
From~\eqref{eq: Hp form propagation boundary D and N bis} we obtain 
\begin{equation*}
 \est{H_p\mu-2a\mu , b\chi(x_d/\eps)}= \est{\tilde \mu_0
 + H'_R \mu^\d-2a\mu^\d, \ell}-2\est{a1_{x_d>0}\mu, 
b \chi(x_d/\eps) }.
\end{equation*}
As $b \chi(x_d/\eps)$ is uniformly bounded and $b\chi(x_d/\eps)$ converges to $b1_{x_d=0}$ everywhere, we obtain 
\begin{equation*}
 \est{a1_{x_d>0}\mu, b \chi(x_d/\eps) }\to  \est{a1_{x_d>0}\mu, b 1_{x_d=0}}= 0 \text{ as } \eps\to 0.
\end{equation*}
We deduce 
 that 
$H'_R\mu^\d- 2a\mu^\d+\tilde\mu_0=0$, which gives the first conclusion of Lemma. 
We deduce from~\eqref{eq: Hp form propagation boundary D and N bis}
\begin{align}\label{eq: formula mu propagation boundary D and N bis}
H_p\mu  = ( \tilde \mu_1- (\d_{x_d} R(x',0,\xi')) 
\mu^\d ) \otimes   \delta_{x_d=0}  \otimes     \delta_{\xi_d=0}'  +2a\mu^\d\otimes  \delta_{x_d=0}\otimes \delta_{\xi_d=0} .  
\end{align}
We then can write
\begin{align}
	\label{eq: Hp propagation boundary ter}
\est{\mu,H_pb(x,\xi)}=\est{ \tilde \mu_1- (\d_{x_d} R(x',0,\xi')) 
\mu^\d , \d_{\xi_d}b(x',0,\xi', 0)}  - \est{ 2a\mu^\d, b(x',0,\xi', 0) },
\end{align}
for $b\in\Con_0^\infty (\R^d\times\R^d)$. 

Now we choose an adapted $b$ to apply~\eqref{eq: Hp propagation boundary ter}.
 Let $\chi\in\Con_0^\infty(\R)$ be such that $\chi(\sigma)=1$ for 
$\sigma$ in a \nhd of 0. Let $\ell\in\Con_0^\infty(x',\xi') $ be supported
in a \nhd of $(x'_0,\xi'_0)$. We set $b(x,\xi)= \xi_d \ell (x',\xi')\chi(\xi_d/\eps)\chi(x_d/\eps)$, where $\eps>0$.
We have
\begin{align*}
H_p b(x,\xi)&= 2 \ell(x',\xi')\chi(\xi_d/\eps)\chi'(x_d/\eps) \xi_d^2/\eps   -  \ell(x',\xi')( \d_{x_d}R(x,\xi'))\chi(x_d/\eps)   
 \big( \chi(\xi_d/\eps)+  \chi'(\xi_d/\eps)  \xi_d /\eps \big) \\  
 &\quad +\xi_d  H'_R( \ell(x',\xi'))\chi(\xi_d/\eps)\chi(x_d/\eps).
\end{align*}
As $\chi'(x_d,\eps)\chi(\xi_d/\eps) \xi_d^2/\eps $, $ \chi(x_d/\eps)\chi(\xi_d/\eps) \xi_d/\eps $ and $ \xi_d\chi(\xi_d/\eps)\chi(x_d/\eps)$ are 
uniformly bounded and goes to 0 as $\eps $ goes to 0, and $ \chi(\xi_d/\eps)\chi(x_d/\eps)  $ goes to $1_{x_d=0,\xi_d=0}$ and is uniformly 
bounded. 
We have $\est{\mu , H_p b}$ goes to $\est{\mu, -1_{x_d=0,\xi_d=0}(\d_{x_d}R(x',0,\xi') )\ell(x',\xi')}$ as $\eps$ goes to 0 and we have 
\begin{align} \label{eq: calculus mu1 tilde bis}
\est{\mu, -1_{x_d=0,\xi_d=0}\d_{x_d}R(x',0,\xi') \ell(x',\xi')}=-\est{\mu^\d,(\d_{x_d}R(x',0,\xi') ) \ell(x',\xi')}.
\end{align}
The term $\est{ 2a\mu^\d, b(x',0,\xi', 0) }$ goes to 0 as $\eps$, as $b$ goes to 0 uniformly as $\eps$ goes to 0.
 Now we compute the limite as $\eps\to 0$ of $\est{ \tilde \mu_1, \d_{\xi_d}b(x',0,\xi', 0)} $.
 We have   
 \begin{align*}
  \d_{\xi_d}b(x',x_d,\xi', \xi_d) = \ell(x',\xi') \chi(\xi_d/\eps)\chi(x_d/\eps)+  \ell(x',\xi') \chi'(\xi_d/\eps)\chi(x_d/\eps)\xi_d/\eps.
 \end{align*}
 Then  $  \d_{\xi_d}b(x',0,\xi', 0) = \ell(x',\xi')  $ and  
 \begin{align*}
  \est{ \tilde \mu_1    - (\d_{x_d} R(x',0,\xi')) \mu^\d
  , \d_{\xi_d}b(x',0,\xi', 0)} &= \est{ \tilde \mu_1  - (\d_{x_d} R(x',0,\xi')) \mu^\d,\ell(x',\xi')}\\
  &= -\est{\mu^\d,(\d_{x_d}R(x',0,\xi') ) \ell(x',\xi')}
 \end{align*}
 from~\eqref{eq: Hp propagation boundary ter} and 
  \eqref{eq: calculus mu1 tilde bis}. We deduce that $\tilde\mu_1=0$ and the last result of 
  the lemma  
 from~\eqref{eq: formula mu propagation boundary D and N bis}.
\end{proof}
We have an analogous result to Lemma~\ref{lem: propagation on boundary if mu0 supported on xi-d equal 0}  in a 
\nhd of Zaremba condition. To be precise we recall the notations defined in 
Formula~\eqref{def: de la notation R0 sur Gamma},
we have
$R(x',0,\xi')=\xi_1^2+R_0(x'',\xi'')+x_1r_2(x,\xi')$, where $x'=(x_1,x'')$ and $\xi'=(\xi_1,\xi'')$, 
 $R_0\in S(\est{\xi''}^2,(dx'')^2+\est{\xi''}^{-2}(d\xi'')^2)$ and $r_2\in S^2_{\rm tan}$.
 %%%%%%%%%%%%%%%%%%%%%%%%%%
 %
 %   Lemma
 %
 %%%%%%%%%%%%%%%%%%%%%%%%%%%
 \begin{lemma}
	\label{lem: propagation boundary Zaremba}
We assume that $(v_h)_h$ satisfies the boundary Zaremba condition \nhd of $(0,x_0'')$.
Let $\ell \in\Con_0^\infty(\R^{d-1}_{x'}\times \R^{d-2}_{\xi''})$, where $\xi'=(\xi_1,\xi'')$. Then $\est{H'_R \mu^\d-2a\mu^\d+\mu_0, \ell}=0$.
In particular if $\mu^\d$ is a measure supported on $x_1=\xi_1=0$, this means that  $\mu^\d=\tilde\mu^\d\otimes\delta_{x_1=\xi_1=0}$ 
and if $\mu_0$ is supported on $x_1=\xi_1=\xi_d=0$, this means that there exists  $\tilde\mu_0$, 
and $\tilde\mu_{\alpha, \beta}$   distributions of order 1  
 on $\R^{d-1}_{x''}\times \R^{d-1}_{\xi''}$, for $\alpha=(0,0,0)$ or  $(1,0,0)$ and  $\beta\in \{ (j,0,k),\ j,k=0 \text{ or }1\}$ such that 
$\mu_0=\tilde\mu_0\otimes \delta_{x_1=\xi_1=\xi_d=0} +\sum_{|\alpha|+|\beta|=1} 
\tilde\mu_{\alpha,\beta}\otimes\d_{x}^\alpha\d_{\xi}^\beta\delta_{x_1=\xi_1=\xi_d=0}$. Then $H_{R_0}''  \tilde\mu^\d-2a\tilde\mu^\d+\tilde\mu_0=0$, where 
 $H_{R_0}'' =\sum_{2\le j\le d-1}\big(  \d_{\xi_j}R_0(x'',\xi'')\d_{x_j}- \d_{x_j}R_0(x'',\xi'')\d_{  \xi_j}
\big)$.
\end{lemma}
\begin{proof}
Let $b\in\Con_0^\infty(\R^d\times \R^{d-2})$ be such that $b_{|x_d=0}=\ell$.
We apply
Proposition~\ref{lemma: first propagation formula}  to $ b(x,\xi'')$. 
As in the end of the proof of Lemma~\ref{lem: measure hyperbolic points} we  have 
\(
\big(  \ell(x',hD'') (v_h)_{|x_d=0} | (hD_{x_d}v_h)_{|x_d=0}  \big)_0=0.
\)
Then we have $\est{H_p\mu-2a\mu, b(x,\xi'')}=0$. 
We follow the same ideas of the proof of 
Lemma~\ref{lem: propagation on boundary if mu0 supported on xi-d equal 0}.
From Formula~\eqref{eq: Hp form propagation boundary D and N bis} we have
\begin{align*}
\est{H'_R\mu^\d-2a\mu^\d+\mu_0, \ell}-2\est{a 1_{x_d>0}\mu, b}=0.
\end{align*}
Let 
\(
\chi\in\Con_0^\infty(\R)
\)
be such that 
\(\chi(\sigma)=1\)
in a \nhd of $0$.
Taking $b_\eps(x,\xi'')=\chi(x_d/\eps) \ell(x',\xi'')$ and letting $\eps$ goes to 0, we obtain
 $\est{H'_R\mu^\d-2a\mu^\d+\mu_0, \ell}=0$.
We have $H'_R=2\xi_1\d_{x_1}  +H''_{R_0}+ H_{x_1r_2}'$, and $H_{x_1r_2}' 
\ell(x',\xi'')=x_1H'_{r_2}\ell(x',\xi'')$, as $\ell$ independent 
of $\xi_1$.
We deduce from the form of $\mu^\d$ that 
\(
H'_R\mu^\d=  H''_{R_0}\tilde\mu^\d \otimes\delta_{x_1=\xi_1=0}.
\)
Taking 
\(
\ell(x',\xi'')=\chi(x_1)\tilde\ell(x'',\xi''),
\)
we deduce from the form of $\mu_0$
\(
\est{H'_R\mu^\d-2a\mu^\d+\mu_0, \ell}
 = \est{ H''_{R_0}\tilde\mu^\d-2a\tilde\mu^\d+ \tilde\mu_0, \ell}=0.
\)
This implies the result.
\end{proof}

\section{Support propagation results}
\label{sec: propagation proofs}

In this section we prove propagation of support of semiclassical measure under the assumption 
mGCC, see Definition~\ref{def: mGCC}.
\begin{proposition}
	\label{prop: propagation support semiclassical measure}
We assume that 	\( P\),  \(a(x)\) and  \(\Omega\) satisfy mGCC. Let  \(\mu\) the semiclassical
measure constructed from \((v_h)_h\) and satisfying \eqref{def: semiclassical defect measure}, 
we have \(\mu=0\).
\end{proposition}
We prove the propagation result, first in interior which is a classical result, second in a 
\nhd of a point on 
the boundary with Dirichlet 
 or Neumann conditions and third in a \nhd of a point on $\Gamma$.

\subsection{Propagation in interior domain $\Omega$}

Here we use the fact that $a\mu=0$.
From Proposition~\ref{prop: interior formula for Hp mu}, we have $H_p\mu=0$. It is then classical that $\mu $ in invariant by 
the flow of $H_p$. More precisely, let $\rho_0\in T^*\Omega$ and  we assume that $\gamma(s,\rho_0)\in  T^*\Omega$ for $s\in[0,t]$.
Let  $b\in\Con^\infty_0(\Omega\times\R^d)$ be such that  $b(\gamma(-s,.))$ is supported in $T^*\Omega$ for $s\in[0,t]$,
we have $\est{\gamma_*(t,.)\mu,b}=\est{\mu,b(\gamma(-t,.))}=\est{\mu,b}$.

 \subsection{Propagation at boundary: hyperbolic points}
 
 The propagation results given in this section are classical for Dirichlet boundary condition. 
 We prove a propagation result for Neumann boundary and in a \nhd of $\Gamma$
 which is new in context of semiclassical measure.
 We use the geometry context defined in section~\ref{sec: Geometry}, in particular $j$ and 
 the definition of the different flows.
 
 We prove that  the support of measure is locally empty in the future  assuming that in the past the support 
 of measure is locally empty  but by symmetry we can deduce that the support of measure is locally empty in the past if we assume that the 
 support of measure is locally empty in the future.
 
 Recall that we choose coordinates such that $p(x,\xi)=\xi_d^2+R(x,\xi')-1$ and locally $\Omega=\{ x_d>0 \}$. In this section we use that 
 $a\mu=0$.

 We recall that a point $(x_0',\xi_0')\in T^*\d\Omega$ is in $\cal H$, if $R(x_0',0,\xi_0')-1<0$. We apply  
 Lemmas~\ref{lem: measure on boundary} 
 and \ref{lem: measure hyperbolic points}. We have $\mu=1_{x_d>0} \mu$ and $H_p \mu=  
 \mu^+\otimes \delta_{\xi_d=\sqrt{1-R(x',0,\xi')}}-\mu^-\otimes \delta_{\xi_d=-\sqrt{1-R(x',0,\xi')}}$ with $\mu^+=\mu^-$. 
 
 We call $\gamma^\pm$ 
 the integral curve of $H_p$   starting 
 from  $ (x_0',x_d,\xi'_0,\pm\sqrt{1-R(x',0,\xi')})$.  If we assume that  the support of $\mu=1_{x_d>0} \mu$ is empty in a \nhd of 
 $ \gamma^-$(s) for $s<0$ and $|s|$ sufficiently small, then $\mu^-=0$. 
 This implies $\mu^+=0$ and $H_p\mu=0$. 
 As $\mu_{|x_d<0}=0 $ 
 and $\gamma^+(s)$ is in $x_d<0$ for $s<0$ and $|s|$ sufficiently small, this implies that  $\mu=0$ in a \nhd of $\gamma^+(0)$.

 \begin{remark}
 In a \nhd of points in 
 \(
 \d\Omega_D\cup\d\Omega_N
 \)
 we can prove a propagation of measure because we have proved $\mu_+=\mu_-$ but 
 we do not know if this property is true for points in $\Gamma$.
 \end{remark}

 \subsection{Propagation at $\d\Omega_D \cup \d\Omega_N$}
 
 \subsubsection{Propagation at   gliding points}
 
We recall that a point $(x_0',\xi_0')\in T^*\d\Omega$ is in ${\cal G}_g $, if $R(x_0',0,\xi_0')-1=0$ and 
$\d_{x_d}R(x_0',0,\xi_0')>0$.  
Let $\gamma$ be the integral curve of $H_p$ starting from $(x'_0,0,\xi_0',0)$.  Then  $\gamma(s)$ 
into $\{x_d<0\}$ for $s\not=0$ and $|s|$ sufficiently small.
In a \nhd of $(x_0',\xi_0')$ in $  T^*\d\Omega$ all the point are either hyperbolic, or  gliding.
We assume that $j^{-1}\big( \gamma_g(s_0;x_0',\xi_0')\big) \cap \supp \mu =\emptyset $   for $s_0<0$ where $|s_0|$ 
is sufficiently small.
Here $\gamma_g(s;x_0',\xi_0')=\Gamma(s; x_0',\xi_0')$, then all the point $\rho$ in a 
\nhd of $ j^{-1}\big( \gamma_g(s_0;x_0',\xi_0')\big)$ 
are not in the support of $\mu$.  By continuity of $\Gamma$ the curve $\Gamma(s;\rho)$ hit the boundary at $\rho'$
 in a \nhd of $(x_0',\xi_0')$. If $\rho'$  is an hyperbolic point, by the previous result  the point $j^{-1}\big( \Gamma(s;\rho)\big)$ 
 are not in the support of $\mu$. If $\rho'$ is a  gliding point, all the points $\Gamma(s;\rho)$ are strictly gliding. In particular this
 implies that $\mu$ is supported on $x_d=0$, then $1_{x_d>0}\mu=0$ and $\mu_0=0$.
We can apply 
Lemma~\ref{lem: propagation on boundary if mu0 supported on xi-d equal 0} and $\mu^\d$ satisfied
$H'_R\mu^\d=0$. Let $\gamma_g$ be the integral curve of $H'_R$ starting from $(x_0',\xi_0')$. 
As by assumption $\mu^\d  \otimes \delta _{x_d=0}\otimes\delta_{\xi_d=0} =\mu$ is 0 in a 
\nhd of $j^{-1}\big( \gamma_g(s_0;x_0',\xi_0')\big)$, we have $\mu^\d=0$ in a \nhd of $\gamma_g(s_0)$ and 
$H'_R\mu^\d=0$, this implies that $\gamma_g(s)$ is not in the support of $\mu^\d$ in a \nhd of $s=0$. As 
$\mu=\mu^\d\otimes \delta _{x_d=0}\otimes\delta_{\xi_d=0}$ we have $\mu=0$ in a \nhd of $j^{-1} (x_0',\xi_0')$.

 \subsubsection{Propagation at  diffractive points}

We recall that a point $(x_0',\xi_0')\in T^*\d\Omega$ is in ${\cal G}_g $, 
if $R(x_0',0,\xi_0')-1=0$ and
 $\d_{x_d}R(x_0',0,\xi_0')<0$.  We keep the previous notation 
for $\gamma$. For a point $\rho$ in a \nhd of $(x_0',\xi_0')\in T^*\d\Omega\cup T^*\Omega$,
there are three cases, first the integral curve  passing through $\rho$  hits $x_d=0$ 
at an hyperbolic point and by previous result the integral
curve is not in support of $\mu$, second it does not hit $x_d=0$ and the integral curve is not 
in the support of 
$\mu$ by propagation result in interior, third the integral
curve hits $x_d=0$ at a diffractive point.
Then the support of $\mu^\d$ is in ${\cal G}_g $ and the support of $1_{x_d>0}\mu$ is into $\{(x,\xi), \ \xi_d^2+ R(x,\xi')=1, \text{ such that } 
 \gamma_{(x,\xi)} \text{ hits }x_d=0,\    \ \xi_d=0    \}$. 

If $(x'_0,0)\in\d\Omega_D\cup\d\Omega_N$ we can apply
Lemma~\ref{lem: diffractive points non in measure support}, 
then
 $\mu^\d=0$  in a \nhd of \((x_0',0,\xi_0')\). As the  integral curves hitting $x_d=0$ at an 
hyperbolic points are not in the support of $1_{x_d>0}\mu$, then
$\mu_0$ is supported on $\xi_d=0$. We can apply 
Lemma~\ref{lem: propagation on boundary if mu0 supported on xi-d equal 0} to obtain $\mu_0=0$. Then
$H_p\mu=H_p(1_{x_d>0}\mu)=0$ and as, by assumption, $\gamma(s)$ is not in the support of $\mu$ for $s<0$, $|s|$ sufficiently small, 
we deduce that $\gamma(0)$ is not in the support of $\mu$.

 \subsubsection{Propagation at boundary: integral curves with high contact order}
 
\medskip
We recall that if $(x'_0,\xi'_0) $ is such that $R(x_0', 0, \xi'_0, 0)=1$, 
$\d_{x_d}R(x_0', 0, \xi'_0, 0)=0$ and if we denote by 
$\gamma(s)=(x'(s),x_d(s), \xi'(s),\xi_d(s))$ the integral curve of $H_p$ starting from 
$(x_0', 0, \xi'_0, 0)$. By the assumption made (see Definition~\ref{def: contact fini}) 
there exist \(k\in\N\), \(k\ge 3\) and \(\alpha\ne0\) such that
$x_d(s) = \alpha s^k+{\cal O}(s^{k+1})$.  We denote $\gamma_g(s)=(x'_g(s),\xi'_g(s))$ the 
integral curve of $H'_R$, starting from $(x'_0,\xi'_0) $.
For each $k$ we assume that we  have already proved that the integral curves hitting $x_d=0$ 
at a point in ${\cal G}^j$ for $j<k$ or $\cal H$ are not 
in the support of $\mu$.

\paragraph{Case $k$ even, $\alpha<0$} 
The integral curve of $H_p$ starting from a point belonging to $  T^*\Omega$  in a \nhd of  $(x'_0,\xi'_0) $  in $T^*\d\Omega\cup T^*\Omega$ 
eventually hits $x_d=0$ at a point $\rho'$, 
in $\cal H$ or ${\cal G}^j$ for 
$j\le k$ (see Section~\ref{sec: Geometry} for definition of ${\cal G}^j$). 
By assumptions and by induction this integral curve is not in the support of $\mu$ except if 
$\rho'$ is 
in ${\cal G}^k$, but in this case this integral curve is in $x_d\le 0$. 
This implies that $1_{x_d>0}\mu=0$, then $\mu_0=0$. 
By Lemma~\ref{lem: propagation on boundary if mu0 supported on xi-d equal 0}, 
we have $H'_R\mu^\d=0$ and as, by assumption, 
$\gamma_g(s)$ is not in the support of $\mu^\d$ for $s<0$, $|s|$ sufficiently small, 
we deduce that 
$\gamma_g(0)$ is not in the support of $\mu^\d$.
 
\paragraph{Case $k$ odd, $\alpha<0$}  
By the same argument as in previous case, the integral curve of $H_p$ starting from a point belonging to $  T^*\Omega$ 
in a \nhd of  $(x'_0,\xi'_0) $ in $T^*\d\Omega\cup T^*\Omega$
hits ${\cal G}^k $ or is not in the support of $\mu$. Denote by $\rho'$ the point of this integral curve hitting $x_d=0$. 
The generalized bicharacteristic starting 
from $\rho'$ is on $x_d=0$ for $s>0$ and in $x_d>0$ for $s<0$, and for $s>0$ all the points on 
the integral curve of $H'_R$ are in ${\cal G}_g$, if $|s|$ is sufficiently small. As, by assumption
the  generalized bicharacteristic is not in support of $\mu $ in the past, this means that $1_{x_d>0}\mu=0$ then $\mu_0=0$. 
We can apply Lemma~\ref{lem: propagation on boundary if mu0 supported on xi-d equal 0} then 
$H'_R\mu^\d=0$. But $\gamma_g(s)$ is not in 
support of $\mu=\mu^\d$ for $s<0$ and $|s|$ sufficiently small as $\gamma_g(s)\in{\cal G}_d$, then $\gamma_g(0)$ is 
not in the support of $\mu^\d$.
 
\paragraph{Case $k$ even, $\alpha >0$}

By induction, only the generalized bicharacteristics with the same order of contact  $k$ and the same sign condition 
$\alpha>0$ can be in the support of 
$\mu$. 
Applying Lemma~\ref{lem: propagation on boundary if mu0 supported on xi-d equal 0} we have 
\(
-\d_{x_d}R(x',0,\xi') \mu^\d\otimes  \delta_{x_d=0}\otimes \delta_{\xi_d=0}'=0
\)
as by induction, $ \mu^\d=0$ when $\d_{x_d}R(x',0,\xi')\ne 0$. We deduce
 $H_p\mu=0$. Then the propagation the support of $\mu$ is invariant by the flow of $H_p$.

 \paragraph{Case $k$ odd, $\alpha>0$}

By induction, only the generalized bicharacteristics with the same order of contact  $k$ and the same sign condition 
$\alpha>0$ can be in the support of 
$\mu$.  We can apply
Lemma~\ref{lem: propagation on boundary if mu0 supported on xi-d equal 0}, and by the same argument used in the previous case, we have $H_p\mu=0$ and as 
$\gamma(s)$ is in $x_d<0$ for $|s|<0$ sufficiently small, $\gamma(s)$  is not in the support of $\mu $ for $s<0$ 
and by propagation $\gamma(0)$ is not in support of $\mu$.

\begin{proof}[Proof of Proposition~\ref{prop: propagation support semiclassical measure}, first case]
By assumption for a point 
\(
\rho \in  T_b\Omega=T^*\Omega\cup T^*\d\Omega
\)
with 
\(
\pi(\rho)\in \Omega\cup \d\Omega_D\cup  \d\Omega_N, 
\)
we have assumed that 
\(
\pi(\Gamma(s_0,\rho) )\in \{x\in\overline{\Omega}, a(x)>0\}
\)
for some 
\(s_0\in \R \) and for every 
\(s\in [0,s_0], \)
if
\(\pi \Gamma(s,\rho)\in \Gamma\)
then 
\(
 \Gamma(s,\rho)\in \cal H.
\)
As 
\(
\supp\mu 
\)
is a closed set, if 
\(
\rho \in\supp\mu
\)
there exist \(s_1\in [0,s_0]\) such that 
\(
j^{-1}\Gamma(s_1,\rho)\cap \supp\mu\ne\emptyset
\) 
and 
\(
j^{-1}\Gamma(s,\rho)\cap \supp\mu=\emptyset
\) 
for 
\( s\in[s_0,s_1).\)
At 
\(
\Gamma(s_1,\rho) 
\)
we can apply the results obtained in this section to prove that 
\(
j^{-1}\Gamma(s_1,\rho)\cap \supp\mu=\emptyset,
\)
and reach a contradiction.
\end{proof}

 \subsection{Propagation on $\Gamma$} \label{sec: Propagation in a nhd of Gamma}

 Now we prove Proposition~\ref{prop: propagation support semiclassical measure} in the 
 second case, i.e.  \(\pi(\rho)\in\Gamma\).
  We recall that we can change the coordinates such that locally in a \nhd of $\Gamma$ 
 we have $\Omega=\{ x_d>0 \}$ and 
 $\Gamma=\{ x_d=x_1=0\}$, moreover, in the coordinates $(x_1,x'',x_d)=(x',x_d)$, we 
 have $R(x',0,\xi')= \xi_1^2+R_0(x'',\xi'')+x_1r_2(x',\xi')$.
   By the result obtained from the previous section, 
  \(
  \supp \mu\subset \{x_1=0, x_d=0, \ \xi_d=0\},
  \)
  in particular 
  \(1_{x_d>0}\mu=0\)
  and \(\mu_0=0.\)
Let $\rho=(x'_0,\xi'_0) \in T^*\d \Omega$ such that $x'_0\in\Gamma$.
Then the measure verifies $\mu=\mu^\d\otimes  \delta_{x_d=0}\otimes  \delta_{\xi_d=0}$ 
and $ \mu^\d$ is supported on $x_1=0$.

We have to distinguish two cases, first if $R_0(x'',\xi'')-1<0$ and second  $R_0(x'',\xi'')=1$ even if the result is the same in both cases.

 \subsubsection{Case $R_0(x'',\xi'')-1<0$}
 
 Let $\xi_1(x'',\xi'')$ the positive solution in $\xi_1$ of $\xi_1^2+R_0(x'',\xi'')-1=0$. 
 There exist $\mu^\pm(x'',\xi'')$ measures such that $\mu^\d= \mu^+(x'',\xi'')\otimes  \delta_{x_1=0}\otimes \delta_{\xi_1=\xi_1(x'',\xi'') } 
 + \mu^-(x'',\xi'')\otimes  \delta_{x_1=0}\otimes \delta_{\xi_1=-\xi_1(x'',\xi'') }$. 
 Lemma~\ref{lem: propagation boundary Zaremba} implies that 
 $\est{H'_R\mu^\d,\ell(x',\xi'')}=0$.

 Let $S(x',\xi')=R_0(x'',\xi'')+x_1r_2(x',\xi')$,
 we have on $ x_d=\xi_d=0$, 
 \begin{equation}
 \label{eq: hamiltonian boundary gamma}
  H'_R=2\xi_1\d_{x_1}-(\d_{x_1}S(x',\xi'))\d_{\xi_1}+H''_R,
 \end{equation}
  where $H''_R=\sum_{j=2}^{d-1}  
 \big(
 \d_{\xi_j}R(x',0,\xi')\d_{x_j}- \d_{x_j}R(x',0,\xi')\d_{\xi_j}
\big)$. Observe that on $x_1=0$, $H''_R=H''_{R_0}$.
 \begin{align*}
 H'_R\mu^\d&= 2\xi_1(x'',\xi'')\mu^+(x'',\xi'')\otimes \delta'_{x_1=0}\otimes \delta_{\xi_1=\xi_1(x'',\xi'')} \\
&\quad - 2\xi_1(x'',\xi'')\mu^-(x'',\xi'')\otimes \delta'_{x_1=0}\otimes \delta_{\xi_1=-\xi_1(x'',\xi'')}\\
 &\quad -(\d_{\xi_1}S(x',\xi'))\mu^+(x'',\xi'')\otimes \delta_{x_1=0}\otimes \delta'_{\xi_1=\xi_1(x'',\xi'')}\\
 &\quad -(\d_{\xi_1}S(x',\xi'))\mu^-(x'',\xi'')\otimes \delta_{x_1=0}\otimes \delta'_{\xi_1=-\xi_1(x'',\xi'')} \\
 &\quad + H''_{R_0} \mu^+(x'',\xi'')\otimes \delta_{x_1=0}\otimes \delta_{\xi_1=\xi_1(x'',\xi'')}  
 + H''_{R_0}  \mu^-(x'',\xi'')\otimes \delta_{x_1=0}\otimes \delta_{\xi_1=-\xi_1(x'',\xi'')}  .%\\
 \end{align*}
 Let $\chi\in\Con_0^\infty(\R)$ be such that  $\chi(\sigma)=1$ for $\sigma$ in a \nhd of 0.
 Let $\ell(x',\xi'')=x_1\chi(x_1) b(x',\xi'')$ where $b\in\Con_0^\infty (\R^{d-1}\times\R^{d-2})$.  We have    
 \begin{align*}
 \est{H'_R\mu^\d,\ell(x',\xi'')}
 =-2\est{\xi_1(x'',\xi'')\mu^+(x'',\xi''), b(0,x'',\xi'')}+2\est{\xi_1(x'',\xi'')\mu^-(x'',\xi''),b(0,x'',\xi'')}
 =0.
 \end{align*}
 Then $\mu^+=\mu^-$ as $\xi_1(x'',\xi'')\not=0$ in a \nhd of $ (x''_0,\xi''_0)$. 
 
 Now we take $\ell(x',\xi'')= b(x'',\xi'')\chi(x_1)$. 
  Observe that $(\d_{\xi_1}S(x',\xi'))=x_1\d_{\xi_1}r_2(x',\xi')$ is null on $x_1=0$, 
 we deduce that
 \begin{align*}
  \est{H'_R\mu^\d,\ell(x',\xi'')}= \est{ H''_{R_0}( \mu^+(x'',\xi'')+\mu^-(x'',\xi'')), b(x'',\xi'')   } 
  =0,
 \end{align*}
 then, as  
 $\mu^+=\mu^-$, $H''_{R_0} \mu^+ = H''_{R_0} \mu^- =0$ with the previous equation.
 Then the support of $\mu^+$, 
 $\mu^-$ and $\mu^\d$ propagate along the integral curves of $H''_{R_0}$. By assumption mGCC (see Definition~\ref{def: mGCC})
  all these curves hit the set $a\ge\delta>0$, 
 we obtain that $\mu^\d=0$ in a \nhd of such a point $(x'_0,\xi'_0)$.
 
 \subsubsection{Case $R_0(x'',\xi'')=1$}
 
By the result obtained in previous section the measure $\mu^\d$ is supported on
 $x_1=\xi_1=0$, then we have 
 $\mu^\d=\tilde\mu(x'',\xi'')\otimes  \delta_{x_1=0}\otimes \delta_{\xi_1=0 } $, where $\tilde\mu $ is a non negative Radon measure.
 We deduce from~\eqref{eq: hamiltonian boundary gamma}, as $\d_{\xi_1}S(x',\xi')=0$ on $x_1=0$,
 \begin{align*}
  H'_R\mu^\d=  H''_{R_0} \tilde\mu(x'',\xi'')\otimes \delta_{x_1=0}\otimes \delta_{\xi_1=0}  .
 \end{align*}
 Taking $\ell(x',\xi'')= b(x'',\xi'')\chi(x_1)$, where 
 $\chi\in\Con_0^\infty(\R)$ and $\chi(s)=1$ for $s$ in a \nhd of 0,  by  Lemma~\ref{lem: propagation boundary Zaremba}, 
 and arguing as in the previous case, we have
 $H''_{R_0}\tilde\mu=0$, 
 then the supports of $\tilde\mu$ and $\mu^\d$ propagate along the integral curves of $H''_{R_0}$. 
 As in the   
 previous case we obtain  $\mu^\d=0$ in a \nhd of such a point $(x'_0,\xi'_0)$.
 
 %%%%%%%%%%%%%%%%%%%%%%
 %
 %   Appendix
 %
 %%%%%%%%%%%%%%%%%%%%%%%
 
 \appendix
 
 \section{Proof of Lemma~\ref{Lem: regularity H s}, Zaremba regularity result}
 \label{Appendix : lemma regularity H s}
 It is well known that the solution of elliptic equation of second order with Zaremba boundary condition is in $H^s$ with $s<3/2$ for a data 
 in $L^2$, see for instance Shamir~\cite{Sh}, Savar\'e~\cite{Savare-1997} . 
 Here we have to prove that the solution is in semiclassical Sobolev spaces.
 
 We start from~\eqref{properties: sequence and measure} and equation $h^2Pv_h-v_h+ih av_h=hq_h$.
 We have $h^2Pv_h+v_h=r_h$ where $r_h= 2v_h-ih av_h+hq_h$ and we deduce $\| r_h\|_{L^2(\Omega)}\le C$. We observe that 
 $h^2P+1$ is a semiclassical elliptic operator. To prove the result we follow the method used in 
 Section~\ref{subsubsection: elliptic points} with the advantage that the operator is globally elliptic 
 and we keep more or less the same notations introduced in this section.
 In particular we do not have to use microlocal cutoff. 
 
 We work in a \nhd of the boundary in coordinate $(x',x_d)$ and $\Omega$ is 
 given by $x_d>0$ and $\Gamma$ by $x_1=0$. The symbol of the operator is given by $\xi_d^2+R(x,\xi')+1$ and 
 $R(x',0,\xi')=\xi_1^2+R_0(x'',\xi'') +x_1r_2(x',\xi')$.
 
Let $\tilde\chi$ a cutoff in a \nhd of a point of $\Gamma$, it will fix to $0$ in what follows. 
We set $w_h=\tilde\chi_\delta v_h$
where $\tilde\chi_\delta(x)=\tilde\chi(x/\delta)$. Let $\chi$ another cutoff function such that $\chi(x)=1$ if $x$ is contained in a 
\nhd of $ \supp\tilde\chi$ and we set $\chi_\delta(x)=\chi(x/\delta)$.
As $v_h$ is uniformly in $H^1_{sc}(x_d>0)$ and $r_2$ is a differential operator, we have
\begin{align*}
h^2D_{x_d}^2w_h+ h^2D_{x_1}^2w_h+ \op_{sc}(R_0(0,x'',x_d,\xi''))w_h + x_1\chi_\delta(x) \op_{sc}(r_2(x,\xi')) w_h= r^0_h,
\end{align*}
where $ \| r_h^0\|_{L^2(\Omega)}\le C$.
Let $\rho(x,\xi')=(\xi_d^2+\xi_1^2+R_0(0,x'',x_d,\xi'')    +   x_1\chi_\delta(x)  r_2(x,\xi') +1)^{1/2}$. 
 By symbol calculus, we have in $x_d>0$
\begin{align}
\label{eq: regularity Hs equation on z}
(hD_{x_d}+i\op_{sc}(\rho))(hD_{x_d}+i\op_{sc}(\rho)) w_h=r_h^1 \text{ where } \| r_h^1\|_{L^2(\Omega)}\le C .
\end{align}
Let $z=(hD_{x_d}+i\op_{sc}(\rho)) w_h$, we then have
\begin{align*}
(hD_{x_d}+i\op_{sc}(\rho)) z=r^1_h  .
\end{align*}
\begin{align}
	\label{eq: est elliptic zaremba}
2\Re ((hD_{x_d}+i\op_{sc}(\rho)) z|i\op_{sc}(\rho) z)
\le 2 \| r^1_h\|_{L^2(\Omega)}  \| \op_{sc}(\rho)  z\|_{L^2_{sc}(x_d>0)}.
\end{align}
Integrating by parts (see \eqref{form: integ. parts}) we have
\[
 (  i\op_{sc}(\rho) z |  hD_{x_d} z)
 = (i hD_{x_d} \op_{sc}(\rho)z | z)-ih( i\op_{sc}(\rho) z_{|x_d=0} | z_{|x_d=0})_0.
\]
We deduce
\[
2\Re (hD_{x_d} z | i\op_{sc}(\rho) z)
=( i [hD_{x_d} , \op_{sc}(\rho)]z | z)+h( \op_{sc}(\rho) z_{|x_d=0} | z_{|x_d=0})_0.
\]
As 
\(
|( i [hD_{x_d} ,\op_{sc}(\rho)]z | z)|\lesssim h  \| z\|_{L^2(0,+\infty, H^1_{sc})}^2,
\)
we deduce from~\eqref{eq: est elliptic zaremba}
\[
 \| \op_{sc}(\rho)  z\|_{L^2(x_d>0)}^2+h( \op_{sc}(\rho) z_{|x_d=0} | z_{|x_d=0})_0
 \lesssim  
  \| r^1_h\|_{L^2(\Omega)}  \| \op_{sc}(\rho) z\|_{L^2(x_d>0)} +h  \| z\|_{L^2(0,+\infty, H^1_{sc})}^2.
\]
As 
\(
 \op_{sc}(\rho^{-1})  \op_{sc}(\rho) =Id+h \op_{sc}(r_1),
\)
where \(r_1\in S_{\rm tan}^{-1}\)
 we have 
\[
 \| z\|_{L^2(0,+\infty, H^1_{sc})}^2\lesssim  \| \op_{sc}(\rho)  z\|_{L^2(x_d>0)}^2  + h\| z\|_{L^2}.
\]
We deduce 
\[
  h( \op_{sc}(\rho) z_{|x_d=0} | z_{|x_d=0})_0
 \lesssim  
   \| r^1_h\|_{L^2(\Omega)} ^2.
\]
And by tangential G\aa rding inequality~\eqref{eq: Garding inequality R d} (in \(\R^{d-1}\) instead of
\(\R^d\))
applied to 
\[
\big(\op_{sc}(\est{\xi'}^{-1/2})   \op_{sc}(\rho) \op_{sc}(\est{\xi'}^{-1/2})   \op_{sc}(\est{\xi'}^{1/2})     z_{|x_d=0} | \op_{sc}(\est{\xi'}^{1/2})   z_{|x_d=0}\big)_0
\]
we obtain, as 
\(
\est{\xi'}^{-1}\rho\ge C>0,
\)
\(
 h| z_{|x_d=0} |_{H^{1/2}_{sc}}
 \lesssim  
  \| r^1_h\|_{L^2(\Omega)} ^2.
\)
By definition of $z$, we have
 \begin{align}
 	\label{eq: regularity Hs trace relation}
 (hD_{x_d} w_h)_{|x_d=0}+i\op_{sc}(\rho_0) (w_h)_{|x_d=0}= h^{-1/2} r^2_h, \text{ where } |r^2_h|_{H^{1/2}_{sc}}\le C,
 \end{align}
 where $\rho_0(x',\xi')=(R(x',0,\xi')+1)^{1/2}$.

 Let $\beta(x'',\xi'')=( R_0(x'',\xi'') +1)^{1/2}$. We have 
 $\rho_0(x',\xi')=(\xi_1^2+\beta^2(x'',\xi'') +x_1\chi_\delta r_2(x',\xi'))^{1/2}$. Let $u_0= (hD_{x_d} w_h)_{|x_d=0}$ 
 and $u_1=(w_h)_{|x_d=0}$,
 we recall that $\supp u_0\subset \{ x_1\le 0   \} $ and $\supp u_1\subset \{ x_1\ge 0   \} $. As in Section~\ref{subsubsection: elliptic points},
 we have
   \begin{align*}
 &\xi_1\mapsto \big(\xi_1  + i \beta(x'',\xi'')  \big)^{\pm 1/2}  \text{ are  holomorphic functions on }   \Im \xi_1>0,\\
  &\xi_1\mapsto \big(\xi_1  - i \beta(x'',\xi'')  \big)^{\pm 1/2}  \text{ are  holomorphic functions on }    \Im \xi_1<0.
 \end{align*}
  Let $v_0=\op_{sc}(\xi_1+i\beta)^{-1/2}u_0$. As $u_0$ is supported in $x_1\le 0$ and $(\xi_1+i\beta)^{-1/2}$ 
is a holomorphic function on $\Im \xi_1>0$, $v_0$ is supported in $x_1\le 0$. 

We have $\xi_1-i\beta(x'',\xi'')\in S(\est{\xi'},(dx')^2+(d\xi')^2)$, then $(\xi_1-i\beta(x'',\xi''))^{\pm1/2}\in  
S(\est{\xi'}^{\pm1/2},(dx')^2+(d\xi')^2)$. This implies by symbol calculus that 
\begin{align}
	\label{eq: definition s0}
 \op_{sc}(\xi_1-i\beta)^{-1/2}  \op_{sc}(\xi_1-i\beta)^{1/2}  =Id+h\op_{sc}(s_0),
\end{align}
where $s_0\in S(1 ,(dx')^2+(d\xi')^2)$.
If $h$ is sufficiently small, $Id+h\op_{sc}(s_0)$ is invertible on  $H^s_{sc}$ for every $s$.
 Let $v_1=\op_{sc}\big(\xi_1  - i \beta(x'',\xi'')  \big)^{ 1/2}  (Id+h\op_{sc}(s_0) )^{-1} u_1$, thus we have
 $\op_{sc}\big(\xi_1  - i \beta(x'',\xi'')  \big)^{ -1/2}  v_1=u_1$. Moreover from~\eqref{eq: definition s0},
$\op_{sc}(s_0)$ map distribution supported on  $x_1\ge0$ 
 to distribution supported on  $x_1\ge0$, then by Neumann series, $ (Id+h\op_{sc}(s_0) )^{-1}$ 
 also satisfies this property. This implies that $v_1$ is supported on  $x_1\ge0$.
 From~\eqref{eq: regularity Hs trace relation}, we obtain
 \begin{align}
 	\label{eq: regularity Hs relation between v0 v1}
\op_{sc}(\xi_1+i\beta)^{-1/2} v_0+i\op_{sc}(\xi_1+i\beta)^{-1/2}\op_{sc}(\rho_0) \op_{sc}(\xi_1-i\beta)^{-1/2}v_1
= h^{-1/2} r^3_h,
 \end{align}
where $|r^3_h|_{H^{1}_{sc}}\le C$.
The principal symbol of $\op_{sc}(\xi_1+i\beta)^{-1/2}\op_{sc}(\rho_0) \op_{sc}(\xi_1-i\beta)^{-1/2}$  is by a simple computation
\begin{align*}
(\xi_1+i\beta)^{-1/2}\rho_0 (\xi_1-i\beta)^{-1/2}= 1+ \frac{x_1 \chi_\delta r_2}{ (\xi_1^2+\beta^2)^{1/2}\big( \rho_0 
+ (\xi_1^2+\beta^2)^{1/2} \big)}=1+x_1 \chi_\delta r_3,
\end{align*}
 where $r_3\in S(1,(dx')^2+(d\xi')^2)$. Formula~\eqref{eq: regularity Hs relation between v0 v1} reads
 \begin{align*}
 \op_{sc}(\xi_1+i\beta)^{-1/2} v_0+i v_1 +ix_1 \chi_\delta \op_{sc}(r_3)v_1=  h^{-1/2} r^3_h.
 \end{align*}
 We restrict this equation on $x_1>0$, as $ \op_{sc}(\xi_1+i\beta)^{-1/2} v_0$ is supported on $x_1\le0$, we obtain
 \begin{align*}
 i (v_1)_{|x_1>0} +ix_1 \chi_\delta (\op_{sc}(r_3)v_1)_{|x_1>0} =  h^{-1/2} (r^3_h)_{|x_1>0} .
 \end{align*}
As $|w_{|x_1>0}|_{H^s_{sc} }\le |w|_{H^s_{sc}} $, we obtain that
\begin{align}
	\label{eq: regularity Hs v1 estimate}
| ( v_1)_{|x_1>0}  |_{H^s_{sc}(x_1>0)}    \le   |  x_1 \chi_\delta \op_{sc}(r_3)v_1  |_{H^s_{sc}}     + h^{-1/2}  | r^3_h   |_{H^s_{sc}}  .
\end{align}
%%%%%%%%%%%%%%%%%%%%%%%
%
%   Lemma
%
%%%%%%%%%%%%%%%%%%%%%%%
\begin{lemma}
	\label{lem: regularity Hs 1}
Let $\chi\in\Con_0^\infty(\R^d)$ and $\chi_\delta(x)=\chi(x/\delta)$, where $\delta>0$. Then there exist $C>0$ such that for 
every $\delta>0$, $\|x_1 \chi_\delta\|_{C^\alpha}\le C \delta^{1-\alpha}$, for $\alpha\in(0,1)$ and $\| x_1 \chi_\delta\|_{L^\infty}\le C\delta$.
\end{lemma}
\begin{proof}
First, we have $|x_1\chi(x/\delta)|\le C \delta$. Second, $|(x_1+y_1)\chi((x+y)/\delta)-x_1\chi(x/\delta)|\le 2C \delta $ and 
$|\d_x (x_1\chi(x/\delta))|\le C$, then we have $|(x_1+y_1)\chi((x+y)/\delta)-x_1\chi(x/\delta)|\le  C |y|$. 
For $\alpha \in(0,1)$, interpolating both estimates, we have $|(x_1+y_1)\chi((x+y)/\delta)-x_1\chi(x/\delta)|\le  C \delta^{1-\alpha}   |y|^\alpha$. 
Which gives the result.
\end{proof}
%%%%%%%%%%%%%%%%%%%%%%
%
%   LEMMA
%
%%%%%%%%%%%%%%%%%%%%%%%%
\begin{lemma}
	\label{lem: product Cs Hs}
Let $s\in(0,1)$ and $0<s<\alpha<1$, there exists $C>0$, such that for every $f\in C^\alpha_{sc}(\R^d)$, and $g\in H^s_{sc}(\R^d)$, 
$ fg\in H^s_{sc}(\R^d)$ and we have
\begin{equation*}
\| fg\|_{H^s_{sc}(\R^d)}\le C \| f\|_{C^\alpha_{sc}(\R^d)}\| g\|_{H^s_{sc}(\R^d)}.
\end{equation*}
\end{lemma}
Here we say that $f\in C^\alpha_{sc}(\R^d)$ if $f$ is bounded and 
$ h^\alpha |f(x+y)-f(x)|\le C|y|^\alpha$. The norm on $ C^\alpha_{sc}(\R^d)$ is 
$\| f\|_{L^\infty(\R^d)}+\sup _{x,y\in\R^d}h^\alpha |f(x+y)-f(x)| |y|^{-\alpha}$.
\begin{proof}
We can follow the classical proof that the multiplication by $C^\alpha$ functions are bounded operators on $H^s$, 
using Littlewood-Paley theory and para-product in spirit of Bony~\cite{Bony-1981}.   

We recall the Littlewood-Paley decomposition (see~\eqref{eq: decomposition Littlewood-Paley} for notations), 
we have $w=\sum_{k\ge-1}\Delta_k w $
where ${\cal F} (\Delta_k(w))(\xi)=\phi(2^{-k}h\xi) {\cal F}(w)(\xi)$ for $k\ge0$ and ${\cal F} (\Delta_{-1}(w))(\xi)=\psi(h\xi) {\cal F}(w)(\xi)$. 
Let $S_k=\sum_{-1\le j\le k}\Delta_j$.

By assumptions on $f$, we have $\| S_k f\|_{L^\infty(\R^d)} \le C\| f\|_{L^\infty(\R^d)}$ and 
by the usual proof of characterization of $C^\alpha$ functions with Littlewood-Paley decomposition
we have
$\| \Delta_k f\|_{L^\infty(\R^d)} \le C2^{-k\alpha}\| f\|_{C^\alpha_{sc}(\R^d)}$ for $k\ge 0$. 
By assumptions on $g$, we have $\| \Delta_k g\|_{L^2(\R^d)} \le c_k2^{-sk}$ for $k\ge 0$, where 
$\| (c_k)\|_{\ell^2(\N)} \le C\| g \|_{H^s_{sc}(\R^d)}$, $\| \Delta_{-1} g\|_{L^2(\R^d)} \le C\| f\|_{L^2(\R^d)}$ 
and $\| S_k g\|_{L^2(\R^d)} \le C\| g \|_{L^2(\R^d)}$.

The product
 $fg= \sum_{k\ge -1} S_k(f)\Delta_k(g)
+ \sum_{k\ge 0}S_{k-1}(g)\Delta_k(f)$. We estimate each term in previous formula.
For $j\ge 2$ we have  
\begin{align*}
\| \Delta_j \big( \sum_{k\ge -1} S_k(f)\Delta_k(g) \big)\|_{L^2(\R^d)}  &\lesssim \|  \Delta_j \big( \sum_{k\ge j-2} 
S_k(f)\Delta_k(g)  \big) \|_{L^2(\R^d)}\\
&\lesssim \sum_{k\ge j-2} \| S_k(f)\|_{L^\infty(\R^d)}\| \Delta_k(g) \|_{L^2(\R^d)}\\
&\lesssim 2^{-sj}  \| f\|_{L^\infty(\R^d)}\sum_{k\ge j-2} c_k  2^{-(k-j)s},
\end{align*}
and $d_j=\sum_{k\ge j-2} c_k  2^{-(k-j)s}\in \ell^2*\ell^1\subset \ell^2$, where $\|(d_j) \|_{\ell^2}\lesssim \| g \|_{H^s_{sc}(\R^d)}$.

\begin{align*}
\| \Delta_j \big(     \sum_{k\ge 0}S_{k-1}(g)\Delta_k(f)    \big)\|_{L^2(\R^d)} &\le 
\| \Delta_j \big(     \sum_{k\ge j-2}S_{k-1}(g)\Delta_k(f)    \big)\|_{L^2(\R^d)} \\
& \le  \sum_{k\ge j-2} \|   S_{k-1}(g)\|_{L^2(\R^d)} \| \Delta_k(f)  \|_{L^\infty(\R^d)}\\
&\lesssim     \| g \|_{L^2(\R^d)} \| f\|_{C^\alpha_{sc}(\R^d)} \sum_{k\ge j-2} 2^{-k\alpha} \\
&\lesssim    \| g \|_{L^2(\R^d)} \| f\|_{C^\alpha_{sc}(\R^d)}  2^{-js} 2^{-j(\alpha-s)}.
\end{align*}
As $(2^{-j(\alpha-s))}$ is in $\ell^2$, and as the result is obvious for $j\le 1$ we obtain the result.
\end{proof}

Observe that 
\(
\| f \|_{C^\alpha_{sc}(\R^d)}  \lesssim \| f \|_{C^\alpha(\R^d)} .
\)
Then by Lemmas~\ref{lem: regularity Hs 1} and \ref{lem: product Cs Hs} we have $|   x_1 \chi_\delta w |_{H^s_{sc}}\le C \delta^{1-\alpha }| w|_{H^s_{sc}}$.

From \eqref{eq: regularity Hs v1 estimate} we have
\begin{align*}
| ( v_1)_{|x_1>0}  |_{H^s_{sc}(x_1>0)}    \le C  \delta^{1-\alpha }   | v_1  |_{H^s_{sc}}   + h^{-1/2}  | r^3_h   |_{H^s_{sc}} ,
\end{align*}
 for $\le s <\alpha<1/2$. As  $ | v_1  |_{H^s_{sc}}  =| ( v_1)_{|x_1>0}  |_{H^s_{sc}(x_1>0)} $ for $s\in[0,1/2)$, we have 
 \begin{align*}
 | ( v_1)_{|x_1>0}  |_{H^s_{sc}(x_1>0)} \le Ch^{-1/2}  | r^3_h   |_{H^s_{sc}}\le  Ch^{-1/2}  ,
 \end{align*}
 for $\delta$ sufficiently small. We obtain
 \begin{align*}
 |  u_1|_{H^{s+1/2}(x_1> 0)}\le | \op_{sc}\big(\xi_1  - i \beta(x'',\xi'')  \big)^{ -1/2}  v_1|_{H^{s+1/2}}\le C| v_1|_{H^{s}} \le Ch^{-1/2}.
 \end{align*}
 From~\eqref{eq: regularity Hs trace relation} we deduce that 
 \begin{align*}
  |  u_0|_{H^{s-1/2}} \le Ch^{-1/2}.
 \end{align*}
 The solution $w_h$ of semiclassical elliptic problem with boundary condition satisfying
 \begin{align*}
 | (w_h)_{| x_d=0}|_{H^{s+1/2}_{sc}}\le Ch^{-1/2} \text{ and }    |  (hD_{x_d}w_h)_{| x_d=0} |_{H^{s-1/2}} \le Ch^{-1/2},
 \end{align*}
is in $H^{1+s}_{sc}(x_d>0)$ and $\| w_h\|_{H^{1+s}(x_d>0)}\le C$. This result is 
 well-known and it is a consequence of Formula~(60) in~\cite{Robbiano-2013}. This achieves the 
 proof of Lemma~\ref{Lem: regularity H s}.
 
%%%%%%%%%%%%%%%%%%%%%%%%%%%%%%%
%
%   Section  
%
%%%%%%%%%%%%%%%%%%%%%%%%%%%%%%%% 
\section{A priori estimate for the trace of solution for Neumann boundary condition}
	\label{sec: trace Neumann boundary condition}
	We begin this section by recall some result on semiclassical Fourier Integral Operator.
%%%%%%%%%%%%%%%%%%%%%%%%%%%%%%%
%
%   Lemma
%
%%%%%%%%%%%%%%%%%%%%%%%%%%%%%%%% 
\begin{lemma}
	\label{lem: Symplectic transformation}
Let $(x_0', \xi_0' )$ be such that $R(x_0', 0, \xi_0' )-1=0$. For all $x_d$ in a \nhd of 0,
there exist a  smooth symplectic transformation $ \kappa :U_0\to U_1$ where $U_0$  and $U_1$ are some open set 
respectively of  $\R^{d-1}_{x'}\times \R^{d-1}_{\xi'}$  and of  $\R^{d-1}_{y'}\times \R^{d-1}_{\eta'}$,
satisfying $ (x_0', \xi_0' )\in U_0 $,  $(0,0) \in U_1$,  $\kappa(x_0', \xi_0' ) = (0,0)$, and $ \kappa^*(\eta_1)= R-1$.
 Moreover $x_d$ acts as a parameter and $\kappa$ 
is smooth with respect $x_d=y_d$.
\end{lemma}
This lemma is classical. We can find a proof in H\"ormander~\cite[Theorem 21.1.6]{HormanderV3-2007}. 
This means we can complete the coordinate $R-1$ in a  symplectic manner.

To avoid ambiguity even if $x_d=y_d$ we denote $x_d$ when we work in $(x,\xi')$ variables and $y_d$ otherwise. 

We call a symbol of order 0 a symbol $a\in S(1,|dx|^2+|d\xi'|^2)    
$ or in $S(1,|dy|^2+|d\eta'|^2)$. 
In this section we only use tangential symbol, but as in what follows we have to use different classes of symbols,
we prefer use everywhere the same kind of notation.
 %%%%%%%%%%%%%%%%%%%%%%%%%%%%%%%
%
%   Lemma
%
%%%%%%%%%%%%%%%%%%%%%%%%%%%%%%%% 
\begin{lemma}
	\label{lem: FIO}
Associated with $\kappa$, there exists $F$ a semiclassical  Fourier Integral Operator 
 satisfying the following properties,
\begin{description}
\item[i)] $F$ is a unitary operator uniformly with respect $x_d$.
\item[ii)] For all $\tilde a\in\Con_0^\infty( U_0) $, $F^{-1}\op_{sc}(\tilde a)F=\op_{sc}(a)$, where $a=\kappa^*\tilde a+hb$ where
$b$ is a symbol of order 0. In particular we have
$F^{-1} \op_{sc}(\eta_1\ct^2(y,\eta')) F= \op_{sc}(\cc^2(R-1)) +h\op_{sc}(b)$, where  
$\ct\in\Con_0^\infty (U_0)$, $\ct\ge0$  and $b$ a symbol of order 0.
\item[iii)]  there exist $\theta$ a symbol of order 0, $B$ a bounded operator on $L^2$ such that
$\op_{sc} (\theta)^*=\op_{sc} (\theta)$, $\kappa^*\tilde\chi=\chi$ and
$(\d_{x_d} F) F^{-1}=ih^{-1}\op_{sc} (\theta)+hB$.
\item[iv)] 
If the operators $A $ and $\tilde A$ are such that \( A=F^{-1}\tilde AF \) then 
\[
\d_{x_d}A=F^{-1}\big(  \d_{y_d} \tilde A+ih^{-1}[\tilde A,\op_{sc} (\theta) ]+  h [  \tilde A , B ] 
\big)F
\]
where $B$ is the operator defined previously.
\item[v)]  In particular we have $\kappa^* \{ \eta_1,\theta \}= \d_{x_d}R$ in a \nhd of $(x'_0,0,\xi_0')$. 
\end{description}
\end{lemma}
\begin{remark}
 Zworski states the result  for Weyl quantification. It is clear that we can deduce
 the result for classical quantification. In the proof of Lemma~\ref{lem: FIO} we use Weyl quantification but in the rest of this section
 we shall use classical quantification to be coherent with notation used in this article.
\end{remark}
A proof of Lemma~\ref{lem: FIO} is given in Section~\ref{sec: proof of lemmas}.

Here we adapt, in the framework of semiclassical analysis, the results obtained by Tataru~\cite{Tataru-1998} especially 
Lemma~4.3, Propositions~4.5 and 4.7. We essentially keep the notation used in that paper.

From now we shall use two semiclassical quantifications of symbol, one with parameter $h$ and
the other with parameter  $h^{1/3}$. 
To avoid ambiguity or confusion between both, we do not use the notation $\op_{sc}$
but we use classical quantification. For instance, for $a$ a symbol of order 0 we have $\op_{sc} (a)=\op(a(x,h\xi'))$ 
that is we keep 
the $h$ or $h^{1/3}$ in the notation. 

Let $g=|dy|^2+h^{2/3}\est{h^{1/3}\eta_1}^{-2}|d\eta'|^2$, this metric gives  symbol classes
essentially as semiclassical symbol classes with $h^{1/3}$ for
semiclassical parameter. 
We let to the reader to check that $g$ is slowly varying and $\sigma$ temperate. The 
"$h$" defined by H\"ormander associated with $g$ is $h^{1/3}\est{h^{1/3}\eta_1}^{-1}$.
It is the quantity we gain in the asymptotic expansion 
for the 
symbol calculus.
In particular the function $\est{h^{1/3}\eta_1}^{\nu}$ is a $g$ continuous and $\sigma, g$ temperate for every $\nu\in\R$.
 We refer to \cite[Chapter 18, Sections 4 and 5]{HormanderV3-2007} for  definitions used freely here.
 
 From~\eqref{properties: sequence and measure} we have $-h^2\d_{x_d}^2v_h+\op(R(x,h\xi')-1)v_h=hq_h$. 
Let $\vv=\op(  \chi_1   (x,h\xi'))v_h$, where $ \chi_1 $ is supported where $\cc=1$ and $\cc$ is the cutoff function define in Lemma~\ref{lem: FIO}. By symbol calculus we have 
\begin{equation}
	\label{eq: on rm v}
	-h^2\d_{x_d}^2\vv+\op(\cc^2(x,h\xi')(R(x,h\xi')-1))\vv=h\qq
\end{equation}
 where  
 \begin{align*}
 \qq&= \op( \chi_1(x,h\xi'))q_h   -h^{-1}[h^2D_{x_d}^2, \op( \chi_1(x,h\xi'))]v_h   \\
 &\quad + h^{-1} \big(  - \op( \chi_1(x,h\xi'))\op(R(x,h\xi')-1)
 +\op(\cc^2(x,h\xi') (R(x,h\xi')-1))\op(\chi_1(x,h\xi'))\big)
 \end{align*}
  is bounded on $L^2(x_d>0)$ from symbol calculus and support properties of $\cc$ and $\chi_1$.
  Moreover we have $|\vv|_{L^2(\R^{d-1})} \lesssim |v_h|_{L^2(\R^{d-1})}$. 
 
We recall some properties of the Airy function which is denoted by $\Ai$. It verifies the equation 
$\Ai''(z) -z\Ai(z)=0$ for $z\in\C$, 
 $\Ai $ is real on the real axis and $\overline{\Ai(z)}=\Ai(\bar z)$. 
Let $\omega=e^{2i\eps \pi/3}$ for $\eps=\pm1 $. 

For $x\in\R$, let $\alpha(x)=-\omega\Ai'(\omega x)/\Ai(\omega x)\in\Con^\infty(\R)$. As the zero of $\Ai$ are on the negative 
real axis, the function $\alpha $ is well defined for $x\in\R$ and smooth.
 The function $\alpha$ 
satisfies the following properties.
%%%%%%%%%%%%%%%%%%%%%%%%%%%%%%%
%
%   Lemma
%
%%%%%%%%%%%%%%%%%%%%%%%%%%%%%%%% 
\begin{lemma}
	\label{lem: alpha properties}
We have
\begin{description}
\item[i)] $ \alpha(x)=-\sqrt{x}+\frac{1}{4x}+b_1(x) $ for $x>0$
\item[ii)] $ \alpha(x)=\eps i\sqrt{-x}+\frac{1}{4x}+b_2(x) $ for $x<0$
\item[iii)] $\Re \alpha(x)<0 $ for all $x\in\R$,
\item[iv)]  $\alpha$ satisfies the differential equation $\alpha'(x)= \alpha^2(x)-x$.
\end{description}	
where $b_j\in S( \est{x}^{-5/2}, |dx|^2)$ for \(j=1,2\). 
\end{lemma}
The proof of lemma is given in Section~\ref{sec: proof of lemmas}.

 Let $\tilde r_d$ be such that $\kappa^*\tilde r_d= -\d_{x_d} R$. We assume that locally $\d_{x_d} R<0$, this implies that 
 $ \tilde r_d>0$ in a  \nhd of $(0,0)$.
 
  Let $\tilde a (y, \eta')=h^{1/3}\tilde \chi (y,h\eta') \tilde r_d^{1/3}(y,h\eta') \alpha(\zeta)$ 
where $\zeta= h^{1/3}\eta_1\tilde r_d^{-2/3}(y,h\eta')$.  
We assume that on the support of  $\tilde \chi (y,h\eta')$, 
$\tilde r_d(y,h\eta') >0$.
In what follows we denote $\tilde \rho=(y,h\eta')$.
We define $\tilde A=\op(\tilde a )$,    $\tilde \Psi =h^{-1/3}\op (\est{h^{1/3}\eta_1}^{-1/2})$, and let 
$A=F^{-1} \tilde A  F$, $\Psi=F^{-1}\tilde \Psi F$.

We have $\tilde a\in S(h^{1/3}\est {h^{1/3}\eta_1}^{1/2}  ,g)$  as 
$h^{-1/3} \est{h^{1/3}\eta_1}^{-1/2}\in S(h^{-1/3} \est{h^{1/3}\eta_1}^{-1/2},g)$
and from Lemma~\ref{lem: alpha properties}

%%%%%%%%%%%%%%%%%%%%%%%%%%%%%%%
%
%   Proposition
%
%%%%%%%%%%%%%%%%%%%%%%%%%%%%%%%% 
\begin{proposition}
	\label{lem: first est r-1}
Let $\vv$ satisfying properties~\eqref{eq: on rm v}.
There exist $C_0>0$ such that 
\[
|( h\d_{x_d}\vv-A\vv)_{|x_d=0}|^2_{L^2(\R^{d-1})}+   \| \Psi ( h\d_{x_d}\vv-A\vv  )\|_{L^2(x_d>0)}^2
\le C_0 \| v_h\|_{H^1_{sc}(x_d>0)}^2+C_0\|  \qq \|^2_{L^2(x_d>0)}.
\]
This result is equivalent to the following. Let \(w_h=F(  h \d_{x_d}\vv-A\vv)\), 
there exist $C_0>0$ such that 
\[
|( w_h)_{|y_d=0}|^2_{L^2(\R^{d-1})}+   \| \tilde\Psi w_h \|_{L^2(y_d>0)}^2
\le C_0 \| v_h\|_{H^1_{sc}(x_d>0)}^2+C_0\|  \qq \|^2_{L^2(x_d>0)}.
\]
\end{proposition}
Here and in this section we denote 
\(
\| u\|_{H^1_{sc}(x_d>0)}= \| u\|_{L^2(x_d>0)}+  \| h\nabla u\|_{L^2(x_d>0)}.
\)

\begin{proof}
To ease notation we write $| u|$ instead of $|u|_{L^2}(x_d)$ when there is no ambiguity on the fact that the 
$L^2$ norm is taken on variables $x'$ or $y'$  at point  $x_d$ or $y_d$. By the same abuse of notation we 
write the inner product $(.|.)$ instead of  $ (.|.)_{L^2(\R^{d-1})}(x_d)$.
 We compute 
\begin{multline*}
\frac12\d_{x_d}| h\d_{x_d}\vv-A\vv  |^2= \Re \big(   \d_{x_d}(h\d_{x_d}\vv-A\vv) | h  \d_{x_d}\vv-A\vv \big)  \\
=   \Re \big(  h\d_{x_d}^2\vv   -   ( \d_{x_d}A) \vv-A  \d_{x_d}\vv | h  \d_{x_d}\vv-A\vv \big)  \\
=   \Re \big(    h^{-1}\op(\cc^2(x,h\xi')(R(x,h\xi')-1))\vv -\qq   -   ( \d_{x_d}A) \vv-A  \d_{x_d}\vv |  h \d_{x_d}\vv-A\vv \big) ,
\end{multline*}
from~\eqref{eq: on rm v}.
Then we have
\begin{equation}
	\label{eq: first energy equality}
\frac12\d_{x_d}| h\d_{x_d}\vv-A\vv  |^2= I_1+I_2+I_3+I_4,
\end{equation}
where
\begin{align*}
I_1&=     \Re \big(  -h^{-1}A( h\d_{x_d}\vv-A\vv)  |  h \d_{x_d}\vv-A\vv \big)  \\
I_2&=     \Re \big(     h^{-1}(\op(\cc^2(x,h\xi')(R(x,h\xi')-1))-A^2)\vv    |  h \d_{x_d}\vv-A\vv \big)  \\
I_3&=      \Re \big( - ( \d_{x_d}A )\vv        |  h \d_{x_d}\vv-A\vv \big)   \\
I_4&=       -   \Re \big(       \qq    |  h \d_{x_d}\vv-A\vv \big) .
\end{align*}
Using $w_h=F(  h \d_{x_d}\vv-A\vv)$, we have
\begin{align*}
I_1=    \Re \big(  -h^{-1}FA F^{-1}w_h |w_h \big) =   \Re \big(  -h^{-1}\tilde A  w_h |w_h \big) .
\end{align*}
From Lemma~\ref{lem: alpha properties} 
 we have
\[
- h^{-1} \Re \tilde a(y,\eta) \ge \delta  h^{-2/3} \tilde \chi^2(\tilde \rho)\est{h^{1/3}\eta_1}^{-1},
  \]
and $h^{-1} \Re   \tilde  a(y,\eta) \in S( h^{-2/3} \est{h^{1/3}\eta_1}^{1/2} ,g)$. Then from Fefferman-Phong inequality 
(see \cite[Theorem 18.6.8]{HormanderV3-2007})  and as the real part of symbol of 
\(
\op( h^{-1/3} \tilde \chi(\tilde \rho)\est{h^{1/3}\eta_1}^{-1/2})^*  \op(h^{-1/3} \tilde \chi(\tilde \rho)\est{h^{1/3}\eta_1}^{-1/2})
\)
 is 
\(
 h^{-2/3} \tilde \chi^2(\tilde \rho)\est{h^{1/3}\eta_1}^{-1}
\)
modulo an operator bounded on $L^2$, 
we have
\begin{equation*}
I_1\ge   \delta | h^{-1/3} \op ( \tilde \chi (\tilde \rho)\est{h^{1/3}\eta_1}^{-1/2})   w_h|^2 -C | w_h|^2, 
\end{equation*}
for $C>0$. 

From Lemma~\ref{lem: estimate by tilde Psi} we obtain   
\begin{equation}
	\label{est: I 1}
I_1\ge   \delta | \tilde \Psi   w_h|^2 -C\big( | w_h|^2+  |v_h|^2+|h\d_{x_d}v_h|^2\big).
\end{equation}
 We have 
\begin{align*}
I_2&=      \Re \big(     h^{-1}F(\op(\cc^2(x,h\xi')(R(x,h\xi')-1)-A^2)  \vv    | w_h\big) \\
&=   \Re \big(     h^{-1}(\op(h\eta_1\tilde \chi^2(\tilde \rho) )-\tilde A^2)  F\vv    | w_h\big)+    \Re \big(  B_0 \vv    | w_h\big)  ,
\end{align*}
where $B_0=\op_{sc}(b)$ is bounded on $L^2$ (see Lemma~\ref{lem: FIO}).

The symbol of $\tilde A^2$ is $\tilde a^2\in S(h^{2/3} \est{h^{1/3}\eta_1} ,g)$ 
modulo a term in $S(h,g)$. From definition of $\tilde a$ 
and Lemma~\ref{lem: alpha properties} we have 
\begin{align*}
\tilde a^2&= h^{2/3} \tilde \chi^2(\tilde \rho)\tilde r_d^{2/3}(\tilde \rho)\alpha^2(\zeta)   \notag  \\
&=  h^{2/3} \tilde \chi^2(\tilde \rho)\tilde r_d^{2/3}(\tilde \rho)\big(  \zeta 
+  \alpha'(\zeta)  \big)  \notag  \\
&= h\eta_1 \tilde \chi^2(\tilde \rho)+   h^{2/3} \tilde \chi^2(\tilde \rho)\tilde r_d^{2/3}(\tilde \rho)  \alpha'(\zeta).
\end{align*}
We have
\[
h^{-1/3} \tilde \chi^2(\tilde \rho)\tilde r_d^{2/3}(\tilde \rho)  \alpha'(\zeta) \in S(h^{-1/3}\est{h^{1/3}\eta_1}^{-1/2}, g),
\]
we then obtain 
\begin{equation}
	\label{est: I 2}
|I_2|\lesssim   |v_h| \big(  | \tilde \Psi w_h | +  |w_h |   \big)
\end{equation}

We have 
\begin{align*}
I_3&= -  \Re \big( F ( \d_{x_d}A ) F^{-1} F \vv        | w_h \big)  \\
&=  -  \Re \big(  ( \op(\d_{y_d} \tilde a(\tilde \rho) )+ih^{-1}[ \op(\tilde a(\tilde \rho) ),\op(\theta(\tilde \rho))]  F \vv        | w_h \big) \\
&\quad  -  \Re \big(  h[ \op(\tilde a(\tilde \rho) ) , B]    F \vv     | w_h \big)   ,
\end{align*}
from Lemma~\ref{lem: FIO}.

Observe that 
\(
\tilde a(\tilde \rho)\in S(1,g)
\)
as $h^{1/3}\est{h^{1/3}\eta_1}^{1/2}$ is bounded on support of  $\tilde\chi(\tilde\rho)$. Then 
\(
\op(\d_{y_d}  \tilde a(\tilde \rho) )
\)
and 
\(
 h[ \op(\tilde a(\tilde \rho) ) , B] 
\)
are bounded operator on $L^2$.

From properties of $\alpha$ and symbol calculus,
the symbol of 
\[
h^{-1}[ \op(\tilde a(\tilde \rho) ),\op(\theta(\tilde \rho))] 
\text{ 
is in  }
S(h^{-1/3}\est{h^{1/3}\eta_1}^{-1/2},g).
\]
 Then   we obtain  
 \begin{equation}
		\label{est: I 3}
|I_3|\lesssim   |v_h| \big(  | \tilde \Psi w_h | +  |w_h |   \big).
\end{equation}

We have 
\begin{equation}
	\label{est: I 4}
|  I_4|=  |   \Re \big(       \qq    , F^{-1}w_h\big)|  \lesssim |\qq||w_h|
\end{equation}
From \eqref{eq: first energy equality}, \eqref{est: I 1}, \eqref{est: I 2}, \eqref{est: I 3} and \eqref{est: I 4} we have
\begin{align}
	\label{est: first energy equality 2}
\frac12\d_{x_d}| h\d_{x_d}\vv-A\vv  |^2
&\ge    \delta |  \tilde \Psi  w_h|^2 -C | w_h|^2-C  |  \tilde \Psi  w_h|  |v_h|  
-C  |v_h|^2-C|h\d_{x_d}v_h  | ^2
 -C  |\qq|^2  \notag\\
&\ge    \delta' |  \tilde \Psi  w_h|^2 -C'\big( | w_h|^2+|v_h| ^2+|h\d_{x_d}v_h  | ^2 +  |\qq|^2\big) .
\end{align}
Observe that 
\(
h^{1/3} \est{h^{1/3}\eta_1}^{1/2}\lesssim \est{h\eta'}^{1/2}, 
\)
then 
\(
  h^{1/3}\tilde \chi(\tilde \rho)^2\tilde r_d^{1/3}\alpha(\zeta)  
\)
is bounded. 

We then have 
\begin{align}
	\label{est: v_h from  v_h}
| A\vv|\lesssim |\tilde A F\vv|& \lesssim  | v_h |.
\end{align}
We also have 
\(
|h\d_{x_d}\vv |\lesssim |h\d_{x_d}v_h |  + |v_h | .
\)
Then 
\begin{equation}
	\label{est: w h by v h}
| w_h| \lesssim  |h\d_{x_d}v_h |  + |v_h | .
\end{equation}

We deduce from~\eqref{est: first energy equality 2} that
\begin{align*}
\frac12\d_{x_d}| h\d_{x_d}\vv-A\vv  |^2
&\ge    \delta' |  \tilde \Psi  w_h|^2  -C  |\qq|^2 -C |v_h|^2 -C|h\d_{x_d} v_h|^2.
\end{align*}
Integrating this inequality between 0 and \( \sigma>0\) 
we have
\begin{multline*}
| h\d_{x_d}\vv-A\vv  |^2(0)+\delta'\int_0^\sigma |  \tilde \Psi  w_h|^2(x_d)dx_d
\lesssim  \|\qq\|^2 + \|v_h\|_{H^1_{sc}(x_d>0)}^2 + |h\d_{x_d}v_h | ^2(\sigma) + |v_h |^2(\sigma)
\end{multline*}
from \eqref{est: w h by v h}.
Integrating this inequality between  two positive values of \(\sigma\) and as 
\(
 |  \tilde \Psi  w_h|(y_d)= |  \Psi (  h \d_{x_d}\vv-A\vv ) |(x_d)
\)
 we obtain the result.
\end{proof}

To state the next result we have to introduce another operator. 
%%%%%%%%%%%%%%%%%%%%%%%%%%%%%%%
%
%   Lemma
%
%%%%%%%%%%%%%%%%%%%%%%%%%%%%%%%% 
\begin{lemma}
	\label{lem: properties beta}
There exist a function $\beta\in\Con^\infty(\R)$ satisfying the following properties
\begin{description}
\item[i)]  $ -\beta'-\beta\Re \alpha\ge 0$
\item[ii)] $ \beta\in S(\est{x}^{-1/4},|dx|^2)$
\item[iii)]  $ \beta\gtrsim \est{x}^{ -1/4}$.  
\end{description}
\end{lemma}
A proof is given in Section~\ref{sec: proof of lemmas}.

We recall the notation $\tilde a (y, \eta')=h^{1/3}\tilde \chi (\tilde \rho) \tilde r_d^{1/3} (\tilde \rho)\alpha(\zeta)$ 
where $\zeta= h^{1/3}\eta_1\tilde r_d^{-2/3}(\tilde \rho)$ and by assumption, $ \tilde r_d (\tilde \rho)>0$ 
on the support of $\tilde \chi (\tilde \rho)$.

Let
$\tilde c(y,\eta')=h^{-1/6}\tilde \chi_2(\tilde \rho)\beta(\zeta)$, where $\tilde \chi_2$ is supported
on $\{ \tilde \chi=1 \}$ and $\tilde \chi_2=1$ on a \nhd of $(0,0)$.
We have $\tilde a\in S(h^{1/3}\est {h^{1/3}\eta_1}^{1/2}  ,g)$ and $\tilde c\in S( h^{-1/6} \est {h^{1/3}\eta_1} ^{-1/4},g)  $.

We define  $\tilde C=\op (\tilde c ) $ and $C=F^{-1} \tilde C  F$.

%%%%%%%%%%%%%%%%%%%%%%%%%%%%%%%
%
%   Proposition
%
%%%%%%%%%%%%%%%%%%%%%%%%%%%%%%%% 
\begin{proposition}
\label{lem: alpha airy}
Let $\vv$ satisfying properties~\eqref{eq: on rm v}.
There exists $C_0>0$  such that 
\[
|C(h \d_{x_d}\vv-A\vv)_{|x_d=0}|^2_{L^2(\R^{d-1})}  
\le C_0 \| v_h\|_{H^1_{sc}(x_d>0)}^2+C_0\|  \qq \|^2_{L^2(x_d>0)}.
\]
\end{proposition}
\begin{proof}
We have 
\begin{multline*}
\frac12\d_{x_d}|C( h\d_{x_d}\vv-A\vv  )|^2= \Re \big(   \d_{x_d}(C(h\d_{x_d}\vv-A\vv)) | C(h  \d_{x_d}\vv-A\vv )\big)  \\
=   \Re \big(( \d_{x_d} C)(  h\d_{x_d}\vv-A\vv   )+
  C(h\d_{x_d}^2\vv   -   ( \d_{x_d}A) \vv-A  \d_{x_d}\vv) |C( h  \d_{x_d}\vv-A\vv )\big)  \\
=   \Re \big( ( \d_{x_d} C)(  h\d_{x_d}\vv-A\vv   )+ C(  h^{-1}\op(\chi^2(x,h\xi')(R(x,h\xi')-1))\vv -\qq   \\
-   ( \d_{x_d}A) \vv-A  \d_{x_d}\vv )  |  C(  h \d_{x_d}\vv-A\vv )\big) ,
\end{multline*}
from~\eqref{eq: on rm v}.
Then we have
\begin{equation}
	\label{eq: second lemma a la Tataru}
\frac12\d_{x_d}|C( h\d_{x_d}\vv-A\vv  )|^2=J_1+J_2+J_3+J_4+J_5=K_1+K_2+K_3+K_4+K_5,
\end{equation}
where
\begin{align*}
J_1&=   \Re\big(  C^*( \d_{x_d} C)  ( h\d_{x_d}\vv-A\vv  )  | h\d_{x_d}\vv-A\vv  \big)  \\
J_2&=  -\Re\big( (  \d_{x_d} A) \vv    |  C^* C ( h\d_{x_d}\vv-A\vv  )  \big)     \\
J_3&=  -\Re\big( \qq   |   C^* C( h\d_{x_d}\vv-A\vv  )  \big)     \\
J_4&=   \Re\big((  -  C^* C)A  \d_{x_d}\vv |  h\d_{x_d}\vv-A\vv \big)         \\
J_5&=   \Re\big(  h^{-1}\op(\chi^2(x,h\xi')(R(x,h\xi')-1)) \vv  |    C^* C(h\d_{x_d}\vv-A\vv )  \big)   .
\end{align*}
Taking as in the proof of Proposition~\ref{lem: first est r-1},   
$w_h=F(  h \d_{x_d}\vv-A\vv)$ and from Lemma~\ref{lem: FIO}  we write 
\begin{align*}
K_1&=   \Re\big( \tilde C^*( \d_{y_d} \tilde C)  w_h  | w_h  \big)  \\
K_2&=  -\Re\big(  ( \d_{y_d} \tilde A  ) F\vv |  \tilde C^* \tilde C  w_h  \big)    
-  \Re\big( h[\tilde A,B]F \vv |    \tilde C^* \tilde Cw_h \big)    
-   \Re\big(    F\op (b(x,h\xi')) \vv |  \tilde C^* \tilde C  w_h \big)  \\
K_3&=  \Re\big( F\qq   |  \tilde C^* \tilde Cw_h  \big)     \\
K_4&=   \Re\big((  - h^{-1} \tilde C^* \tilde C\tilde A+ i h^{-1} \tilde C^* [ \tilde C , \op(\theta(\tilde \rho)) ]) w_h |  w_h \big)    
+   \Re\big(  h \tilde C^* [ \tilde C ,B]  w_h |  w_h \big)    \\
K_5&= -  \Re\big( ( h^{-1}\tilde A^2  + i h^{-1}[ \tilde A , \op(\theta(\tilde\rho)) ] -  h^{-1}\op (h\eta_1 \tilde \chi^2(\tilde\rho))) F\vv  |
 \tilde C^* \tilde Cw_h  \big)  .
\end{align*}
To estimate $K_1$,  observe that  the symbol of 
\( \tilde C^*( \d_{y_d} \tilde C) \) is in \( \in S(h^{-1/3}\est{h^{1/3}\eta_1}^{-1/2},g) ,\) then 
\begin{equation}
	\label{est: K 1}
|K_1|\lesssim | \tilde \Psi w_h||w_h|   .
\end{equation}

From symbol calculus the symbol of  $ ( \d_{y_d} \tilde A  )^*  \tilde C^* \tilde C $ is in $S(1,g)$, thus this operator is 
bounded on $L^2$. Clearly the terms $[\tilde A,B]$ coming from remainder term of $\d_{x_d}A$ (see  \textbf{iv)} 
Lemma~\ref{lem: FIO}) and  $\op (b(x,h\xi'))$ coming from remainder term of  
$F\op(\chi^2(x,h\xi')(R(x,h\xi')-1)) F^{-1} $ (see  \textbf{ii)} Lemma~\ref{lem: FIO}) are bounded on $L^2$. 
As $ \tilde C^* \tilde C$ has a symbol in $S(h^{-1/3}\est{h^{1/3}\eta_1}^{-1/2},g)$,
we obtain
\begin{equation}
	\label{est: K 2}
|K_2|\lesssim   |w_h| |v_h|+  |v_h|  |\tilde \Psi w_h|.
\end{equation}
For the same argument we have
\begin{equation}
	\label{est: K 3}
|K_3|\lesssim |\qq| | \tilde\Psi w_h|   . 
\end{equation}

To estimate the last term of $K_4$ we write
\(
  \tilde C^* [ \tilde C ,B]=   \tilde C^* \tilde C B-  \tilde C^*  B\tilde C, 
\)
the first term gives a term estimated by 
\(
 | \tilde \Psi w_h||w_h|  
\)
and the second is estimated
\(
| \tilde Cw_h |^2\lesssim | \tilde \Psi w_h||w_h|  .
\)

To estimate the other terms of $K_4$ observe that 
 the symbol of 
 \(
  - h^{-1} \tilde C^* \tilde C\tilde A +i h^{-1} \tilde C^* [ \tilde C , \op(\theta) ])
\)
is  
 \(
  - h^{-1} \tilde  c^2 \tilde a + h^{-1} \tilde c \{ \tilde c , \theta(\tilde \rho) \} 
\)
modulo a symbol in 
\(
S(h^{-2/3}\est{h^{1/3}\eta_1}^{-1},g)
\)
and    
this  term can be estimate by 
\(
|  \tilde \Psi w_h|^2  .
\)
We compute
\begin{align*}
h^{1/6}\{ \tilde c,\theta(\tilde \rho)\}
&= \{ \tilde \chi_2 (\tilde \rho) ,\theta(\tilde \rho)\}\beta(\zeta)+\{\beta(\zeta) ,\theta(\tilde \rho)\} \chi_2(\tilde \rho)   \\
&= \{ \tilde \chi_2 (\tilde \rho) ,\theta(\tilde \rho)\}\beta(\zeta)+\{\zeta ,\theta(\tilde \rho)\} \chi_2(\tilde \rho)   \beta'(\zeta ).
\end{align*}
The term 
\(
 h^{-1} \tilde c  h^{-1/6}  \{ \tilde \chi_2 (\tilde \rho) ,\theta(\tilde \rho)\}\beta(\zeta) \in S(   h^{-1/3}\est{h^{1/3}\eta_1}^{-1/2},g),
\)
and the term of $K_4$ coming from this term can be estimate by $|  \tilde \Psi w_h||  w_h|$.
For the other term we have
\begin{align*}
 h^{-1/3}\{\zeta ,\theta(\tilde \rho)\} &=   \{  \tilde r_d^{-2/3} (\tilde \rho)  ,\theta(\tilde \rho)\}  \eta_1  
 +  \{  \eta_1   ,\theta(\tilde \rho)\}  \tilde r_d^{-2/3}(\tilde \rho).
\end{align*}
The term 
\(
 h^{-1} \tilde c   h^{-1/6}  \chi_2(\tilde \rho)   \beta'(\zeta ) h^{1/3}  \{  \tilde r_d^{-2/3} (\tilde \rho)  ,\theta(\tilde \rho)\}  \eta_1  
  \in S(   h^{-1/3}\est{h^{1/3}\eta_1}^{-1/2},g)
\)
and the term of $K_4$ coming from this term can be estimate by $|  \tilde \Psi w_h||  w_h|$.

Thus, modulo remainder terms, 
 the symbol of 
\(
  - h^{-1} \tilde C^* \tilde C\tilde A +i h^{-1} \tilde C^* [ \tilde C , \op(\theta) ])
\)
is  given by 
 \begin{align*}
L=&  - h^{-1} \tilde  c^2 \tilde a 
  + h^{-1} \tilde c h^{-1/6} \tilde  \chi_2(\tilde \rho)   \beta'(\zeta )   h^{1/3} \{  \eta_1   ,\theta(\tilde \rho)\}  \tilde r_d^{-2/3}(\tilde \rho) \\
  &\qquad= -h^{-1}  h^{-1/3}\tilde \chi_2^2(\tilde \rho)\beta^2(\zeta)
   h^{1/3}\tilde \chi (\tilde \rho) \tilde r_d^{1/3} (\tilde \rho)  \alpha(\zeta)  \\
 &\qquad\quad   - h^{-1}     h^{-1/6}\tilde \chi_2(\tilde \rho)\beta(\zeta)            
  h^{-1/6}  \tilde \chi_2(\tilde \rho)   \beta'(\zeta )   h^{1/3}       \tilde r_d
    \tilde r_d^{-2/3}(\tilde \rho)
\end{align*}
from \textbf{v)} of Lemma~\ref{lem: FIO}.
We thus obtain
 \begin{align*}
 L=  h^{-1}   \tilde \chi_2^2(\tilde \rho)  \beta(\zeta)    \tilde r_d^{1/3} (\tilde \rho) \big(   - \beta(\zeta) \tilde \chi (\tilde \rho) \alpha(\zeta) 
 - \beta'(\zeta )   \big)\in S( h^{-1},g).
 \end{align*}
As $ \tilde \chi $ is equal $1$ on the support of $ \tilde \chi_2^2$, $\beta\ge 0$ and 
\(
 - \beta(\zeta) \Re \alpha(\zeta) 
 - \beta'(\zeta )  \ge0
 \)
we have $\Re L\ge0$.  We can apply sharp G\aa rding inequality (see \cite[Theorem 18.6.7]{HormanderV3-2007}), we yield, taking account remainder terms
\begin{equation}
	\label{est: K 4}
K_4\ge -C\big( |\tilde \Psi w_h|^2 + | \tilde \Psi w_h||w_h|\big)  .
\end{equation}

The last term is $K_5$.
The symbol of 
\(
   h^{-1}\tilde A^2  + i h^{-1}[ \tilde A , \op(\theta(\tilde \rho) ) ] -  h^{-1}\op (h\eta_1 \tilde \chi^2(\tilde \rho)) 
\)
is 
\(
   h^{-1}\tilde a^2  +  h^{-1}\{ \tilde a ,\theta(\tilde \rho)  \} -  \eta_1 \tilde \chi^2(\tilde \rho) 
\)
modulo a symbol in $S(1,g)$.
We have
\begin{align*}
\{ \tilde a ,\theta (\tilde\rho)\}&= h^{1/3} \alpha(\zeta)  \{(\tilde \chi  \tilde r_d^{1/3} )(\tilde \rho) ,\theta(\tilde \rho)  \}
+h^{2/3}(\tilde \chi  \tilde r_d^{1/3} )(\tilde \rho)  \alpha'(\zeta)   \eta_1 \{    \tilde r_d^{-2/3}(\tilde \rho)  ,\theta(\tilde \rho)  \}  \\
&\quad+h^{2/3}(\tilde \chi  \tilde r_d^{-1/3} )(\tilde \rho) \alpha'(\zeta) \{   \eta_1  ,\theta(\tilde \rho)  \}.
\end{align*}
The first two terms give a term  estimated by 
\(
h^{1/3}\est{h^{1/3}\eta_1}^{1/2} \lesssim 1
\)
as  on the support of $\tilde \chi$ we have $|\eta_1|\lesssim h^{-1}$. 
Then both terms give  associated operators bounded on $L^2$.  Modulo a bounded operator on $L^2$ we have to consider
the symbol, taking account  \textbf{v)} of Lemma~\ref{lem: FIO}
\begin{align*}
&  h^{-1/3}   \tilde \chi ^2(\tilde \rho)  \tilde r_d^{2/3}(\tilde \rho)  \alpha^2(\zeta) 
-   h^{-1/3} (\tilde \chi  \tilde r_d^{2/3} )(\tilde \rho) \alpha'(\zeta)  
 -    \eta_1 \tilde \chi^2(\tilde \rho)  \\
 &=   h^{-1/3}   \tilde \chi ^2(\tilde \rho)  \tilde r_d^{2/3}(\tilde \rho)\big(
\alpha^2(\zeta) -\alpha'(\zeta) - h^{1/3}   \tilde r_d^{-2/3}(\tilde \rho) \eta_1
 \big)  \\
&\quad  -   h^{-1/3} (\tilde \chi  \tilde r_d^{2/3} ) (\tilde \rho)    (1-\tilde \chi(\tilde \rho) )   \alpha'(\zeta) .
 \end{align*}
The first term is null from differential equation satisfying by $\alpha$ and the value of $\zeta$.
We claim that 
\begin{equation}
	\label{claim: estimate on K5}
|\op \big(   h^{-1/3} (\tilde \chi  \tilde r_d^{2/3} ) (\tilde \rho)    (1-\tilde \chi(\tilde \rho) )   \alpha'(\zeta)   \big)  F\vv |	\lesssim |v_h|.
\end{equation}
The proof of the claim is given below.
With this claim, \eqref{est: v_h from  v_h} and what we do above, the operator
\(
 h^{-1}\tilde A^2  + i h^{-1}[ \tilde A , \op(\theta(\tilde \rho)) ] -  h^{-1}\op (h\eta_1 \tilde \chi^2(\tilde \rho)))
\)
gives a term bounded by 
\(
|v_h|.
\)
As the symbol of 
\(
 \tilde C^* \tilde C
\)
is in 
\(
S(h^{-1/3} \est{h^{1/3}\eta_1}^{-1/2} ,g),
\)
we obtain that 
\begin{equation}
	\label{est: K 5}
|K_5|\lesssim  |v_h|  |\tilde \Psi w_h|   . 
\end{equation}
From \eqref{est: w h by v h}, \eqref{eq: second lemma a la Tataru},  \eqref{est: K 1},  \eqref{est: K 2}, \eqref{est: K 3}, 
\eqref{est: K 4} and  \eqref{est: K 5} we obtain
\begin{equation*}
\frac12\d_{x_d}|C( h\d_{x_d}\vv-A\vv  )|^2\gtrsim -\big( |\qq|^2+|\tilde \Psi w_h|^2 +  |v_h|^2    +| h\d_{x_d}v_h|^2  \big).
\end{equation*}
Integrating this inequality between 0 and \( \sigma>0\),
we have, estimating  the term coming from \( |\tilde \Psi w_h| \)  by Proposition~\ref{lem: first est r-1}, 
\begin{equation}
	\label{est: fin second lemma a la Tataru}
|C(  h\d_{x_d}\vv-A\vv )  |^2(0)
\lesssim  \|\qq\|^2 + \|v_h\|_{H^1_{sc}(x_d>0)}^2 +  |C(  h\d_{x_d}\vv-A\vv )  |^2(\sigma).
\end{equation}
As 
\[
  |C(  h\d_{x_d}\vv-A\vv )  |^2(\sigma)= (\tilde C ^*\tilde C  w_h,w_h),
\]
we have from~\eqref{est: w h by v h}
\[
  |C(  h\d_{x_d}\vv-A\vv )  |^2(\sigma)\lesssim |\tilde \Psi w_h |^2(\sigma)+ |v_h|^2 (\sigma)+ | h\d_{x_d}v_h|^2(\sigma)
\]
Integrating estimate~\eqref{est: fin second lemma a la Tataru}  between  two positive values of \(\sigma\) and estimating 
as above the term \(  |\tilde \Psi w_h |^2(\sigma)\), we obtain the conclusion of Proposition~\ref{lem: alpha airy}.
\end{proof}
\begin{proof}[Proof of Claim~\eqref{claim: estimate on K5}]
As   $\kappa^* \tilde  \chi_1 =  \chi_1$, from Lemma~\ref{lem: FIO} we thus have 
\[
F^{-1}\op( \tilde  \chi_1(\tilde \rho))F= \op(  \chi_1(x,h\xi')) +h K,
\]
where $K$ is bounded on $L^2$. We then have
\(
F \vv =\op( \tilde  \chi_1(\tilde \rho))F v_h+ hK' v_h,
\)
where $K'$ is bounded on $L^2$. Then
\begin{align*}
&\op \big(   h^{-1/3} (\tilde \chi  \tilde r_d^{2/3} ) (\tilde \rho)    (1-\tilde \chi(\tilde \rho) )   \alpha'(\zeta)   \big)  F\vv \\
&\quad= \op \big(   h^{-1/3} (\tilde \chi  \tilde r_d^{2/3} ) (\tilde \rho)    (1-\tilde \chi(\tilde \rho) )   \alpha'(\zeta)   \big)
\big( \op( \tilde  \chi_1(\tilde \rho))F v_h+ hK' v_h  \big).
\end{align*}
The first term coming from \(  \op( \tilde  \chi_1(\tilde \rho)) \) gives an operator with null symbol.
As 
\[
  h^{-1/3} (\tilde \chi  \tilde r_d^{2/3} ) (\tilde \rho)    (1-\tilde \chi(\tilde \rho) )   \alpha'(\zeta)\in
  S(h^{-1/3},g),
\]
 then the second term is also bounded by \( |v_h| \). This proves the claim.
\end{proof}

%%%%%%%%%%%%%%%%%%%%%%%%%%%%%%%
%
%   Proposition
%
%%%%%%%%%%%%%%%%%%%%%%%%%%%%%%%% 
\begin{proposition}
	\label{prop: boundary term diffractive Neumann}
Let $(x_0',0)\in \d\Omega_N$.
Let $\chi_4\in \Con_0^\infty(\R)$ be supported on a \nhd of 0 and $\chi_4=1 $ in a  \nhd of 0.
Let $(x_0',0,\xi_0')$ and  $U_0$ be as in the statement of Lemma~\ref{lem: Symplectic transformation}.
Let $\ell\in \Con_0^\infty(\R^{d}\times \R^{d-1} ) $ supported on $\{ \chi_2=1\}$  for every $x_d$, where 
$\kappa^*\tilde \chi_2= \chi_2$. We moreover assume  $\d_{x_d} R(x,\xi')<0$ on support of $\ell$. 
We have 
\[
\lim_{\eps\to 0} \lim_{h\to 0}  \Big| \Big( 
\op\big(  \ell(x',0,h\xi') \chi_4\big(( R(x',0,h\xi')-1)/\eps\big)  ( R(x',0,h\xi')-1 )\big)  (v_h)_{|x_d=0}|  (v_h)_{|x_d=0}  \Big)_0  \Big|=0
\]
\end{proposition}
\begin{proof}
We can assume that the support of $\ell$ is contained in $\{ \chi_1=1 \}$. We then have from symbol calculus
\begin{align*}
&\lim_{\eps\to 0} \lim_{h\to 0} \Big( \op\big(  \ell(x',0,h\xi') \chi_4\big(( R(x',0,h\xi')-1)/\eps\big) 
 ( R(x',0,h\xi')-1 )\big)  (v_h)_{|x_d=0}|  (v_h)_{|x_d=0}  \Big)_0\\
&\qquad=\lim_{\eps\to 0} \lim_{h\to 0} \Big( \op \big(  \ell(x',0,h\xi') \chi_4\big(( R(x',0,h\xi')-1)/\eps\big) 
 ( R(x',0,h\xi')-1 )\big)  (\vv)_{|x_d=0}| ( \vv)_{|x_d=0}  \Big)_0
\end{align*}
Let $z_h=(F\vv)_{|x_d=0}$. From  Proposition~\ref{lem: alpha airy}  we obtain
\(
|\tilde C \tilde Az_h| 
\)
is bounded.  By symbol calculus and the support properties of $\tilde \chi$ and $\tilde \chi_2$
\[
\tilde C \tilde A= h^{1/6}\op\big(\tilde \chi_2(\tilde \rho) \beta (\zeta ) 
 \tilde r_d^{1/3}(\tilde \rho)  \alpha(\zeta ) \big)
\]
modulo an operator with symbol in $S(h^{1/2}\est{ h^{1/3}\eta_1}^{-3/4}  ,g)$. From properties of traces, 
see \eqref{properties: sequence and measure}, this remainder  term goes to 0 as $h$ to 0.

Let $\tilde \ell$  be such that $\kappa^*\tilde \ell=\ell$ and 
$\kappa^*   \chi_4(\eta_1/\eps)= \chi_4( (R(x',0,\xi')-1)/\eps)$.  In what follow, to be 
coherent with our notation we define $\tilde \chi_4=\chi_4$ and we use the notation $\tilde \chi_4$ when the function 
is defined in $(y,\eta)$ variables.

From   Lemma~\ref{lem: FIO} we have 
\begin{multline*}
F^{-1}\op(h\eta_1 \tilde \ell (\tilde \rho)  \tilde \chi_4  (h\eta_1/\eps  ))F\\
= \op \big(  \ell(x',0,h\xi') \chi_4\big(( R(x',0,h\xi'-1)/\eps\big) 
 ( R(x',0,h\xi')-1 )\big) +h\op(r_0(x,h\xi')),
\end{multline*}
where  $r_0$ is of order 0. The term coming from $r_0$ goes to 0 as $h$ to 0, from  properties of traces.
Then it is sufficient to prove that 
\(
\lim_{\eps\to 0} \lim_{h\to 0} \Big( \op(h\eta_1 \tilde \ell (\tilde \rho)  \tilde \chi_4  (h\eta_1/\eps  ))  z_h| z_h \Big)_0=0.
\)
Considering the symbol 
\(
h\eta_1 \tilde \ell (\tilde \rho)  \tilde \chi_4  (h\eta_1/\eps  )\in S(h^{2/3}\est{h^{1/3}\eta_1},g)
\)
and from support properties of $  \tilde \ell  $ and $\tilde \chi_2$, we have
\begin{align*}
 \op(h\eta_1 \tilde \ell (\tilde \rho)  \tilde \chi_4  (h\eta_1/\eps  )) 
=\tilde\gamma^*
\op\big(h^{2/3}\eta_1 \tilde \ell (\tilde \rho)  \tilde \chi_4  (h\eta_1/\eps  )      
\beta^{-2} (\zeta )  \tilde r_d^{-2/3}(\tilde \rho) | \alpha(\zeta ) |^{-2} \big)
 \tilde\gamma
\end{align*}
where 
\(
 \tilde\gamma = \op\big( h^{1/6}    \tilde \chi_2 (\tilde \rho)   \beta (\zeta )  \tilde r_d^{1/3}(\tilde \rho)  \alpha(\zeta )\big),
\)
modulo an operator with symbol in $S(h,g)$ then this last term involves a term going to 0 as $h$ to 0.
Then we obtain an estimation, modulo a term going to 0 as $h$ to 0, 
\def\zz{{\textrm{z}}_h}
\begin{multline}
	\label{est: Neumann diffractif}
| \Big( \op(h\eta_1 \tilde \ell (\tilde \rho)  \tilde \chi_4  (h\eta_1/\eps  ))  z_h| z_h \Big)_0| 
\\
\lesssim  | \op\big(h^{2/3}\eta_1 \tilde \ell (\tilde \rho)  \tilde \chi_4  (h\eta_1/\eps  )      
\beta^{-2} (\zeta )  \tilde r_d^{-2/3}(\tilde \rho) | \alpha(\zeta ) |^{-2} \big)
\zz\big)|  | \zz|,
\end{multline}
where 
\(
\zz= \op\big( h^{1/6}    \tilde \chi_2 (\tilde \rho)   \beta (\zeta )  \tilde r_d^{1/3}(\tilde \rho)  \alpha(\zeta )\big)z_h.
\)
Observe from Proposition~\ref{lem: alpha airy}, $|\zz|$ is bounded.

We claim that 
\(
h^{2/3}\eta_1 \tilde \chi_4  (h\eta_1/\eps  )\in S(\est{h^{1/3}\eta_1}^{1/2}, g).
\)
Indeed 
\(
h^{2/3}|\eta_1|\lesssim \est{ h\eta_1 }^{1/2} \est{h^{1/3}\eta_1}^{1/2},    
\)
and this gives the sought estimate.
We have for $k\ge1$, 
\[
\d_{\eta_1}^k\big(  h^{2/3}\eta_1 \tilde \chi_4  (h\eta_1/\eps  ) \big) =h^{2/3}  (h/\eps)^{k-1}    \tilde \chi_4 ^{(k-1)} (h\eta_1/\eps  )
+ h^{2/3} \eta_1 (h/\eps)^{k}    \tilde \chi_4 ^{(k)} (h\eta_1/\eps  ).
\]
As $h\eta_1$ is bounded on the support of $  \tilde \chi_4 ^{(k)} (h\eta_1/\eps  )$ both terms are bounded by $h^{k-1/3}$.
From estimate
\(
h^{2/3}\est{h^{1/3}\eta_1}\lesssim  \est{h\eta_1}
\)
we have 
\[
h^{k-1/3}\lesssim h^{k/3} \big(   \est{h\eta_1}  /  \est{h^{1/3}\eta_1}   \big)^{k-1/2},
\]
which proves the claim. Observe that the constant of estimation are not uniform with respect $\eps$.

With the previous claim and as 
\(
\beta^{-2} (\zeta )  | \alpha(\zeta ) |^{-2}\lesssim \est{h^{1/3}\eta_1}^{-1/2}, 
\)
we have 
\[
L(y',\eta')=h^{2/3}\eta_1 \tilde \ell (\tilde \rho)  \tilde \chi_4  (h\eta_1/\eps  )      
\beta^{-2} (\zeta )  \tilde r_d^{-2/3}(\tilde \rho) | \alpha(\zeta ) |^{-2} \in S(1,g).
\]
As 
\(
h^{2/3}|\eta_1|\lesssim | h\eta_1 |^{1/2} \est{h^{1/3}\eta_1}^{1/2}\lesssim \sqrt{\eps}\est{h^{1/3}\eta_1}^{1/2}, 
\)
on the support of $  \tilde \chi_4  (h\eta_1/\eps  )   $, we deduce from G\aa rding inequality 
(see \cite[Theorem 18.6.7]{HormanderV3-2007})
that the operator norm from $L^2$ to $L^2$ of
$\op(L)$ is bounded by $ C \sqrt{\eps} + C_\eps h^{1/3}$ where $C$ is independent of $\eps$ and $ C_\eps $ may depend on 
$\eps$. From that and \eqref{est: Neumann diffractif} we deduce the result.
\end{proof}
%%%%%%%%%%%%%%%%%%%%%
%
%   Section   Proof of Lemmas
%
%%%%%%%%%%%%%%%%%%%%%%%
\section{Proof of Lemmas}
	\label{sec: proof of lemmas}

%%%%%%%%%%%%%%%%%%%%%%%%%%%%%%%
%
%   Lemma
%
%%%%%%%%%%%%%%%%%%%%%%%%%%%%%%%% 
\begin{lemma}
	\label{lem: estimate by tilde Psi}
Let $\tilde \chi_3\in\Con_0^\infty(U_1)$ be such that $(1-\tilde \chi_3)\tilde \chi_1=0$ where $\kappa^* \tilde \chi_1=\chi_1$.
  We have 
\begin{align*}
& |\tilde  \Psi w_h | \lesssim  |h^{-1/3}\op(\tilde \chi_3(\tilde \rho) \est{h^{1/3}\eta_1}^{-1/2})w_h| + |v_h|+|h\d_{x_d}v_h| . \\
\end{align*}
\end{lemma}
\begin{proof}
We write 
\(
\tilde \Psi = h^{-1/3}\op(\tilde \chi_3(\tilde \rho) \est{h^{1/3}\eta_1}^{-1/2})
+ \op( h^{-1/3} (1-\tilde \chi_3(\tilde \rho))\est{h^{1/3}\eta_1}^{-1/2}),
\)
it suffices to estimates terms coming from the second operator.
Recall that  $w_h=F(  h \d_{x_d}\vv-A\vv)$, we have 
\[
 \op(h^{-1/3} (1-\tilde \chi_3(\tilde \rho))\est{h^{1/3}\eta_1}^{-1/2})FA= \op(h^{-1/3} (1-\tilde \chi_3(\tilde \rho))\est{h^{1/3}\eta_1}^{-1/2})\tilde AF,
\]
as 
\(
h^{-1/3} (1-\tilde \chi_3(\tilde \rho))\est{h^{1/3}\eta_1}^{-1/2}\in S(h^{-1/3} \est{h^{1/3}\eta_1}^{-1/2}  ,g)
\)
and 
\(
\tilde a \in S(h^{1/3} \est{h^{1/3}\eta_1}^{1/2}  ,g),
\)
we have
\[
| \op(h^{-1/3} (1-\tilde \chi_3(\tilde \rho))\est{h^{1/3}\eta_1}^{-1/2})FA\vv   | \lesssim |v_h|.
\]   
We have 
\begin{equation}
	\label{eq: commut d x b}
 F h\d_{x_d} \op (\chi_1(x,h\xi'))
 = F \op( \chi_1(x,h\xi'))h\d_{x_d} 
 + F h \op(\d_{x_d}\chi_1(x,h\xi')).
\end{equation}
The second term gives
\[
| \op(h^{-1/3} (1-\tilde \chi_3(\tilde \rho))\est{h^{1/3}\eta_1}^{-1/2}) F h \op(\d_{x_d}\chi_1(x,h\xi'))v_h|\lesssim h^{1/3} |v_h|.
\]
As   $\kappa^* \tilde  \chi_1 =  \chi_1$, from Lemma~\ref{lem: FIO} we thus have 
\(
F^{-1}\op( \tilde  \chi_1(\tilde \rho))F= \op(  \chi_1(x,h\xi')) +h K,
\)
where $K$ is bounded on $L^2$.
Then the first term at the right hand side of~\eqref{eq: commut d x b} gives
\begin{multline*}
| \op(h^{-1/3} (1- \tilde \chi_3(\tilde \rho))\est{h^{1/3}\eta_1}^{-1/2})  F  \op( \chi_1(x,h\xi'))h\d_{x_d} v_h|  \\
\lesssim | \op(h^{-1/3} (1-\tilde \chi_3(\tilde \rho))\est{h^{1/3}\eta_1}^{-1/2})  
\big( \op( \tilde \chi_1(\tilde \rho))  -hF K F^{-1}\big)
Fh\d_{x_d} v_h| \le |  h\d_{x_d} v_h |,
\end{multline*}
as the both terms are bounded on $L^2$ indeed the asymptotic expansion of the first symbol is null and  the second
is bounded by $h^{2/3}$ times an operator bounded on $L^2$. This concludes the proof of the lemma.
\end{proof}

\begin{proof}[Proof of Lemma~\ref{lem: FIO}]
We recall that Zworski use Weyl quantification. We give the proof in this context. From that, it is easy to obtain the result for classical
quantification.  We denote by $\op^w_{sc} (a)$ the operator associated with symbol $a$ by the Weyl quantification.
Items \textbf{i)} and  \textbf{ii)}  come for Zworski~\cite[Theorem 11.5]{Zworski-2012}.
To prove the others Items we have to use the construction of $F$ given by Zworski~\cite[Section 11.1 and 11.2]{Zworski-2012}.
Let $\kappa_t$ a smooth family of symplectic transformations, $t\in[0,1]$, $\kappa_0=Id$ and $\kappa_1=\kappa$ and 
$q_t\in\Con_0^\infty(U_0)$ be real valued,
such that $\d_t \kappa_t=(\kappa_t)_*H_{q_t}$ (see \cite[Theorem 11.4]{Zworski-2012}). Let $Q(t)=\op_{sc}^w(q_t)$ 
(here as we use 
the Weyl quantification,  $Q(t)$ is selfadjoint). Let $F(t) $ the solution of 
\begin{equation*}
\begin{cases}
hD_t F(t)+ F(t) Q(t)=0\\
F(0)=Id .
\end{cases}
\end{equation*}
The Fourier Integral Operator we search,   is $F=F(1)$. We then have  for $G(t)=\d_{x_d}F(t)$
\begin{equation*}
\begin{cases}
hD_t G(t)+ G(t) Q(t)=-F(t) \d_{x_d} Q(t)\\
G(0)=0.
\end{cases}
\end{equation*}
The Duhamel formula yields
\[
G(t)F^{-1}(t) =-ih^{-1}\int_0^tF(\sigma) \d_{x_d} Q(\sigma)F^{-1}(\sigma)d\sigma.
\]
Taking $\tilde q_0(t)$ such that  $\kappa_t^*\tilde q_0(t)= \d_{x_d} q(t)$ which it is possible as  $\kappa_t$ is a diffeomorphism, 
we have from Item \textbf{ii)}, 
\(
 F^{-1}(\sigma)  \op_{sc}^w(\tilde q_0(\sigma))F(\sigma) =  \d_{x_d} Q(\sigma) + h\op_{sc}^w(q_1).
\)
We can repeat the construction taking $\kappa_t^*\tilde q_1(t)= q_1(t)$ and we have
\(
 F^{-1}(\sigma)  \op_{sc}^w(\tilde q_0(\sigma)-h \tilde q_1(\sigma) )F(\sigma) 
 =  \d_{x_d} Q(\sigma) + h^2\op_{sc}^w(q_2(\sigma)).
\)
This implies that 
\(
F(\sigma) \d_{x_d} Q(\sigma)F^{-1}(\sigma) = \op_{sc}^w(\tilde q_0(\sigma)-h \tilde q_1(\sigma) ) +h^2  B(\sigma),
\)
where $B(\sigma)$ is bounded on $L^2$ uniformly with respect $\sigma$. From that and taking $t=1$ we deduce
Item \textbf{iii)}.

To prove Item~\textbf{iv)} we have  \(\d_{x_d } (F^{-1})=-F^{-1} \d_{x_d}  (F) F^{-1}\), we deduce
\begin{align*}
\d_{x_d} A&= -F^{-1} (\d_{x_d}  F) F^{-1} \tilde A F+ F^{-1}(   \d_{x_d}   \tilde A )F+  F^{-1}   \tilde A (\d_{x_d} F)  \\
&=  -F^{-1}\big(   ih^{-1}\op_{sc}^w (\theta)+hB  \big) \tilde A F  + F^{-1}(   \d_{x_d}   \tilde A )F
+  F^{-1}   \tilde A \big(   ih^{-1}\op_{sc}^w (\theta)+hB  \big) F  \\  
&=   F^{-1}\big(  (   \d_{x_d}   \tilde A )+ ih^{-1}  [  \tilde A ,\op_{sc} ^w(\theta)    ]+  h [  \tilde A , B ] \big) F,
\end{align*}
Which gives Item~\textbf{iv)} as there exist a symbol $\theta_1$ such that $\op_{sc}(\theta_1)=\op_{sc}^w(\theta)$. 

Let $\chi_1$ and $\chi_2$ be $\Con_0^\infty$ functions such that $\chi_1 =1$ and   $\chi_2 =1$ in a \nhd of $(x_0',0,\xi_0')$,
and we assume $\chi_2$ supported on $\{ \chi_1=1 \}$. Applying \textbf{iv)} to $A=(R-1)\chi_1$, taking account \textbf{ii)}, 
we obtain $\kappa^*(\{\eta_1,\theta\}\tilde\chi_2(y,\eta'))=\chi_2\d_{x_d}R$ where $\kappa^* \tilde\chi_2=\chi_2$.
We deduce  Item \textbf{v)} taking the previous formula on $\{\chi_2=1  \}$.
\end{proof}

\begin{proof}[Proof of Lemma~\ref{lem: alpha properties}]
The asymptotic expansion of $\Ai$ is well-known  (see \cite[Formula 10.4.59]{AS-1964}. 
We recall that for $z\in\C$ with $|\arg z|<\pi$ we have
\begin{align}
	\label{est: airy asympt.}
&\Ai(z)\sim 2^{-1} \pi^{-1/2}z^{-1/4}e^{-\zeta}\sum_{n=0}^\infty(-1)^nc_n\zeta^{-n}, \text{ with } \zeta= \frac23 z^{2/3}  \\
&\Ai'(z)\sim -2^{-1} \pi^{-1/2}z^{1/4}e^{-\zeta}\sum_{n=0}^\infty(-1)^nd_n\zeta^{-n},  \notag
\end{align}
where $c_0=d_0=1$,  $\DS c_1=\frac{ 5}{ 2^{3}3^{2}}$ and $\DS d_1= -\frac{7c_1}{5}$. From that we obtain
\begin{align*}
\frac{\Ai'(z)}{\Ai(z)}\sim -z^{1/2}\sum_{n=0}^\infty f_n\zeta^{-n}\sim -z^{1/2}\sum_{n=0}^\infty  \ell_n z^{-3n/2},
\end{align*}
where $f_0=\ell_0=1$, $f_1=1/6$ and $\ell_1=1/4$.

For $x>0$, as $ \omega^{3/2}=-1$,  we have,
\begin{equation*}
\alpha(x)\sim \omega \omega^{1/2} x^{1/2} \sum_{n=0}^\infty  \ell_n \omega^{-3n/2} x^{-3n/2}
\sim -  x^{1/2} \sum_{n=0}^\infty (-1)^n  \ell_n x^{-3n/2}.
\end{equation*}
As this asymptotic expansion is also valid for derivative with respect $x$ we deduce  Item \textbf{i)}.

For $x<0$, we have $\omega x = -e^{2i\eps \pi/3}|x|= e^{-i\eps \pi/3}|x|$ to have $|\arg (\omega x)|<\pi$.
As $-\omega e^{-i\eps \pi/6}= -\eps i$ we obtain 
\begin{equation*}
\alpha(x)\sim  \eps i |x|^{1/2} \sum_{n=0}^\infty  \ell_n (e^{-i\eps \pi/3})^{-3n/2} |x|^{-3n/2}
\sim  \eps i |x|^{1/2} \sum_{n=0}^\infty (\eps i)^n  \ell_n|x|^{-3n/2}.
\end{equation*}
This gives Item \textbf{ii)} from properties of this asymptotic expansion.

Let $F(z)= \Ai'(z)/\Ai(z)$ we have $F'(z)=z-A^2(z)$ for $z$ different of  a zero of $\Ai$ which are on the  negative real axis. 
As $\alpha(x)= -\omega F(\omega x)$ we have 
\begin{equation*}
\alpha'(x)=-\omega^2 F'(\omega x)=-\omega^2\big(  \omega x-F^2(\omega x) \big)= \alpha^2(x)-x.
\end{equation*}
This gives Item \textbf{iv)}.

Item~\textbf{iii)} is probably classical but we do not find this property in literature. Here we give a proof of that.
Let $\alpha_1(x)$ and $\alpha_2(x)$ real valued be such that $\alpha(x)=\alpha_1(x)+i\alpha_2(x)$. We have
\begin{equation}
	\label{syst: alpha}
\begin{cases}
\alpha_1'(x)=\alpha_1 ^2(x) - \alpha_2^2(x)-x\\
\alpha_2'(x)=2\alpha_1(x)\alpha_2(x).
\end{cases}
\end{equation}
We also use a nice formula given in \cite[Section 3]{Vodev-2015}
\begin{equation*}
F(z)= C_1+\sum_{j=1}^\infty\big(   (z-\nu_j )^{-1} +\nu_j^{-1}\big),
\end{equation*}
where $\nu_j$'s are the zeros of $\Ai$ (observe that $\nu_j<0$) and $C_1\in \R$ is an explicit negative constant. We deduce 
\begin{equation*}
F'(z)= -\sum_{j=1}^\infty   (z-\nu_j )^{-2} 
\end{equation*}
It is easy to prove that both series converge.

We fix $\eps=1$, observe that $\overline{\alpha(x)}$ is the $\alpha(x) $ defined with $\eps=-1$. 
Observe that $\alpha_2(x)>0 $, indeed if $\alpha_2(x_0)=0$ for some $x_0\in\R$, $\alpha_2$ is identically null by uniqueness of 
System~\eqref{syst: alpha}. Then $\alpha_2(x)>0 $ as it is true for $x\ll0$ from Item~\textbf{ii)}.

We have   
\begin{align*}
\alpha'(x)= -\omega^2F'(\omega x)&= \omega^2\sum_{j=1}^\infty   (\omega x-\nu_j )^{-2}  \\
&=\sum_{j=1}^\infty   |\omega x-\nu_j| ^{-4} \big(  x^2-2\omega \nu_j x + \omega^2\nu_j^2\big).
\end{align*}
We deduce that 
\begin{align*}
\alpha_2'(x)&=\sum_{j=1}^\infty   |\omega x-\nu_j| ^{-4} \big( - \sqrt{3} \nu_j x- \frac{\sqrt{3}}{2}\nu_j^2\big)   
\end{align*}
Assuming $x\le 0$ we have $\alpha_2'(x)<0$ and from~\eqref{syst: alpha},  $\alpha_1(x) <0$.

Now for $x\ge 0$ we compute 
\begin{equation*}
(\alpha_1/\alpha_2)'(x)=\frac{-\alpha_1^2(x)\alpha_2  (x) -\alpha_2^3(x)-x\alpha_2(x)}{\alpha_2^2(x)}< 0.
\end{equation*}
This implies $(\alpha_1/\alpha_2)(x)\le (\alpha_1/\alpha_2)(0)<0$ then $\alpha_1(x)<0$. This concludes Item~\textbf{iii)}.
\end{proof}

\begin{proof}[Proof of Lemma~\ref{lem: properties beta}]
Let $\gamma_1(x) = |\Ai(\omega x)|   $  
and $\gamma_2(x)= C_0\est x^{-1/4}$ where $C_0>0$ will be fixed below. 

We shall begin to prove $\gamma_1'+\gamma_1\Re\alpha=0$. Writing 
\(
\gamma_1(x) =\sqrt{\Ai(\omega x)\Ai(\bar\omega x)} 
\)
we obtain
\[
\gamma_1'(x)=\frac{\omega\Ai'(\omega x)\Ai(\bar\omega x)+ \bar\omega\Ai(\omega x)\Ai'(\bar\omega x)  }{2|\Ai(\omega x)|  },
\]
and 
\[
\Re\alpha(x)=-\Re\frac{\omega\Ai'(\omega x)\Ai(\bar\omega x)}{|\Ai(\omega x)|^2}=-\frac{\gamma_1'(x) }{\gamma_1(x) }.
\]
From definition of $\gamma_2$ we have for $x>0$,  $\gamma_2>0$,  $\gamma_2'<0$ and as $\Re \alpha<0$ we have
 $\gamma_2'+\gamma_2 \Re\alpha<0$ . 

Let $\chi_0\in\Con^\infty(\R)$ be such that $\chi_0(x)=0$ if $x\le 0$,  $\chi_0(x)=1$  if
$x\ge1$ and we assume $\chi_0'\ge0$ on $\R$. Let $\beta=\chi_0 \gamma_2+(1-\chi_0)\gamma_1$. Clearly $\beta$ is a  smooth
function. We have
\(
\beta'= \chi_0 \gamma_2'+(1-\chi_0)\gamma_1'+  \chi_0'( \gamma_2-\gamma_1).
\)
As $\gamma_1>0$, if $C_0$ is chosen sufficiently small, $ \gamma_2-\gamma_1<0$ on the support of $ \chi_0'$. 
As $ \chi_0'\ge0$,  $ \chi_0'( \gamma_2-\gamma_1)\le0$. This and above properties imply \textbf{i)}.
We only have to prove \textbf{ii)} for $x<0$. That is a consequence of asymptotic expansion of Airy 
function~\eqref{est: airy asympt.}. Indeed for $x<0$, $\omega x= |x|e^{-i\eps\pi/3}$, then $(\omega x)^{3/2}= -\eps i|x|^{3/2}$. We then
have
\(
\beta(x)= |\Ai(\omega x)|  \sim  2^{-1} \pi^{-1/2}|x|^{-1/4}\Big|\sum_{n=0}^\infty(-1)^nc_n\zeta^{-n}\Big|, 
\text{ with } \zeta= \frac23 (\omega x)^{2/3}.
\)
Clearly the asymptotic expansion satisfies symbol estimates. This asymptotic expansion also gives estimates 
\textbf{iii)} for $x<0$ 
with $|x|$ sufficiently large.  For $x>0$ \textbf{iii)} is obvious, and by construction $\beta>0$. This achieves the proof.
\end{proof}


\begin{thebibliography}{}


\bibitem{AS-1964} \textsc{Abramowitz, M.,  Stegun, I. A.,} Handbook of mathematical functions with formulas, graphs, and
              mathematical tables, National Bureau of Standards Applied Mathematics Series, 
              For sale by the Superintendent of Documents, U.S. Government
              Printing Office, Washington, D.C., 1964.

\bibitem{AKR-2016}   \textsc{Aloui, L., Khenissi, M., Robbiano, L.,}  The Kato Smoothing Effect 
for Regularized Schr\"odinger Equations in Exterior Domains.  \textit{Ann. Inst. H. Poincar\'e Anal. Non Lin\'eaire}, online,  2017

\bibitem{ALM-2016} \textsc{Anantharaman, N.,  L{\'e}autaud, M., Maci{\`a}, F.,} Wigner measures and
observability for the {S}chr\"odinger equation on the disk. \textit{Invent. Math.}  2016 \textbf{206}, 485--599



\bibitem{ALM-2016-2} \textsc{Anantharaman, N.,  L{\'e}autaud, M., Maci{\`a}, F.,} Delocalization 
of quasimodes on the disk. \textit{C. R. Math. Acad. Sci. Paris}, 2016 \textbf{354},  257--263

  \bibitem{AM-1977} \textsc{Andersson, K. G., Melrose, R. B.,} The propagation of singularities along gliding rays.  
  \textit{Invent. Math.}, 1997 \textbf{41}, 197--232
    
     
\bibitem{BD} \textsc{Batty, C., Duyckaerts, T.,} Non-uniform stability for bounded semi-groups on 
Banach spaces. 
\textit{J. Evol. Equ.}, 2008 \textbf{8}, no. 4, 765-780.

\bibitem{BLM-1999} \textsc{Bey, R., Loh{\'e}ac, J.-P.,  Moussaoui, M.,} Singularities of the solution of a 
mixed problem for a general second order elliptic equation and boundary stabilization of the wave 
equation. \textit{J. Math. Pures Appl.}, 1999 \textbf{78}, 1043--1067.



\bibitem{BLR1} \textsc{Bardos, C., Lebeau, G., Rauch, J.,}   Sharp
sufficient conditions for the observation, control and stabilization of
waves from the boundary.  \textit{SIAM J. Control Optim.}, 1992 \textbf{30}, no. 5,
1024-1065.

\bibitem{Bony-1981} \textsc{Bony, J-M.,} Calcul symbolique et propagation des singularit\'es pour les
              \'equations aux d\'eriv\'ees partielles non lin\'eaires. \textit{Ann. Sci. \'Ecole Norm. Sup. (4)}, 1981 \textbf{14}, 209--246


\bibitem{BT-2010} \textsc{Borichev, A., Tomilov, Y.,} Optimal polynomial decay of functions and operator semigroups.
\textit{Math. Ann.} 2010 \textbf{347}, 455--478.



\bibitem{Burq-1997} \textsc{Burq, N.,} Mesures semi-classiques et mesures de d\'efaut. 
S{\'e}minaire Bourbaki, Vol. 1996/97, \textit{Ast\'erisque}, 1997 \textbf{245}, 167--195.

\bibitem{Burq-Lebeau2001} \textsc{Burq, N., Lebeau, G.,} Mesures de d\'efaut de compacit\'e, application 
au syst\`eme de {L}am\'e. \textit{Ann. Sci. \'Ecole Norm. Sup.}, 2001 \textbf{34}, 817--870.
              
              
\bibitem{CLO} \textsc{Cornilleau, P., Loh\'eac, J.-P., Osses, A .,}
Nonlinear Neumann boundary stabilization of the wave equation using rotated multipliers. 
\textit{J. Dyn. Control Syst.},  2010 \textbf{16}, no. 2, 163-188. 


\bibitem{CR2014} \textsc{Cornilleau, P., Robbiano, L.,} 
Carleman estimates for the Zaremba Boundary Condition and Stabilization of Waves.
\textit{American Journal Math.}, 2014 \textbf{136}, 393-444. 


\bibitem{DA}   \textsc{Davies, E.B.,} \newblock  Spectral theory and
differential operators. Cambridge studies in advanced mathematics, 
\textbf{42} Cambridge Univers. press.

\bibitem{DBLR-2014} \textsc{Dehman, B., Le Rousseau, J.,  L{\'e}autaud, M.,} Controllability of two 
coupled wave equations on a compact manifold.  Arch. Ration. Mech. Anal. 2014 \textbf{211}, 113--187

\bibitem{Eskin-1977} \textsc{Eskin, G.,} Parametrix and propagation of singularities for the interior
              mixed hyperbolic problem. J. Analyse Math. 1977 \textbf{32}, 17--62



\bibitem{EN-2000} \textsc{Engel, K.-J., Nagel, R.,} One-parameter semigroups for linear evolution equations, 
Graduate Texts in Mathematics 194, Springer-Verlag, New York, {2000}

      
      
\bibitem{Fu-2015} \textsc{Fu, X.,} Stabilization of hyperbolic equations with mixed boundary
              conditions. Math. Control Relat. Fields, 2015 \textbf{5}, 761--780


\bibitem{Gearhart-1978} \textsc{Gearhart, L.,} Spectral theory for contraction semigroups on {H}ilbert space.  Trans. Amer. Math. Soc.,
1978, \textbf{236}, 385--394



\bibitem{Gerard-1990} \textsc{G{\'e}rard, P.,} 
Mesures semi-classiques et ondes de {B}loch
\textit{S\'eminaire sur les \'{E}quations aux {D}\'eriv\'ees
              {P}artielles, 1990--1991}, Exp.\ No.\ XVI, 19


\bibitem{Gerard-1991} \textsc{G{\'e}rard, P.,} Microlocal defect measures. \textit{Comm. Partial 
Differential Equations}, 1991 \textbf{16}, 1761--1794

\bibitem{GL-1993} \textsc{G{\'e}rard, P., Leichtnam, \'{E}.,} 
Ergodic properties of eigenfunctions for the Dirichlet problem
\textit{Duke Math. J.}, 1993, \textbf{71}, 559--607


\bibitem{Grigis-1976} \textsc{Grigis, A.,} Hypoellipticit\'e et param\'etrix pour des op\'erateurs
              pseudodiff\'erentiels \`a caract\'eristiques doubles, Journ\'ees: \'Equations aux {D}\'eriv\'ees {P}artielles de {R}ennes
              (1975),  Ast\'erisque,  1976, \textbf{34--35},  183--205


\bibitem{HS:1989}  \textsc{Helffer, B. and Sj\"{o}strand, J.} {\'{E}quation de 
{S}chr\"{o}dinger avec champ magn\'{e}tique et \'{e}quation de {H}arper}, Schr\"{o}dinger
operators, ({S}\o nderborg, 1988), Lecture Notes in Phys.
 \textbf{345}, 118--197

		

\bibitem{HormanderV3-2007} \textsc{H{\"o}rmander, L.,} The analysis of linear partial differential 
operators {vol. III}. Springer, Berlin


\bibitem{Huang85} \textsc{Huang, F. L.,} Characteristic conditions for exponential stability of linear
              dynamical systems in {H}ilbert spaces. \textit{Ann. Differential Equations}, 1985, \textbf{1}, 43--56
              
 \bibitem{LRL-2012} \textsc{Le Rousseau, J.,  Lebeau,G.,} On Carleman estimates for elliptic and parabolic operators. Applications to unique continuation and control of parabolic equations. \textit{ESAIM Control Optim. Calc. Var.} 2012, \textbf{18}, 712--747
 
              
\bibitem{LRLTT-2016} \textsc{Le Rousseau, J., Lebeau G., Terpolilli, P., Tr\'elat, E.,} Geometric 
control condition for the wave equation with a time-dependent observation domain.
\textit{Anal. PDE},  2017, \textbf{10}, 983--1015
         
         

%
\bibitem{Leb} \textsc{Lebeau, G.}  Contr\^{o}le et stabilisation hyperboliques. 
\textit{S\'eminaire \'Equations aux d\'eriv\'ees partielles (\'Ecole Polytechnique)}, 1990, Expos\'e \textbf{16}, 1-16.
%
\bibitem{Lebeau:96} {\sc Lebeau, G.,}  \newblock \'{E}quation des ondes amorties.
\newblock In {\em Algebraic and geometric methods in mathematical physics
({K}aciveli, 1993)}, {\bf 19}. \textit{Math. Phys. Stud.}, 73--109.
Kluwer Acad. Publ., Dordrecht.

\bibitem{LR2} \textsc{Lebeau, G., Robbiano, L.,}  Stabilisation de l'\'equation des ondes par le bord. \textit{Duke Math. J.},  1997
\textbf{86}, no. 3, 465--491.

\bibitem{LM:1968} \textsc{Lions, J.-L. and Magenes, E.} Probl{\`e}mes aux limites non homog{\`e}nes et applications. {V}ol. 1, Dunod, Paris, 1968


\bibitem{Martinez:02} \textsc{Martinez, A .,} 2002. An introduction to semiclassical and microlocal analysis. Universitext. Springer-Verlag, New York.


\bibitem{MS-1978} \textsc{Melrose, R.B., Sj\"ostrand, J.,} Singularities of boundary value problems {I}.
\textit{Comm. Pure Appl. Math.} 1978, \textbf{31},  593--617

\bibitem{MS-1982} \textsc{Melrose, R.B., Sj\"ostrand, J.,} Singularities of boundary value problems {II}.
\textit{Comm. Pure Appl. Math.} 1982, \textbf{35},  129--168


\bibitem{Miller-2000} \textsc{Miller, L.,} Refraction of high-frequency waves density by sharp interfaces 
and semiclassical measures at the boundary. \textit{J. Math. Pures Appl. (9)}, 2000, \textbf{79}, 227--269.


\bibitem{Pruss-1984} \textsc{Pr\"uss, J.,} On the spectrum of {$C_{0}$}-semigroups. \textit{Trans. Amer. Math. Soc.}, 1984 \textbf{284}, 847--857.

\bibitem{Robbiano-2013} \textsc{Robbiano, L.,} Spectral analysis of the interior transmission eigenvalue problem.   
\textit{Inverse Problems},  2013, \textbf{29}, 104001, 28


\bibitem{Rob-Zuily-2009} \textsc{Robbiano, L., Zuily, C.,} The {K}ato smoothing effect 
for {S}chr\"odinger equations with unbounded potentials in exterior domains. 
\textit{Int. Math. Res. Not. IMRN}, 
 2009, \textbf{9}, 1636--1698.

	
\bibitem{Savare-1997} \textsc{Savar{\'e}, G.,} Regularity and perturbation results for mixed second order
              elliptic problems. \textit{Comm. Partial Differential Equations}, 1997, \textbf{22}, 869--899.
              
\bibitem{Sh} \textsc{Shamir, E.,}  Regularization of mixed second order
elliptic problems. \emph{Israel J. Math.}, 1968,   \textbf{6}, 150-168.

		

\bibitem{Tartar-1990} \textsc{Tartar, L.,} {$H$}-measures, a new approach for studying homogenisation,
              oscillations and concentration effects in partial differential
              equations. \textit{Proc. Roy. Soc. Edinburgh Sect. A}, 1990, \textbf{115}, 93--230
              
 \bibitem{Tataru-1998} \textsc{Tataru, D., } On the regularity of boundary traces for the wave equation. 
  \textit{Ann. Scuola Norm. Sup. Pisa Cl. Sci.}, 1998, \textbf{26}, 185--206
              
              
\bibitem{Vodev-2015} \textsc{Vodev, G.,} Transmission eigenvalues for strictly concave domains,   Arxiv1501.00797
             

\bibitem{Zworski-2012} \textsc{Zworski, M.,} Semiclassical analysis. Graduate Studies in Mathematics \textbf{138},
American Mathematical Society, Providence, RI.  2012.

\end{thebibliography}
\end{document}